\documentclass[review,onefignum,onetabnum]{siamart190516}

\usepackage{amsmath,bm}
\usepackage{amsfonts}
\usepackage{graphicx}
\usepackage{psfrag}
\usepackage{tikz}
\usetikzlibrary{calc,positioning}
\definecolor{corange}{rgb}{0.93, 0.57, 0.13}
\usepackage{simplewick}
\usepackage{mathrsfs}
\usepackage{url,hyperref}
\usepackage{setspace}

\usepackage{algorithm}
\usepackage{algorithmic}
\usepackage{amssymb}
\usepackage{booktabs}
\usepackage{comment}
\usepackage{subfigure}
 \usepackage{float}
 \usepackage{cite}
 \usepackage{epstopdf}
\newtheorem{exa}{\bf Example}

\newcommand{\bs}[1]{\boldsymbol{#1}}
\allowdisplaybreaks[4]

\def \bc{\bs c}

\def \bx{\bs x}
\def \by{\bs y}
\def \bz{\bs z}
\def \bxi{\bs \xi} 
\def\B{\mathbb{B}}

\def\Q{\widetilde{Q}}
\def\R{\mathbb{R}}
\def\i{\mathrm{i}}

\def\d{\mathrm{d}}
\def\V{\mathrm{Var}}
\newcommand \dint {\displaystyle\int}

\usepackage{lipsum}
\usepackage{amsfonts}
\usepackage{graphicx}
\usepackage{epstopdf}
\usepackage{algorithmic}
\ifpdf
  \DeclareGraphicsExtensions{.eps,.pdf,.png,.jpg}
\else
  \DeclareGraphicsExtensions{.eps}
\fi

\newsiamremark{remark}{Remark}
\newsiamremark{hypothesis}{Hypothesis}
\crefname{hypothesis}{Hypothesis}{Hypotheses}
\newsiamthm{claim}{Claim}

\headers{Integral Fractional Laplacian in Multiple Dimensions }{C. Sheng, B. Su, and  C. Xu}

\title{Efficient Monte Carlo Method for Integral Fractional Laplacian in Multiple Dimensions \thanks{Submitted to the editors DATE.
\funding{The research of the first author is partially supported by Shanghai Pujiang Program 21PJ1403500, the Fundamental Research Funds for the Central Universities 2021110474 and Shanghai Post-doctoral Excellence Program 2021154. The research of the third author is partially supported by he Fundamental Research Funds for the Central Universities.}}}

\author{ Changtao Sheng \thanks{School of Mathematics, Shanghai University of Finance and Economics, Shanghai 200433, China. Email: \email{ctsheng@sufe.edu.cn} (C. Sheng); \email{subihao@163.sufe.edu.cn} (B. Su); \email{xu.chenglong@shufe.edu.cn} (C. Xu).}
\and Bihao Su $^\dagger$\and Chenglong Xu $^\dagger$}
\usepackage{amsopn}

\ifpdf
\hypersetup{
  pdftitle={Efficient Monte Carlo Method for Integral Fractional Laplacian in Multiple Dimensions},
  pdfauthor={C. Sheng, B. Su, and  C. Xu}
}
\fi

\begin{document}
\nolinenumbers

\maketitle

% REQUIRED
\begin{abstract}
In this paper, we develop a Monte Carlo method for solving PDEs involving an integral fractional Laplacian (IFL) in multiple dimensions. We first construct a new Feynman-Kac representation based on the Green function for the fractional Laplacian operator on the unit ball in arbitrary dimensions. Inspired by the ``walk-on-spheres" algorithm proposed in \cite{MR1567642}, we extend our algorithm for solving fractional PDEs in the complex domain. Then, we can compute the expectation of a multi-dimensional random variable with a known density function to obtain the numerical solution efficiently.  The proposed algorithm finds it remarkably efficient in solving fractional PDEs: it only needs to evaluate the integrals of expectation form over a series of inside ball tangent boundaries with the known Green function. Moreover, we carry out the error estimates of the proposed method for the $n$-dimensional unit ball.
Finally, ample numerical results are presented to demonstrate the robustness and effectiveness of this approach for fractional PDEs in unit disk and complex domains, and even in ten-dimensional unit balls. 
\end{abstract}

% REQUIRED
\begin{keywords}
Integral fractional Laplacian, Green function, Monte Carlo method,  Spherical coordinate
\end{keywords}

% REQUIRED
\begin{AMS}
37P30, 60G22,  65C05, 65L70, 91G60
\end{AMS}

\section{Introduction}
Fractional powers of the Laplacian operator are a powerful tool in modeling phenomena in anomalous diffusion, which arise naturally in the $\alpha$-stable L$\acute{\text{e}}$vy process instead of the Brownian motion for normal diffusion (see, e.g., \cite{shlesinger1987levy,metzler2000random,metzler2004restaurant,brockmann2006scaling} and the references therein). 
It is known that for $\alpha\in (0,2),$ the fractional Laplacian of a sufficiently nice function $u(\bx)\!:  \mathbb R^n\to \mathbb R$ has the hypersingular integral representation (cf.\! \cite{Nezza2012BSM}):   
\begin{equation}\label{fracLap-defn}
(-\Delta)^{\frac{\alpha}2} u(\bx)=C_{n,\alpha}\, {\rm p.v.}\! \int_{\mathbb R^n} \frac{u(\bx)-u(\by)}{|\bx-\by|^{n+\alpha}}\, {\rm d}\by,\quad
C_{n,\alpha} :=\frac{2^{\alpha}\Gamma(\frac{n+\alpha}{2})}{\pi^{\frac{n}2}\Gamma(1-\frac{\alpha}2)}, 
\end{equation}
where ``p.v." stands for the principle value and $C_{n,\alpha} $ is  the normalisation constant.  This point-wise expression
for the fractional Laplacian can directly be derived from the following Fourier transform:
\begin{equation}\label{Ftransform}
\begin{split}
 (-\Delta)^{\frac{\alpha}2} u(\bx)={\mathscr F}^{-1}\big[|\bxi|^{\alpha}  {\mathscr F}[u]( \bxi)\big](\bx), \;\;\bx \in {\mathbb R}^n.
\end{split}
\end{equation}
The nonlocal and singular nature of this operator poses  major difficulties in discretisation and analysis. In the case
where $\alpha=2$ the operator coincides with the usual negative Laplacian $-\Delta$.
Most recent concerns are with PDEs involving the IFL operator on an open bounded Lipschitz domain $\Omega\subset \mathbb R^n.$ 
More precisely, given $f: \Omega\to \mathbb R$ in a suitable space, we look for  $u$ on $\Omega$ satisfying the fractional Poisson equation with the (nonlocal) Dirichlet boundary condition:
\begin{equation}\label{ufg}
\begin{cases}
 (-\Delta)^{\frac{\alpha}2}u(\bx)=f(\bx),\;&{\rm in}\ \Omega,\\[4pt]
u(\bx)=g(\bx),\quad &{\rm on}\ \Omega^c:=\mathbb{R}^n\backslash  \Omega.
\end{cases}
\end{equation}

\subsection{Review on existing studies}
The numerical method for fractional Poisson equation \eqref{ufg} is challenging due to the nonlocality of the fractional Laplacian operator, hyper-singular integrals, and the low regularity of the solution at the boundary. In general, the existing approaches can be classified into the following two categories. 

The first is to discretize the fractional Poisson equations using demerministic numerical methods.
It is customary to apply finite difference methods for integral fractional Laplacian (see, e.g., \cite{duo2015computing,duo2018finite,victor2020simple,hao2019fractional} and many references therein). However, those finite difference discretizations are only available for simple domains such as rectangular domains and require stronger regularity assumptions for the solution. The finite element methods are more preferable for solving fractional PDEs, as it has a well-established theoretical foundation and the capability to deal with complex domain. 
The IFL operator has interwoven connections with the fractional Sobolev framework (cf.\! \cite{Nezza2012BSM}), so the finite element scheme can be established under more realistic assumptions, i.e., in fractional Sobolev space and without additional regularity requirements. There have been many recent works devoted to the FEM analysis for fractional PDEs (see, e.g., \cite{acosta2017fractional,ainsworth2017aspects,borthagaray2018finite,bonito2019numerical,borthagaray2020local,faustmann2020local} and the references therein), but the literature on FEM implementation in multiple dimensions is very limited. % where it is accomplished based on either the hypersingular integral representation \eqref{fracLap-defn} \cite{acosta2017short, ainsworth2017aspects} or the alternative Dunford-Taylor formulation of the IFL operator \cite{bonito2019numerical}. 
Compared with the local schemes, the special designed spectral methods could compensate for the weakly singular behavior of the solution and hence are expected to approximate the solution of \eqref{ufg} accurately. However, the spectral methods are limited to one dimension, simple domains, and the whole space case, we refer to \cite{mao2017hermite,tang2018hermite,sheng2019fast,chen2018jacobi,MR4049399,MR4280294} and the references therein for more details. Needless to say, for IFL on the bounded domain, these methods become complicated even for $n=2,3$ and computationally prohibitive for the higher-dimensional cases.

The second is to use  Monte Carlo methods based on the stochastic processes to alleviate its proverbial numerical difficulties. It is well-known that the Feynman-Kac formula is an important bridge to link PDEs and stochastic processes (see, e.g., \cite{MR1329992, MR1621249, MR0494491, MR1567642,MR880021} and the references therein), which offers a method of solving certain PDEs via simulating random paths of a stochastic process. For the classical Poisson equations, that is, 
\begin{equation}\label{classu}
-\Delta u(\bs x)=f(\bs x)\quad   {\rm in}\;\;  \Omega; \quad u(\bs x)=g(\bx)\quad  {\rm on}\;\; \partial\Omega,
\end{equation}
which is governed by Brownian motion and the solution \eqref{classu} has a Feynman-Kac representation, expressed as an expectation at first exit from $\Omega$ of the associated Wiener process (cf. \cite{MR1070473}). 
Remarkably, in \cite{MR1070473,MR1567642}, the Feynman-Kac representation was extended to fractional Laplacian, where Brownian motion was replaced by $\alpha$-stable L$\acute{\text{e}}$vy process.  
Let $\Omega\subset \mathbb R^n$ be a bounded domain, with $n\ge 1$. For any $\alpha\in(0,2)$ and a Borel set $\Lambda\subset \mathbb{R}^n$,  we denote 
$$L^{1}_{\alpha}(\Lambda)=\Big\{u\in L^{1}(\Lambda) \;\;{\rm s.t.}\;\int_{\Lambda}\frac{|u(x)|}{1+ |x|^{n+\alpha}}\,{\rm d}x <\infty\Big\}.$$
 Suppose that $g(x)$ is a continuous function that belongs to $L^{1}_\alpha(\Omega^c)$, and $f(x)$ is a function in $C^{\alpha+\epsilon}(\Omega)$ for some $\epsilon>0$. Then there exists a unique continuous solution to problem \eqref{ufg} in $L^{1}_\alpha(\mathbb{R}^{n})$, which is given by (see \cite[Thm.\,6.1]{MR1567642})
\begin{equation}\label{uold}
u(x) = \mathbb{E}_{X_{0}^{\alpha }=x}\big[g(X^{\alpha}_{\tau_{\Omega}})\big] + \mathbb{E}_{X_{0}^{\alpha }=x}\Big[\int_{0}^{\tau_{\Omega}}f(X^{\alpha}_{s})\,{\rm d}s\Big],\;\; x\in\Omega,
\end{equation}
where $\tau_{\Omega} = \inf \{  t>0:X_{t}^{\alpha}\notin \Omega \}$, and $\{X_{t}^{\alpha}\}_{t\geq0}$ is a symmetric $\alpha$-stable L$\acute{\text{e}}$vy process with $X_{0}^{\alpha}=x$. 
We recall that a stochastic process $\{ X_{t}^{\alpha} ;t\geq 0\}$ is called symmetric $\alpha$-stable L$\acute{\text{e}}$vy process if:
%\begin{description}
%\item[(1).] $X^{\alpha}_{0} = 0$ a.s.;
%\item[(2).] $X^{\alpha}_{t}$ has independent and stationary increments;
%\item[(3).] $X^{\alpha}_{t}-X^{\alpha}_{s} \backsim S_{\alpha}((t-s)^{1/\alpha},0,0)$ for any $ 0\leq s<t<\infty$.
%\end{description}

(1). $X^{\alpha}_{0} = 0$ a.s.;\\[-8pt]

(2). $X^{\alpha}_{t}$ has independent and stationary increments;\\[-8pt]

(3). $X^{\alpha}_{t}-X^{\alpha}_{s} \backsim S_{\alpha}((t-s)^{1/\alpha},0,0)$ for any $ 0\leq s<t<\infty$, that is, $\alpha$-stable distribution with \\[-10pt]

\hspace{20pt}scale parameter $(t-s)^{1/\alpha}$, and skewness and shift parameters equal to zero.\\[-8pt]

\noindent In particular, when $\alpha = 2$, then $X_{t}^{2} = \sqrt{2}B_{t}$, $B_{t}$ represents the standard Brownian motion. The symmetric $\alpha$-stable L$\acute{\text{e}}$vy process is $1/\alpha$ self-similar. That is for $k>0$, the process $\{  X^{\alpha}_{kt};t\geq0  \}$ and $\{  k^{1/\alpha}X^{\alpha}_{t};t\geq0 \} $ have the same finite-dimensional distributions. The characteristic function (see \cite{ Janicki1994can}) of a $\alpha$-stable random variable $X$ is given by
\begin{eqnarray}
\ln\psi(\theta) =
\begin{cases}
-\sigma^{\alpha}|\theta|^{\alpha}\big[  1-\i\beta \text{sgn}(\theta)\tan(\frac{\alpha\pi}{2}) + \i\mu\theta \big],&{\rm if} ~\alpha\in(0,1)\cup(1,2], \\[4pt]
-\sigma|\theta| \big[  1+\i\beta\frac{2}{\pi} \text{sgn}(\theta)\ln |\theta|+ \i\mu\theta\big],&{\rm if} ~\alpha=1,
\end{cases}
\end{eqnarray}
where $\alpha\in(0,2]$ is the index of stability, $\beta\in [-1,1]$ denotes the skewness parameter, $\mu\in \mathbb{R}$ and $\sigma>0$ represent the shift and scale parameters respectively.
Then $X$ is called $\alpha$-random variable, and we use the notation $X\backsim S_{\alpha}(\sigma,\beta,\mu)$. Obviously, $S_{1}(\sigma,0,\mu)$ and $S_{2}(\sigma,0,\mu)$ are the Cauchy distribution and the Gaussian distribution $\mathcal{N}(\mu,\sigma)$, respectively (cf. \cite{ Janicki1994can}).

This groundbreaking representation provides a viable alternative for its mathematical and numerical treatment. Kyprianou et al. \cite{MR1567642} developed an efficient Monte Carlo method for \eqref{ufg} in two dimensions by using the walk-on-spheres algorithm, which is based on the Feynman-Kac representation. With this idea, Shardlow \cite{MR3985475} proposed the multilevel Monte Carlo method combined with the Walk outside spheres algorithm for solving the eigenvalues problem involving fractional Laplacian operator. Those methods can be extended to high-dimensional cases. However, it will become complicated, as the Feynman-Kac representation is associated with a high dimensional $\alpha$-stable process, which seems not an easy task. Thus, extending those approaches to a higher dimensional case is nontrivial.

%We also remark that direct learning of the IFL and/or nonlocal models on bounded domains based on the definition \eqref{fracLap-defn},  was discussed in some recent works (see, e.g., \cite{doi:10.1137/18M1229845,chen2021data} for Physical-informed neural networks (PINN); and   \cite{PANG2020109760,doi:10.1137/18M1204991,MR4022347,you2022nonlocal} for deep learning methods). Combining the traditional automatic differential method with numerical discretization, a fractional PINN can be constructed to solve the space-time fractional advection diffusion equations. In addition, deep learning methods are also used to  to study non-local operator, such as a nonlocal Laplace operator that is parameterized by the nonlocal interaction radius and the decay rate, and so on. Because these methods are data-driven and do not rely on fixed grids, they have higher flexibility in dealing with high-dimensional problems in complex regions.

We also remark that direct learning of the IFL and/or nonlocal models on bounded domains was discussed in some recent works (see, e.g., \cite{Pang2019fpinns,chen2021data} for Physical-informed neural networks (PINN); and \cite{PANG2020109760,Gulian2019M,MR4022347,you2022nonlocal,Guo2022M} for deep learning methods). We note that most current works tend to combine the traditional automatic differential method with some typical numerical discretization to solve the fractional or nonlocal models.

\subsection{Our contributions and the organization of the paper}
This paper proposes and analyzes a Monte Carlo method for PDEs involving integral fractional Laplacian on the bounded domain in arbitrary dimensions. 
The proposed method takes advantage of the Feynman-Kac representation based on Green functions for the simulation of both classical and fractional Laplacians in a unified framework, which avoids the enormous computational cost of evaluating the fractional derivatives. 
We highlight below the main advantages of the proposed methods and main contributions
of the paper:
\begin{itemize}
\item We provide a new Feynman-Kac representation for the solution of fractional Poisson equations on the ball in arbitrary dimensions (see Lemma \ref{newres} below). Benefiting from the new representation, to avoid calculating the expectation of a complex integral that involves the function of $\alpha$-stable L\'{e}vy process, we can evaluate the expectation of a random variable associated with the explicit expression of the Green functions to obtain the solution. 
%We provide a new Feynman-Kac representation for the solution of fractional Laplacian equations on unit ball in arbitrary dimensions  (see Lemma \ref{newres} below). Benefiting from the new representation, we can avoid calculating the expectation of a complex integral that involves a function of a stochastic process, but replace it with the expectation of a random variable with a known density function, which can be constructed by Green’s function.

\item  We further extended the representation for fractional PDEs on the irregular domain. More precisely, as shown in Theorem \ref{irre} of this paper, the solution of \eqref{ufg} on an open bounded domain can be evaluated by the expectation of a sequence of inside balls tangent to the boundary, as the Green's function and Poisson's kernel on the balls are known. The existing work \cite{MR1567642} is to adopt the walk-on-spheres algorithm for solving fractional Poisson equation in two dimensions, in which the walk-on-spheres precess is associated with the $\alpha$-stable process. 

\item The proposed method is capable of fractional PDEs in high dimensions. Moreover, in the Monte Carlo procedure, thanks to spherical coordinates, the distribution in the angular direction for each jump is uniform, allowing us only need to appropriately increase the samples in the angular direction with the increase of dimensionality so that the computation time will increase at a moderate pace. 

\item We fully analyze the error analysis and characterize the efficiency of the proposed method for fractional Poisson equations on the ball in arbitrary dimensions, with the aid of the property of the Monte Carlo procedure and the explicit expressions of the Green functions.
\end{itemize}

The rest of the paper is organized as follows. In the next section, we derive an alternative formula based on the Green function of integral fractional Laplacian operator on the unit ball in arbitrary dimensions and extend the result to the case of irregular domain. Section 3 describes the detailed algorithm for computing the fractional PDEs based on the Monte Carlo method. In Section 4, we provide a complete error analysis of the proposed methods for the fractional Poisson equation. Numerical experiments for PDEs with integral fractional Laplacian are carried out in Section 5. Some concluding remarks are given in the last section.

%\lipsum[2-3]
%
%% The outline is not required, but we show an example here.
%The paper is organized as follows. Our main results are in
%\cref{sec:main}, our new algorithm is in \cref{sec:alg}, experimental
%results are in \cref{sec:experiments}, and the conclusions follow in
%\cref{sec:conclusions}.
%
\section{Main results} \setcounter{lem}{0} \setcounter{thm}{0}  \setcounter{rem}{0} 
\subsection{Feynman-Kac representation in expectation form in a ball}
In this section, we introduce some representation formula for the solution of fractional Laplacian operator on the ball in arbitrary dimensions. 
To this ends, we set $\bc_{0}\in \mathbb{R}^{n}$ and define $r>0$ to be the radius of a ball inscribed in $\mathbb{R}^{n}$ that is centered at $\bc_{0}$. This sphere we will call $\B^{n,\bc_{0}}_{r} = \{\bx\in \mathbb{R}^{n}: |\bx-\bc_{0}| \leq r \}$. Let $\B_{r}^n=\B_{r}^{n,\bm{0}}$ for simplicity.

We first show that the solution of \eqref{ufg} on the ball can be explicitly represented as the following integral form as follows, which play important role in
the algorithm development. 
%The following alternative formulation with explicit expression for the solution of \eqref{ufg} is more convenient for computation.
\begin{lemma}\label{mainthm}
Let $r>0$ and $\alpha\in(0,2]$. If $f\in L^{1}_{\alpha}(\B_{r}^{n})\cap C(\overline{\B}_{r}^{n})$, $g\in L^{1}_{\alpha}(\R^{n}\!\setminus\!\B_{r}^{n}) $, 
then the solution of \eqref{ufg} with $\Omega=\mathbb{B}^n_r$ can be represented as
\begin{equation}
\label{newrepre}
u(\bx) =\begin{cases}
\dint_{\B^{n}_{r}}f(\by)Q_{r}(\bx,\by)\,\d\by +\dint_{\mathbb{R}^{n}\setminus \B^{n}_{r}}g(\bz)P_{r}(\bx,\bz){\rm d}\bz,& {\rm in}\,\B^{n}_{r}, \\[6pt]
 g(\bx),& {\rm on}\,\mathbb{R}^{n}\!\setminus\!\B^{n}_{r},
\end{cases}
\end{equation}
where $P_{r}(\bx,\bz)$ is the Poisson kernel defined by 
\begin{equation}
\label{Pr}
P_{r}(\bx,\bz) = \widetilde{C}^\alpha_n\Big(\frac{r^{2} - |\bx|^{2}}{|\bz|^{2} -r^{2}} \Big)^{\alpha/2} \frac{1}{|\bx-\bz|^{n}},\;\;\bx\in \B^{n}_{r},\;\;\bz\in \mathbb{R}^{n}\!\setminus \! \overline{\B}^{n}_{r},
\end{equation}
and for $\bx\neq \by$,
\begin{equation}
\label{Qr}
Q_{r}(\bx,\by) = \begin{cases} \widehat{C}^\alpha_n |\by-\bx|^{\alpha-n}\dint_{0}^{\varrho(\bx,\by)}\frac{t^{\frac{\alpha}2-1}}{(t+1)^{\frac{n}{2}}}\,\d t,\;\;&\alpha\neq n,\\[10pt]
\widehat{C}^{\frac12}_1\log \Big{(}   \frac{ r^{2}-\bx \by+\sqrt{ (r^{2} -\bx^{2})(r^{2}-\by^{2}) }}{r|\by-\bx|}  \Big{)},\;\;&\alpha=n,
\end{cases}
\end{equation}
with 
\begin{equation}\label{funr}
\varrho(\bx,\by) = \frac{(r^{2} - |\bx|^{2})(r^{2}-|\by|^{2})}{r^{2}|\bx-\by|^{2}}.
\end{equation}
In the above, the constants
\begin{equation}\label{constants1}
\widetilde{C}^\alpha_n= \frac{\Gamma(n/2)\sin(\pi \alpha/2)}{\pi^{\frac{n}{2}+1}},\;\;\;\widehat{C}^\alpha_n=\frac{\Gamma(n/2) }{2^{\alpha}\pi^{\frac{n}{2}}\Gamma^{2}(\alpha/2)}.
\end{equation}
\end{lemma}
%\begin{lemma}\label{mainthm}
%Let $r>0$, and assume that $f\in L^{1}_{\alpha}(\B_{r}^{n})\cap C(\overline{\B}_{r}^{n})$, $g\in L^{1}_{\alpha}(\R^{n}\!\setminus\!\B_{r}^{n}) $, 
%then the solution of \eqref{ufg} can be represented as
%\begin{equation}
%\label{newrepre}
%u(\bx) =\begin{cases}
%\dint_{\B^{n}_{r}}f(\by)Q_{r}(\bx,\by)\,\d\by +\dint_{\mathbb{R}^{n}\setminus \B^{n}_{r}}g(\bz)P_{r}(\bx,\bz){\rm d}\bz,& {\rm in}\,\B^{n}_{r}, \\[6pt]
% g(\bx),& {\rm on}\,\mathbb{R}^{n}\!\setminus\!\B^{n}_{r},
%\end{cases}
%\end{equation}
%where $P_{r}(\bx,\bz)$ is the Poisson kernel defined by 
%\begin{equation}
%\label{Pr}
%P_{r}(\bx,\bz) = a(n,\alpha) \Big(\frac{r^{2} - |\bx|^{2}}{|\bz|^{2} -r^{2}} \Big)^{\alpha/2} \frac{1}{|\bx-\bz|^{n}},\;\;\bx\in \B^{n}_{r},\;\;\bz\in \mathbb{R}^{n}\!\setminus \! \overline{\B}^{n}_{r},
%\end{equation}
%with constant $a(n,\alpha) = \frac{\Gamma(n/2)\sin(\pi \alpha/2)}{\pi^{\frac{n}{2}+1}}.$
%If $n\neq \alpha$, $Q_{r}(\bx,\by)$ is given by 
%\begin{eqnarray}
%\label{Qr}
%Q_{r}(\bx,\by) =  b(n,\alpha) |\by-\bx|^{\alpha-n}\int_{0}^{r(\bx,\by)}\frac{t^{\frac{\alpha}2-1}}{(t+1)^{\frac{n}{2}}}\,\d t,\;\;\bx\neq \by,
%\end{eqnarray}
%while for $n=\alpha$, there holds
%\begin{eqnarray}
%Q_{r}(\bx,\by) = b(1,\frac{1}{2})\log \Big{(}   \frac{ r^{2}-\bx \by+\sqrt{ (r^{2} -\bx^{2})(r^{2}-\by^{2}) }}{r|\by-\bx|}  \Big{)},
%\end{eqnarray}
%where
%\begin{eqnarray}
%b(n,\alpha) = \frac{\Gamma(\frac{n}{2}) }{2^{\alpha}\pi^{\frac{n}{2}}\Gamma^{2}(\frac{\alpha}{2})},~~r(\bx,\by) = \frac{(r^{2} - |\bx|^{2})(r^{2}-|\by|^{2})}{r^{2}|\bx-\by|^{2}}.
%\end{eqnarray}
%
%\end{lemma}

\begin{remark} {\em 
%It is worth noting that when $\alpha=2$, the stochanstic process is the standard Brownian motion, then the corresponding Green's function and Poisson's kernel are given by
Note that if $\alpha=2$, the stochastic process is the standard Brownian motion, and the corresponding Green's function and Poisson's kernel are given by
 \begin{equation*}
Q_{r}(\bx,\by)=-\frac{1}{2\pi}\log\frac{1}{|\by-\bx|},\;\;\; P_{r}(\bx,\bz)=\frac{1}{2\pi r}\frac{r^{2}-|\bx|^{2}}{|\bx-\bz|^{2}},\;\;\;{\rm if}\,n=2.
\end{equation*}
%which is the Green's function of Laplacian operator in $\mathbb{R}^{2}$, the corresponding Poisson kernel is 
 If $\alpha=2$ and $n=3$, there holds
 \begin{equation*}\begin{split}
&Q_{r}(\bx,\by)=-\frac{1}{4\pi}\bigg[\frac{1}{|\by-\bx|} -\frac{1}{\sqrt{|\by|^{2} +|\bx|^{2}-2\bx\by\cos\theta}-\sqrt{r^{2}+\frac{|\bx|^{2}|\by|^{2}}{r^{2}}-2\bx\by\cos\theta}} \bigg],\\&
P_r(\bx,\bz)=\frac{1}{4\pi r}\frac{r^{2}-|\bx|^{2}}{|\bx-\bz|^{3}}.
\end{split}\end{equation*}
where $\theta$ represents the angle between the vectors $\bx$ and $\by$.
}
%and the corresponding Poisson kernel is given by
%\begin{equation*}
%P_r=\frac{1}{4\pi r}\frac{r^{2}-|x|^{2}}{(r^{2}+|x|^{2}-2rx\cos\theta)^{3/2}}.
%\end{equation*}}
\end{remark}

The following alternative formulation with explicit expression for the solution of \eqref{ufg} is more convenient for computation, where we transform the expression of the original solution \eqref{newrepre} into an expectation form.
\begin{lemma}\label{newres}
Let $\Omega=\mathbb{B}^n_r$ with $r>0$, and assume that $f\in L^{1}_{\alpha}(\B_{r}^{n})\cap C(\overline{\B}_{r}^{n})$ and $g\in L^{1}_{\alpha}(\R^{n}\!\setminus\!\B_{r}^{n}) $, then the solution of problem \eqref{ufg} in $L^{1}_\alpha(\mathbb{R}^{n})$ can also be expressed as 
\begin{equation}\label{resu}
\begin{split}
&u(\bx)=\zeta(\bx)\mathbb{E}_{\widetilde{Q}_{r}}[f(Y)]+\mathbb{E}_{P_{r}}[g(Z)],\;\;\; Y\in \B_{r}^n,\; Z\in \R^{n}\!\setminus\!\B^{n}_{r},
\end{split}
\end{equation}
where $\zeta(\bx)$ is the weight function of the form
\begin{equation}\label{funzeta}
\zeta(\bx) = \int_{\B_r^n} Q_{r}(\bx,\by)\,\d \by.
\end{equation}
% In the above,  the first expectation represents the expected contribution from source $f$ inside the sphere and the second one describes a mean value with respect to the function $P_{r}$ outside the sphere.
  In the above,  the first term represents a weighted average of the source $f$ inside the ball, and the second term denotes a mean value with respect to the function $P_{r}$ outside the ball.
\end{lemma}
\begin{proof}
To obtain probability density function for $f(x)$, we normalize the Green function as
\begin{equation}
\label{rho}
\widetilde Q_{r}(\bx,\by)= \frac{Q_{r}(\bx,\by)}{\zeta(\bx)},
\end{equation}
where $\zeta(x)$ is defined in \eqref{funzeta}. According to \cite[Lemma A.5]{MR3461641}, we have 
\begin{equation}\label{normP}
\dint_{\mathbb{R}^{n}\setminus \B^n_{r}} P_{r}(\bx,\bz)\,\d \bz=1,
\end{equation}
which implies $P_r(\bx,\bz)$ is probability density function for $g(\bz)$. Thus,  we have
\begin{equation}\begin{split}
u(\bx) =& \int_{\B^{n}_{r}}f(\by)Q_{r}(\bx,\by)\,\d \by +\int_{\mathbb{R}^{n}\setminus \B^{n}_{r}}g(\bz)P_{r}(\bx,\bz)\,\d \bz \\ 
=&\zeta(\bx)\dint_{\B^{n}_{r}}f(\by)\Q_{r}(\bx,\by)\,\d \by +\int_{\mathbb{R}^{n}\setminus \B^{n}_{r}}g(\bz)P_{r}(\bx,\bz)\,\d \bz  
=\zeta(\bx)\,\mathbb{E}_{\bx}[f(Y)]+\mathbb{E}_{\bx}[g(Z)].
\end{split}\end{equation}
This completes the proof.
\end{proof}
\begin{remark}{\em It should be pointed out that both \eqref{uold} and \eqref{resu} can be used to construct the Monte Carlo method for solving \eqref{ufg} on $\mathbb{B}^n_r$. The main difference is the corresponding stochastic process of \eqref{uold} is associated with $\alpha$-stable L$\acute{\text{e}}$vy process. While \eqref{resu} can simplify the evaluation of each jump significantly with the aid of the explicit expression of Green function $\Q_r$ and Poisson kernel $P_r$ {\rm(}see \eqref{Pr} and \eqref{Qr}{\rm)}.}
%It should be pointed out that \eqref{resu} is different from the existing one \eqref{uold}, as it is a direct consequence of \eqref{newrepre} and the fact that $\Q_r$ and $P_r$ are the probability density functions. While for the representation \eqref{uold}, it is obtained from the paths of $\alpha$-stable L$\acute{\text{e}}$vy process. 
%We note that the Monte Carlo method can be constructed directly based on \eqref{resu}.}
\end{remark}
%\begin{remark}{\em 
%Although Green's function in the complex domain is not easy to find, the movement path in the complex domain can still be simulated according to the random way \eqref{uold}, and the path is composed of a series of balls, $P_{r}$ and $Q_{r}$ in the balls are known, so the approximate value of the function in this type of region can be obtained.}
%\end{remark}

%We also remark that, with Lemma \ref{newres}, the Monte Carlo method can be constructed directly based on \eqref{resu}. Moreover, thanks to the mutual independence of Monte Carlo random paths, the solution of the fractional Poisson equation \eqref{ufg} can be computed via the mean value of the obtained results according to the random paths. %We will show the detail in the  section. 

\subsection{A new representation in expectation form for the irregular domain}
This section will focus on the expressions for the solutions of PDEs involving IFL \eqref{ufg} on the irregular domain in arbitrary dimensions. Although it is difficult to find the Green function $\widetilde{Q}_r$ and Poisson kernel $P_r$ on the irregular domain $\Omega$, the fractional Poisson equation on the irregular domain can be approximated by the whole motion path with a sequence of small ball.  According to \eqref{uold}, we know that the solution of \eqref{ufg} on any irregular domain can be simulated via an $\alpha$-stable process (cf. Fig.\,\ref{astab}), and this process stops when it reaches the outside of the region. Recall that the solution of \eqref {ufg} corresponds to $\alpha$-stable process $X^{\alpha}_{t} $ and can be rewritten in the form of path integral as (cf. \cite{MR1567642})
%
%Recall that the solution to \eqref{ufg} at the point $\bx$, given in the form of the path-integral with respect to an $\alpha$-stable process $X^{\alpha}_{t}$  as follows:
%It is obviously extremely difficult to directly solve problems in complex regions, as a consequence, we introduce some expressions for the solutions of fractional Laplacian operators over a general domain of arbitrary dimensions in this section. According to \eqref{uold}, we know that for any shape of the region, the solution of the fractional Poisson equation can be simulated by an $\alpha$-stable process, the process stops when it reaches the outside of the region. 
%The solution to problem \eqref{ufg} at the point $\bx$, given in the form of the path-integral with respect to an $\alpha$-stable process $X^{\alpha}_{t}$, is as follows:
\begin{equation*} 
u(\bx) = \mathbb{E}_{X_{0}^{\alpha }=\bx}\Big[\int_{0}^{\tau_{\Omega}}f(X^{\alpha}_{s})\,{\rm d}s\Big] +\mathbb{E}_{X_{0}^{\alpha }=\bx}\big[g(X^{\alpha}_{\tau_{\Omega}})\big] ,\;\; \bx\in\Omega,
\end{equation*}
where $\tau_{\Omega}$ is the first passage time and $X^{\alpha}_{\tau_{\Omega}}$ is the first passage location on the region $\R^{n}\setminus\Omega$, and $\mathbb{E}_{\bx}[\tau_{\Omega}]<\infty$, for all $\bx\in\Omega$. %According to the properties of the $\alpha$-stable process, we can simulate its trajectory.

\begin{figure}[!h]
\subfigure[$\alpha = 0.4$]{
\includegraphics[width=3.75cm,height=3.5cm]{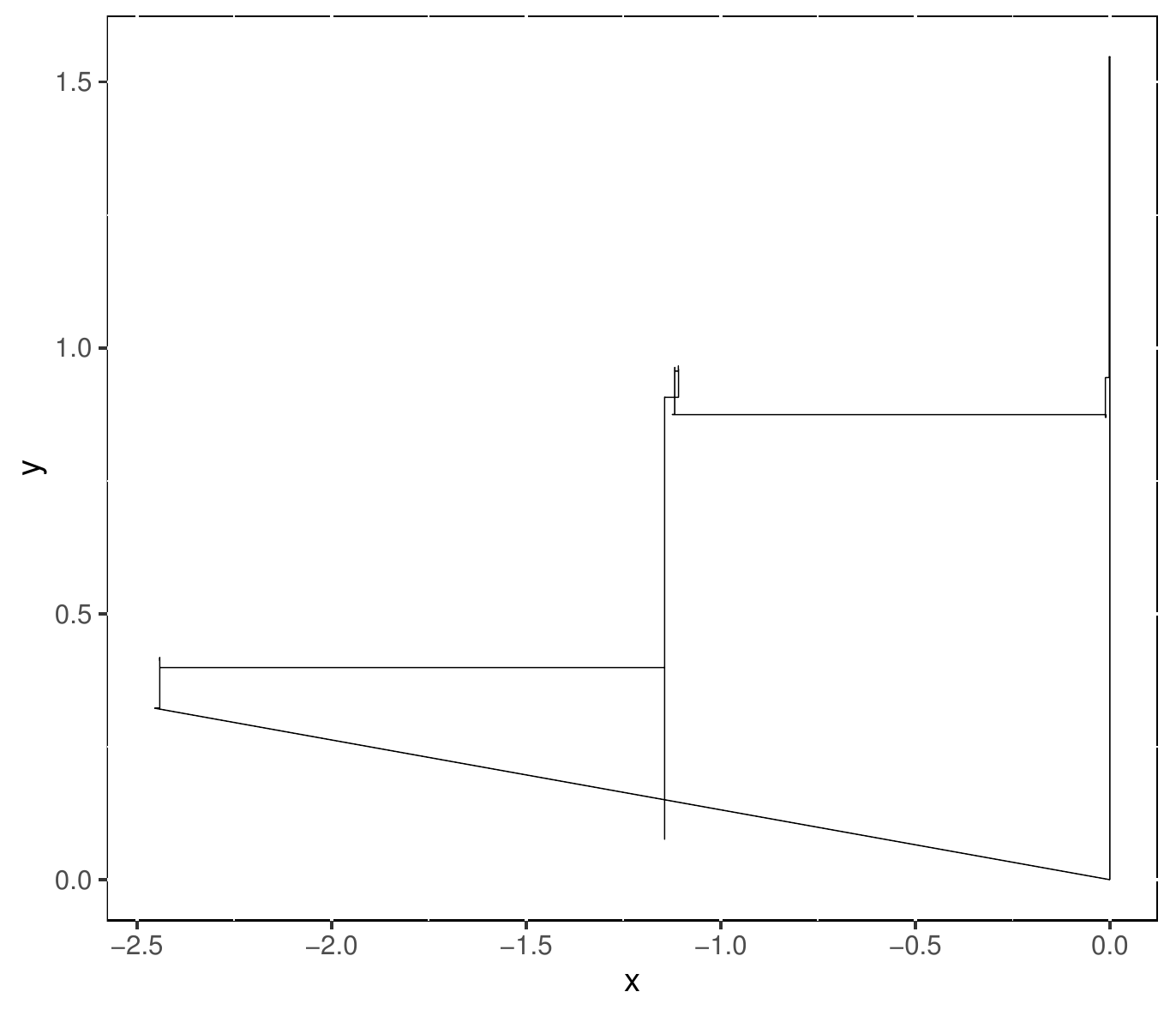}}\hspace{-0.1in}
\subfigure[$\alpha = 0.8$]{
\includegraphics[width=3.75cm,height=3.5cm]{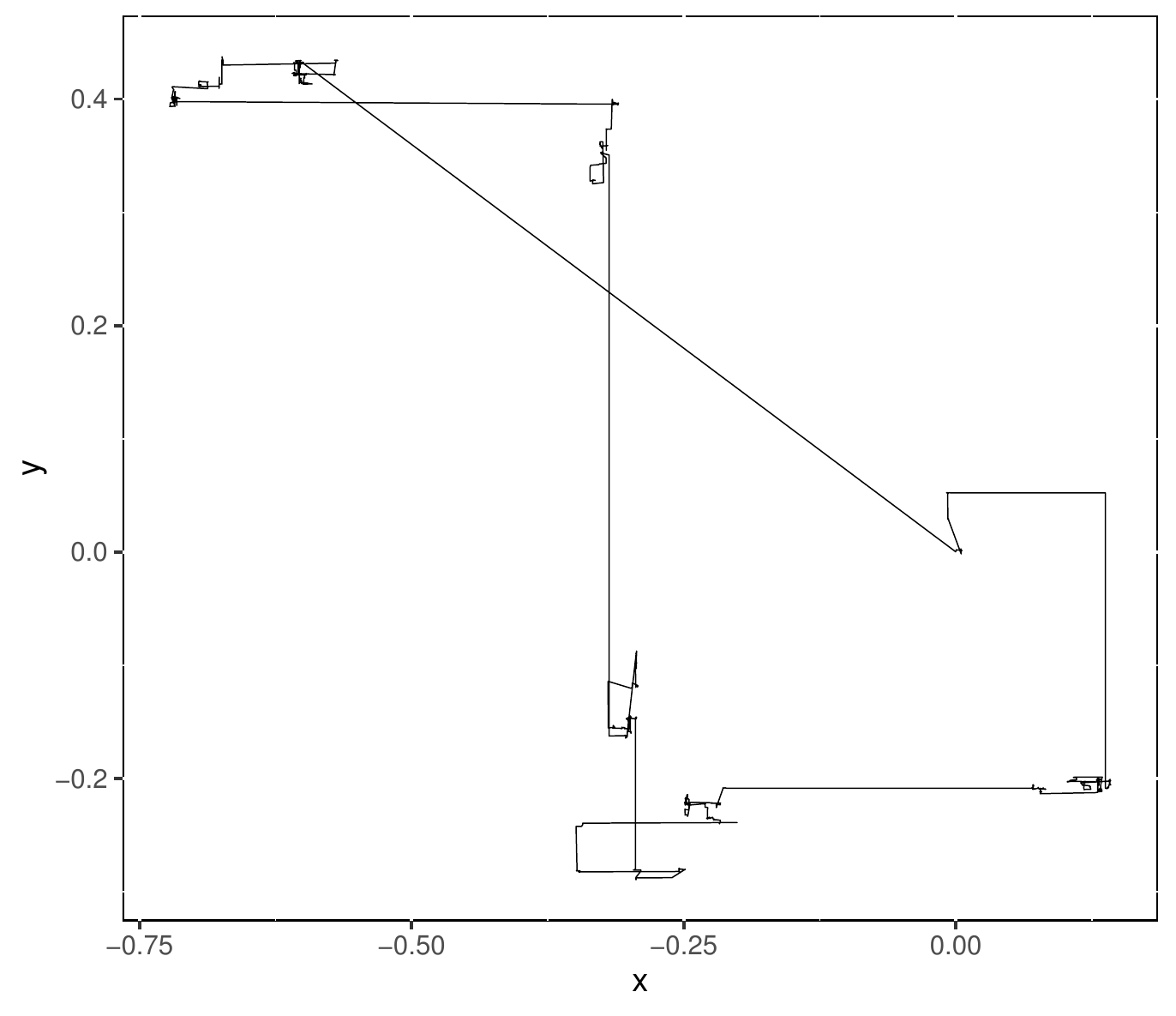}}\hspace{-0.1in}
\subfigure[$\alpha = 1.2$]{
\includegraphics[width=3.75cm,height=3.5cm]{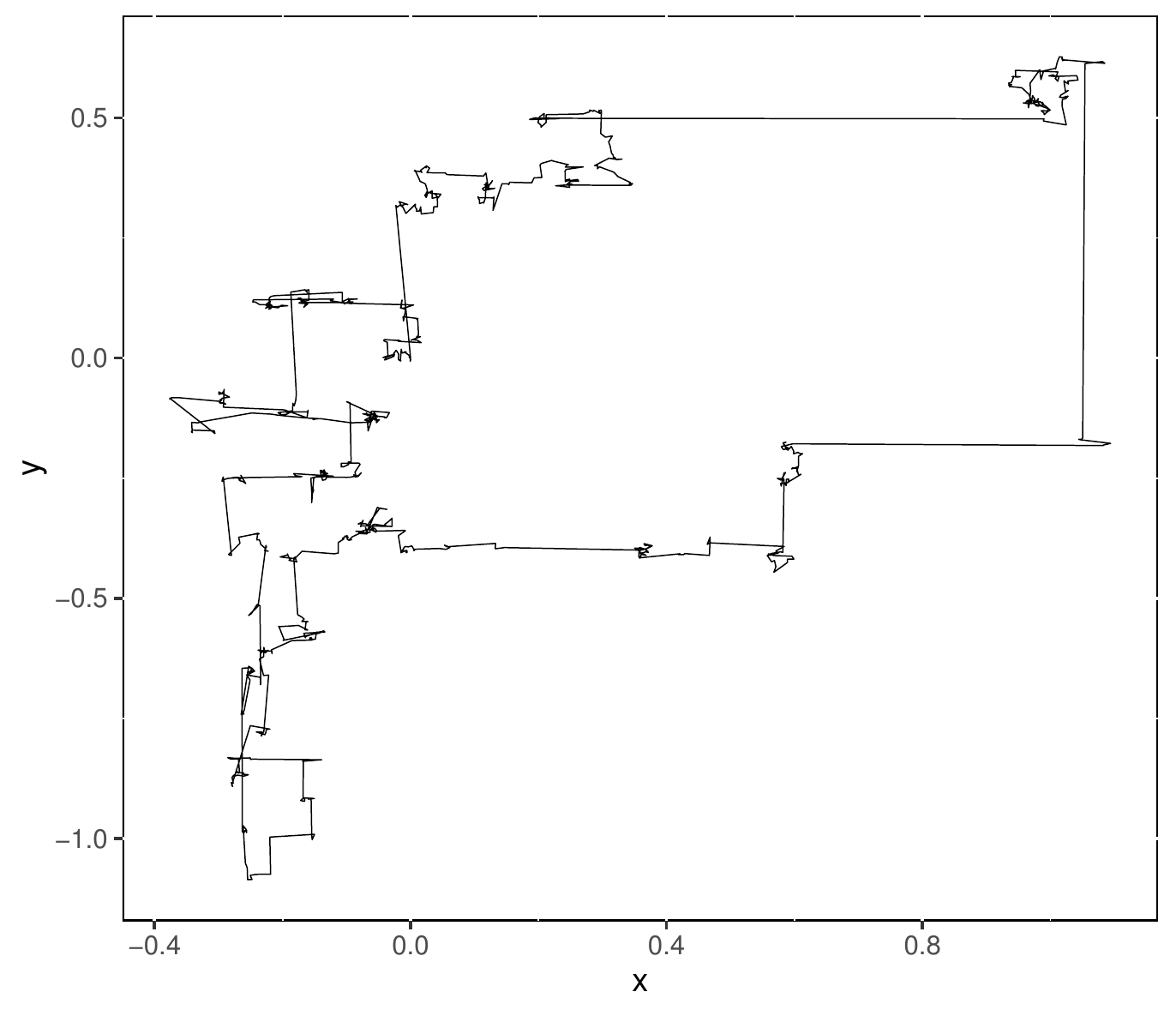}}\hspace{-0.1in}
\subfigure[$\alpha = 1.6$]{
\includegraphics[width=3.75cm,height=3.5cm]{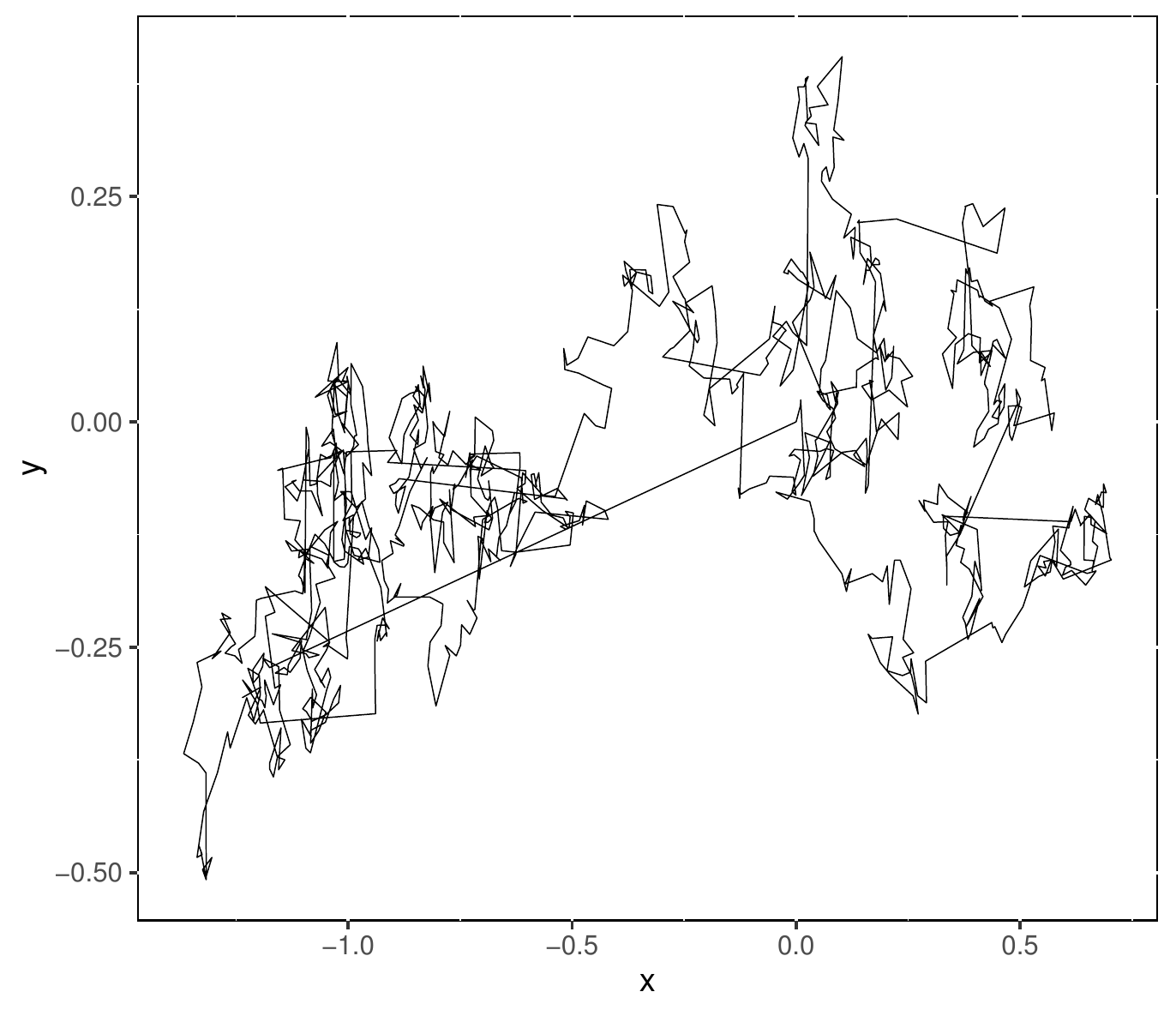}}
\caption{The motion trajectory of $\alpha$-stable process.}\label{astab}
\end{figure}

%In the previous section, we analyzed the problem on the ball of arbitrary dimension $\B_{r}^{n}$. Follow the similar idea as in the ball, a point $\bx$ inside the general regional $\Omega$ can always be found with $\bx$ as the center point and a ball tangent to the regional $\Omega$. Thus, an $\alpha$-stable process can still be approximated by a series of balls, starting from the initial point until the process leaves the regional $\Omega$, each of which is tangent to the regional $\Omega$. We show it as the ball and irregular region whose entire movement path are shown in Fig.\,\ref{ballregion}-\ref{complexregion}.

%\begin{figure}[!th]
%\centering
%\rotatebox[origin=cc]{-0}{\includegraphics[width=0.6\textwidth,height=0.45\textwidth]{figs/Complex_domain.pdf}}
%\caption{\small The path in an 2-D irregular domain.}\label{complexregion}
%\end{figure}
% \begin{figure}[!th]
% \subfigure{
%\begin{minipage}[t]{0.43\textwidth}
%\centering
%\rotatebox[origin=cc]{-0}{\includegraphics[width=1.0\textwidth,height=0.5\textwidth]{figs/Complex_domain.pdf}}
%\end{minipage}}\hspace{18pt}
%\subfigure{
%\begin{minipage}[t]{0.43\textwidth}
%\centering
%\rotatebox[origin=cc]{-0}{\includegraphics[width=0.5\textwidth,height=0.5\textwidth]{figs/3D1.png}}
%\end{minipage}}
%%\vskip -5pt
%\caption{\small The path of walk; Left: on $2$-D irregular domain; Right: on the unit ball in $3$-D.}\label{ballregion}
%\end{figure}

\begin{figure}[H]
\hspace{0.0in}
\subfigure%[The average number of steps]
{
%\label{fig:ns} %% label for first subfigure
\includegraphics[width=6.3cm,height=5cm]{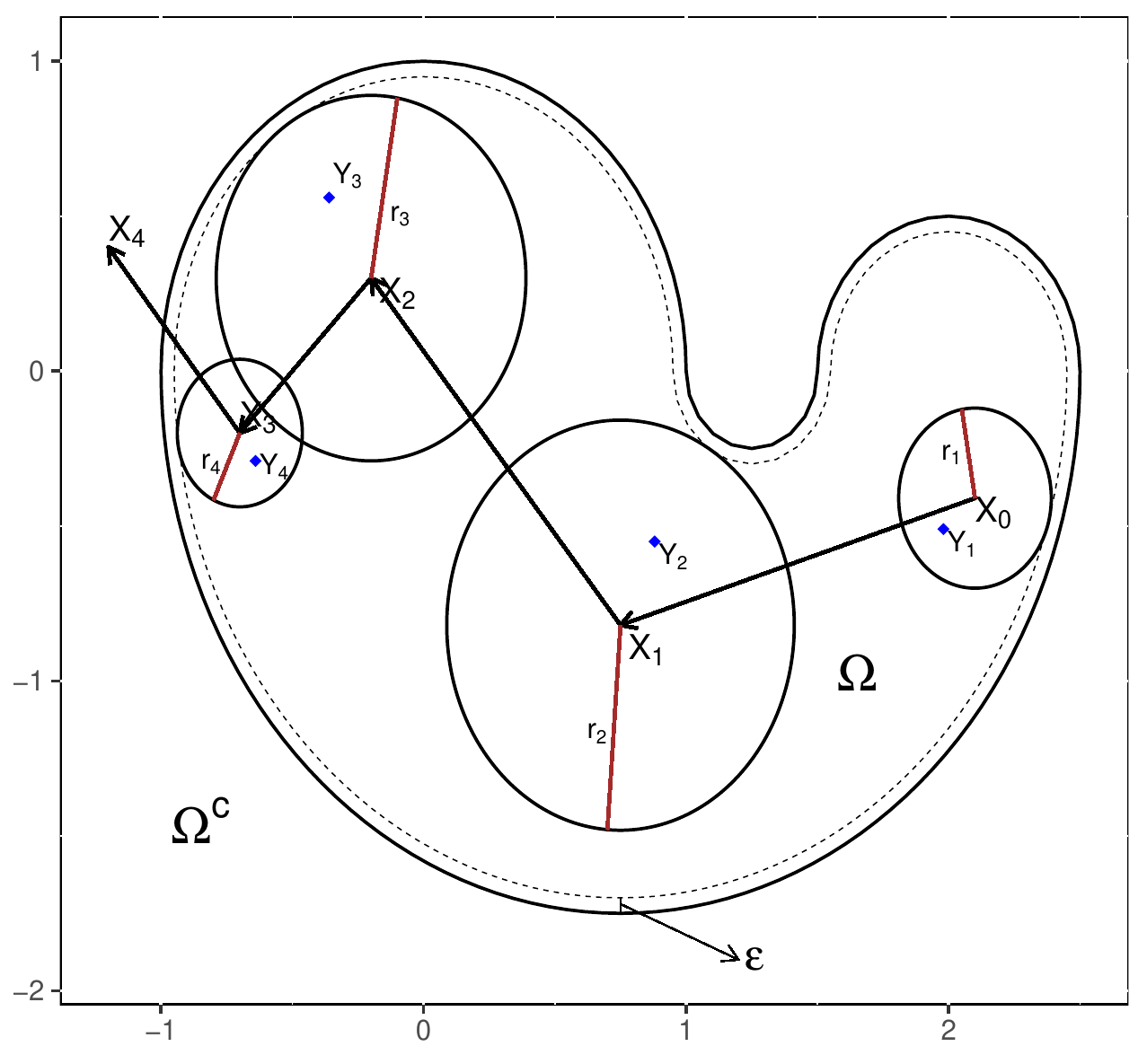}}
\hspace{0.3in}
\subfigure%[Error]
{
%\label{fig:sv} %% label for first subfigure
\includegraphics[width=6.3cm,height=5cm]{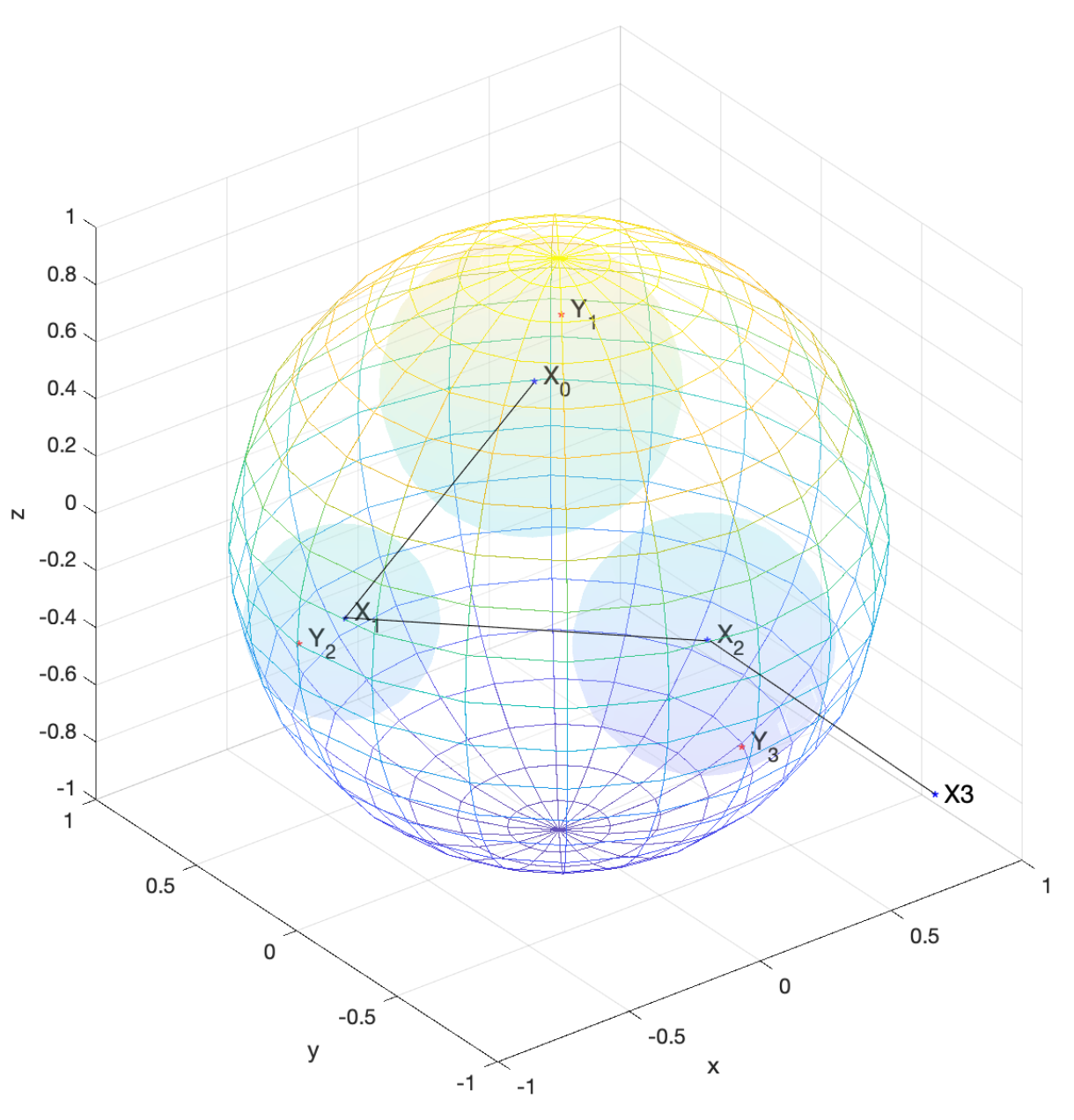}}

\caption{The path of walk; Left: on $2$-D irregular domain; Right: on the unit ball in $3$-D. }\label{ballregion}
\end{figure}

Since the Markov process will leave the given domain at a finite time, the discrete point sequence can reach outside the domain $\Omega$ after a limited number of steps. Therefore, at each jump, the probability of this point leaving the given region at the next move always be positive, that is, $P (m^{\ast}<\infty) = 1,$ where we denoted by 
\begin{equation}
m^{\ast} = \text{inf} \{ m : X_{m}\notin \Omega \},
\end{equation}
the stopping step for the random walk.  
Hence, the Monte Carlo method can be use to calculate the existing observation paths, to obtain the value of $u(\bx)$, see Fig.\,\ref{ballregion}. 

\begin{theorem}\label{irre}
Let $\alpha\in(0,2]$ and $\Omega$ be an open bounded domain, and assume that $f\in L^{1}_{\alpha}(\Omega)\cap C(\overline{\Omega})$ and $g\in L^{1}_{\alpha}(\R^{n}\!\setminus\!\Omega) $, then the solution of \eqref{ufg} in $L^{1}_\alpha(\mathbb{R}^{n})$ can be expressed as 
\begin{equation}\label{iresolu}
\begin{split}
&u(\bx)=\sum_{k=1}^{m^{\ast}-1}\zeta(X_{k})\mathbb{E}_{\tilde{Q}_{r_{k}}}[f(Y_{k})]+\mathbb{E}_{P_{r_{m^{\ast}-1}}}[g(Z_{m^{\ast}})],\;\;\; Y_{k}\in \B_{r_{k}},\; Z_{k}\in \R^{n}\!\setminus\!\B^{n}_{r_{k}},
\end{split}
\end{equation}
where $\zeta(X_{k})$ is the weight function in $k$-th ball,
\begin{equation}
\zeta(X_{k}) = \int_{\B_{r_{k}}^n} Q_{r_{k}}(X_{k},\by)\,\d \by.
\end{equation}
In the above, the first expectation represents the contribution of source $f$ in $k$-th inside ball associated with the Green's function $Q_{r_{k}}$, and the second one describes a mean value of $g$ outside the last ball $\B^{n}_{r_{m^{\ast}-1}}$ associated with $P_{r_{m^{\ast}-1}}$.
%In the above, the first expectation represents the expected contribution from source $f$ in $k$-th inside ball with the Green's function $Q_{r_{k}}$, and the second one describes a mean value with respect to the function $P_{r_{m^{\ast}-1}}$ outside the last ball $\B^{n}_{r_{m^{\ast}-1}}$.
\end{theorem}
\begin{proof}
The $\alpha$-stable process $X_{t}^{\alpha}$ can be surrounded by a series of small balls tangent to region $\Omega$, that is, we can split the process $X_{t}^{\alpha}$  according to the time of movement within the region $\Omega$. Let $0 =\tau_{0}<\tau_{1}<\tau_{2}<\cdots<\tau_{m^{\ast}} = \tau_{\Omega}$, $\tau_{k}$ $(0<k<m_{\ast})$ denotes the time of movement away from the $k$-th ball. In view of \eqref{uold}, the solution of \eqref{ufg} have the following form
\begin{equation*} 
\begin{split}
u(\bx) &= \mathbb{E}_{X_{0}^{\alpha }=\bx}\Big[\int_{0}^{\tau_{\Omega}}f(X^{\alpha}_{s})\,{\rm d}s\Big]+\mathbb{E}_{X_{0}^{\alpha }=x}\big[g(X^{\alpha}_{\tau_{\Omega}})\big] \\
 &= \mathbb{E}_{X_{0}^{\alpha}=\bx}\Big[\int_{\tau_{0}}^{\tau_{1}}f(X^{\alpha}_{s})\,{\rm d}s + \int_{\tau_{1}}^{\tau_{2}}f(X^{\alpha}_{s})\,{\rm d}s + \cdots + \int_{\tau_{m^{\ast}-1}}^{\tau_{m^{\ast}}}f(X^{\alpha}_{s})\,{\rm d}s\Big] +\mathbb{E}_{X_{0}^{\alpha }=\bx}\big[g(X^{\alpha}_{\tau_{\Omega}})\big] \\
& = \mathbb{E}_{X_{0}^{\alpha}=\bx}\Big[\sum_{k=0}^{m^{\ast}-1}\int_{\tau_{k}}^{\tau_{k+1}}f(X^{\alpha}_{s})\,{\rm d}s\Big] +\mathbb{E}_{X_{0}^{\alpha }=\bx}\big[g(X^{\alpha}_{\tau_{m^{\ast}}})\big].
\end{split}
\end{equation*}
As a result, a similar procedure as in Lemma \ref{newres} can be applied to each step of the above process. It is clear that each expectation, i.e., $\mathbb{E}_{X_{0}^{\alpha}=\bx, X_{\tau_{k}}^{\alpha}=X_{k}}\Big[\int_{\tau_{k}}^{\tau_{k+1}}f(X^{\alpha}_{s})\,{\rm d}s\Big] $ can be view as the stochastic solution of the following fractional PDEs with homogeneous boundary condition
\begin{equation*}
\begin{cases}
 (-\Delta)^{\frac{\alpha}2}u(\bx)=f(\bx),\;&{\rm in}\ \B^{n}_{r_{k}},\\[4pt]
u(\bx)=0,\quad &{\rm on}\ \R^n\backslash  \B^{n}_{r_{k}}.
\end{cases}
\end{equation*}
Moreover, in the last step,  the term $\mathbb{E}_{X_{0}^{\alpha}=\bx,{X^{\alpha}_{\tau_{m_{\ast}-1}}}=X_{m_{\ast}-1}}\Big[\int_{\tau_{m_{\ast}-1}}^{\tau_{m_{\ast}}}f(X^{\alpha}_{s})\,{\rm d}s\Big]   + \mathbb{E}_{X_{0}^{\alpha }=\bx}\big[g(X^{\alpha}_{\tau_{m_{\ast}}})\big] $ represents the stochastic solution of the following problem 
\begin{equation*}
\begin{cases}
 (-\Delta)^{\frac{\alpha}2}u(\bx)=f(\bx),\;&{\rm in}\ \B^{n}_{r_{m_{\ast}-1}},\\[4pt]
u(\bx)=g(\bx),\quad &{\rm on}\ \R^n\backslash  \B^{n}_{r_{m_{\ast}-1}}.
\end{cases}
\end{equation*}
Then, we derive from \eqref{mainthm} and \eqref{resu} that 
\begin{equation*}
\begin{split}
\mathbb{E}_{X_{0}^{\alpha}=\bx}\Big[\int_{\tau_{k}}^{\tau_{k+1}}f(X^{\alpha}_{s})\,{\rm d}s\Big] &=  \int_{\B^n_{r_{k}}}f(\by)Q_{r_{k}}(X_{k},\by)\,\d \by =\zeta(X_{k}) \int_{\B^n_{r_{k}}}f(\by)\tilde{Q}_{r_{k}}(X_{k},\by)\,\d \by\\
&=\zeta(X_{k})\mathbb{E}_{\tilde{Q}_{r_{k}}}[f(Y_{k})].
\end{split}
\end{equation*}
Since $\tau_{m_{\ast}}$ denotes the time for the process $X_{t}^{\alpha}$ to leave the $(m_{\ast}-1)$-th ball, and also leave domain $\Omega$, then 
\begin{equation*}
\mathbb{E}_{X_{0}^{\alpha }=\bx}\big[g(X^{\alpha}_{\tau_{m_{\ast}}})\big]  = \mathbb{E}_{P_{r_{m^{\ast}-1}}}[g(Z_{m^{\ast}})].
\end{equation*}
Consequently, a combination of the above results lead to the desired result.
%the solution of \eqref{ufg} can be expressed as 
%\begin{equation}
%\begin{split}
%&u(\bx)=\sum_{k=1}^{m^{\ast}-1}\zeta(X_{k})\mathbb{E}_{\tilde{Q}_{r_{k}}}[f(Y_{k})]+\mathbb{E}_{P_{r_{m^{\ast}-1}}}[g(Z_{m^{\ast}})],\;\;\; Y_{k}\in \B_{r_{k}},\; Z_{k}\in \R^{n}\!\setminus\!\B^{n}_{r_{k}}.
%\end{split}
%\end{equation}
%This ends the proof.
\end{proof}

%Let  $S_{j}$ represents the result of the $j$-th experiment of the following form
%\begin{equation}
%S_{j} =\sum_{k=0}^{m^{\ast}-1}\Big[    \int_{\B^n_{r_{k}}}f(\by)Q_{r_{k}}(X_{k}^{j},\by)\,\d \by  \Big] + g(X^{j}_{m^{\ast}})= \sum_{k=0}^{m^{\ast}-1}\zeta(X_{k}^{j})   \mathbb{E}_{\bx}[f(Y_{k+1}^{j}) ]+ g(X_{m^{\ast}}^{j}).
%\end{equation}
%Hence, we can compute the numerical solution starting at $\bx$ in the domain $\Omega$ via
%\[
%u_N(\bx)= \mathbb{E}_{\bx}(S) =\lim_{N\to \infty}\frac{1}{N}\sum_{j=1}^{N}S_{j},
%\]
%where $N$ is the total number of tests.

\subsection{Efficient algorithm}
This section describe the Monto Carlo method using Green function and the representation \eqref{resu} and \eqref{iresolu} to solve the fractional PDEs on bounded domain. To fix the idea, we only consider the algorithm based on the representation \eqref{iresolu}, as  \eqref{resu}  is a special case of \eqref{iresolu}.
%In this section, we explain how to combine the Green's function with the Feynman-Kac representation \eqref{resu} for solving the Dirichlet problem for fractional Poisson’s equation \eqref{ufg}. 

Instead of simulating the irregular trajectories of symmetric $\alpha$-stable processes, we take advantage of the explicit expressions of $\widetilde{Q}_r$ and $P_{r}$ as the probability density functions, which gives the transition probability of the symmetric $\alpha$-stable process trajectory inside the ball during its passage from the center to the boundary $\partial \Omega$ or outside the domain $\Omega$ (cf. Fig.\,\ref{ballregion}).  More precisely, the representation \eqref{iresolu} allows us to construct a sequence of inside balls $\{\B^n_{r_{i}}:=\B^{n,\bx_i}_{r_{i}}\}_{i\ge0}$ tangent to the boundary $\partial\Omega$. In each ball, we construct a random variable whose value is only obtained inside the corresponding ball. The expected value of the symmetric $\alpha$-stable process integral is replaced by the weighted expected value of this set of random variables. The centers of these balls are generated by a series of discrete transitions in continuous space $\R^{n}$ that eventually terminate when they first reach the outside of domain $\Omega$.

We denote by $X_i$ the $i$-th random variable, and then generate the corresponding inside ball centered at $X_i$ tangent the boundary $\partial\Omega$, i.e., $\B^{n,X_i}_{r_ i}$. It is evident that the closer $X_{i}$ approaches to the center of $\Omega$, the bigger area of the ball will be. In other words, the area $\Omega \setminus \B^{n,X_i}_{r_ i} $ will decrease as it moves away from the center. One verifies readily that the next time the ball moves, the less likely it is to be in $\Omega\setminus \B^{n,X_i}_{r_{i}}$. Conversely, when the center of the current ball is near $\partial \Omega$, the probability that the ball will be in $\Omega\setminus \B^{n,X_i}_{r_{i}}$ on its next move increases. Given an accuracy threshold $\varepsilon>0$, we define an inwardly `thickening' the boundary $\partial\Omega$  as (see Fig.\,\ref{ballregion} (left)): 
\begin{equation} \label{gammaep}
\Gamma_{\varepsilon}=\{\bx\in \Omega: {\rm dist}(\bx,\partial\Omega)<\varepsilon\},
\end{equation}
where ${\rm dist}(\bs x,\partial \Omega)$ denotes the distance from $\bx\in\Omega$ to $\partial\Omega$.
Therefore, to reduce the calculation, %we draw a layer $\Gamma_{\varepsilon}$ with width $\varepsilon$ at the boundary of the $\Omega$, and 
we assume that the Moto Carlo process stops when it reaches $\Gamma_\varepsilon\cup\Omega^c$.
When $\varepsilon$ is small enough if the ball moves into this domain $\Gamma_\varepsilon$, it's considered outside the domain $\Omega$.

Therefore, in each experiment, we need to simulate several important quantities as below
\smallskip

{\bf (i).} The coordinates of the center of $i$-th ball;\\[-8pt]

{\bf (ii).} The radius of the ball $r_{i}$, which is the distance in each discrete jump;\\[-8pt]

{\bf (iii).} The weight of the expectation $\zeta(\bx_{i})$ ;\\[-8pt]

{\bf (iv).} The expectation of the random variable in each ball $\mathbb{E}[f(Y)]$.\\[-8pt]

Dealing with the above quantities in high dimensional Cartesian coordinates appears difficult, therefore, to simplify the implementation, we recommend proceeding with this calculation in spherical coordinates. To this end, we denoted by $\hat \bx =\bx/|\bx|$ the unit vector along the nonzero vector $\bs x$.
We now recall  the $d$-dimensional spherical coordinates 
\begin{equation}\label{d_sphere}
\begin{split}
&x_{1}=\rho\cos\theta_{1};\; x_{2}=\rho\sin\theta_{1}\cos\theta_{2};\;\cdots\cdots; \;x_{n-1}=\rho\sin\theta_{1}\cdots\sin\theta_{n-2}\cos\theta_{n-1}; 
\\ &x_{n}=\rho\sin\theta_{1}\cdots\sin\theta_{n-2}\sin\theta_{n-1}, \;\;\; \theta_1, \cdots, \theta_{n-2}\in [0,\pi],  \;\;  \theta_{n-1}\in[0,2\pi],
\end{split}
\end{equation}
%where  $r=|\bx|$, $\theta_1, \cdots, \theta_{d-2}\in [0,\pi]$ and $\theta_{d-1}\in[0,2\pi].$  
with  the spherical volume element 
\begin{equation}\label{coord1} 
\begin{split} \d \bx=\rho^{n-1} \sin ^{n-2}\theta_{1} \sin ^{n-3}\theta_{2} \cdots \sin\theta_{n-2}\,\d \rho \,\d \theta_{1}\, \d \theta_{2} \cdots \d \theta_{n-1}.%:=\rho^{d-1} \d \rho\, \d\sigma (\hat\bx). 
\end{split}
\end{equation}

\subsubsection{Computation of the jump distance and weight function $\zeta(x)$}\label{Comp_r_zeta}
We can evaluate the jump distance (denoted by $\gamma$) from the current ball to the next ball by using the following formula, which indicates the jump distance is a uniformly distributed random number.
%We can evaluate the radius (denoted by $\gamma$) of each points by the following formula which indicates by using the inverse function of the distribution method, we can find the radius of the corresponding ball every time we select a uniformly distributed random number. Secondly, the transformation of Green's function also helps us to find the weight function for calculating the expected value of each ball.
\begin{lemma}
Let $\alpha\in(0,2]$ and let the radius of the current ball $r>0$. Assume that the ball jumps in the region $\Omega$, then the jump distance $\gamma$ from the current ball to the next ball  is given by
 \begin{eqnarray}
 \label{radius}
 \gamma(\omega;r,n,\alpha)= \sqrt{\frac{r^{2}}{B(1-\frac{\alpha}{2},\frac{\alpha}{2})- B^{-1}(\frac{\pi\,\omega}{\sin(\pi\alpha/2)};1-\frac{\alpha}{2},\frac{\alpha}{2})}},\;\;\;\omega\in(0,1),
 \end{eqnarray}
 where $B^{-1}(\cdot\,;a,b)$ denote the inverse function of  incomplete Beta function $B(\cdot\,;a,b)$, and $B(a,b):=B(1;a,b)$ denote the Beta function.
% \begin{equation}\label{constk}
% \begin{split}
%% &k:=k(\omega) = \frac{\pi \omega}{\sin(\pi \alpha/2)B(1-\frac{\alpha}{2},\frac{\alpha}{2})},\\%\;\;\; I(\chi;\nu_{1},\nu_{2}) = \frac{B(\chi,\nu_{1},\nu_{2})}{B(\nu_{1},\nu_{2})},\\
%  &k^\alpha_n(\omega) =\frac{\Gamma(n/2)}{\pi^{n/2}\widetilde{C}^\alpha_n}\omega.
%  %\\%\;\;\; I(\chi;\nu_{1},\nu_{2}) = \frac{B(\chi,\nu_{1},\nu_{2})}{B(\nu_{1},\nu_{2})},\\
%%&B(\chi,\nu_{1},\nu_{2})=\int_{0}^{\chi}  t^{\nu_{1}-1}(1-t)^{\nu_{2}-1}\d t.
% \end{split}
%\end{equation}
\end{lemma}
\begin{proof}
Let $\bx = |\bx|e_{n}$, then we have $| \bx-\bz|^{2} =   |\bx|^{2} + |\bz|^{2} -2|\bx||\bz|\cos\theta= \rho^{2} + |\bx|^{2} -2|\bx|\rho\cos\theta$ (see Fig.\,\ref{diagfig}).  
For $|\bz|>r$ and $\bx\neq\bm{0}$, we find from \eqref{Pr} and the above identity that
%For $|\bz|>r$ and $\bx\neq\bm{0}$, and by the properties of the polar coordinate system, without loss of generality and up to rotations, we assume that $\bx = |\bx|e_{n}$ to obtain the identity $| \bx-\by|^{2} = \rho^{2} + |\bx|^{2} -2|\bx|\rho\cos\theta$ (see Figure 1), we have
\begin{eqnarray}%\begin{split}
\label{polarPr}
&&\int_{r<|\bz|<\gamma}P_{r}(\bx,\bz)\,\d \bz\nonumber\\
 &&\quad\quad\quad=2\pi\widetilde{C}^\alpha_n(r^{2}-|\bx|^{2})^{\frac{\alpha}2}\prod_{k=1}^{n-3}\int_{0}^{\pi}\sin^{k}\theta \d\theta\int_{r}^{\gamma}\!\int_{0}^{\pi}\frac{\rho^{n-1}\sin^{n-2}\theta }{(\rho^{2}-r^{2})^{\alpha/2}(\rho^{2}+|\bx|^{2}-2\rho|\bx|\cos\theta)^{n/2}}\d\theta \d\rho\\
&&\quad\quad\quad=2\pi\widetilde{C}^\alpha_n(\bar{r}^{2}-1)^{\frac{\alpha}2}\prod_{k=1}^{n-3}\int_{0}^{\pi}\sin^{k}\theta \d\theta\int_{r}^{\gamma}\!  \frac{\bar{\rho}^{n-1}}{(\bar{\rho}^{2}-r^{2})^{\alpha/2}   }\lambda(\bar\rho) \,\d\bar{\rho}, \nonumber
%\end{split}
\end{eqnarray}
where we used $\bar{r} = r/|\bx| >1$, $\bar{\rho} = \rho/|\bx|>1$ in the last identity, and denoted
%Let 
\begin{equation*}
%\lambda(\bar\rho) = \frac{\bar{\rho}^{n-1}}{(\bar{\rho}^{2}-\bar{r}^{2})^{\alpha/2}}\int_{0}^{\pi}\frac{\sin^{n-2}\theta }{(\bar{\rho}^{2}+1-2\bar{\rho}\cos\theta)^{n/2}}\,\d\theta.
\lambda(\bar\rho) = \int_{0}^{\pi}\frac{\sin^{n-2}\theta }{(\bar{\rho}^{2}+1-2\bar{\rho}\cos\theta)^{n/2}}\,\d\theta.
\end{equation*}
According to \cite[(A.25)]{MR3461641}, we have
%According to \cite[Lemma A.5, Equation A.25]{MR3461641}, 
\begin{equation}\label{lamrho}
\lambda(\bar\rho) = \frac{1}{\bar{\rho}^{n-2}(\bar{\rho}^{2}-1)}\int_{0}^{\pi}\sin^{n-2}\sigma \,\d\sigma.
\end{equation}
Then, we derive from \eqref{polarPr} and \eqref{lamrho} that
\begin{equation}\label{print}
\int_{r<|\bz|<\gamma}P_{r}(\bx,\bz)\,\d \bz=2\pi \widetilde{C}^\alpha_n(\bar{r}^{2}-1)^{\alpha/2}\prod_{k=1}^{n-2}\int_{0}^{\pi}\sin^{k}\theta \,\d\theta\int_{r}^{\gamma} \frac{\bar{\rho}}{(\bar{\rho}^{2}-\bar{r}^{2})^{\alpha/2}(\bar{\rho}^{2}-1)}\,\d\bar{\rho}.
\end{equation}
Recall the formulas \cite[Proposition A.10]{MR3461641}, 
\begin{equation}\label{intforcite}
\pi\prod_{k=1}^{n-2}\int_{0}^{\pi}\sin^{k}\theta \,\d\theta=\frac{\pi^{n/2}}{\Gamma(n/2)}.
\end{equation}
%\begin{equation}
%\begin{split}
%\int_{r<|\bz|<r'}P_{r}(\bx,\bz)\,\d \bz 
%&= \frac{\pi^{n/2}}{2\Gamma(n/2)}\widetilde{C}^\alpha_n\int_{0}^{\frac{\gamma^{2}}{r^{2}}-1} \frac{1}{( 1+t ) t^{\alpha/2} }\,\d t \\
%%=a(n,\alpha)\frac{\pi^{n/2}}{\Gamma(n/2)}B(1-\frac{\alpha}{2},\frac{\alpha}{2})I(1-\frac{r^{2}}{r'^{2}};1-\frac{\alpha}{2},\frac{\alpha}{2}).
%&=\frac{\pi^{n/2}}{\Gamma(n/2)}\widetilde{C}^\alpha_nB\Big(1-\frac{r^{2}}{r'^{2}};1-\frac{\alpha}{2},\frac{\alpha}{2}\Big).
%\end{split}
%\end{equation}
%Specially, if $|\bx|=0$, we find from \eqref{Pr} that
%%\begin{equation}
%%P_{r}(0,\bz)= \widetilde{C}^\alpha_n (|\bz|^{2} -r^{2}  )^{-\frac{\alpha}2} r^{\alpha}|\bz|^{-n},
%%\end{equation}
%%then
%\begin{equation*}\begin{split}
%&\int_{r<|\bz|<\gamma}P_{r}(0,\bz)\,\d \bz = \widetilde{C}^\alpha_n\int_{r<|\bz|<\gamma} \frac{ r^{\alpha}}{(|\bz|^{2} -r^{2}  )^{\alpha/2}|\bz|^{n} }\,\d \bz= \frac{2\pi^{n/2}}{\Gamma(n/2)}\widetilde{C}^\alpha_n\int_{r<|\bz|<\gamma} \frac{r^{\alpha}}{(\rho^{2} -r^{2}  )^{\alpha/2}\rho }\,\d \rho.
%\end{split}\end{equation*}
%%where $\omega_{n} = 2\pi^{n/2}/\Gamma(n/2)$ is the measure of $(n-1)$-dimensional unit sphere. 
In view of \eqref{print} and \eqref{intforcite}, we make the change of variable $t = (\bar{\rho}/\bar{r})^{2}-1$, and obtain
%We make the change of variable $t = (\rho/r)^{2}-1$, we can rearrangement above formula as
\begin{equation}\begin{split}
\int_{r<|\bz|<\gamma}P_{r}(\bx,\bz)\,\d \bz %=\int_{r<|\bz|<\gamma}P_{r}(0,\bz)\,\d \bz 
&= \frac{\pi^{n/2}}{\Gamma(n/2)}\widetilde{C}^\alpha_n\int_{0}^{\frac{\gamma^{2}}{r^{2}}-1 }\frac{1}{( 1+t ) t^{\alpha/2} }\,\d t ,
%&=& \frac{a(n,\alpha)\omega_{n}}{2}\int_{0}^{1-\frac{r^{2}}{x^{2}}} ( 1-t )^{\alpha/2-1} t^{-\alpha/2} d t \nonumber\\
%&=\frac{a(n,\alpha)\omega_{n}}{2}B(1-\frac{\alpha}{2},\frac{\alpha}{2})I(1-\frac{r^{2}}{x^{2}};1-\frac{\alpha}{2},\frac{\alpha}{2}).
%\\&=\frac{\pi^{n/2}}{2\Gamma(n/2)}\widetilde{C}^\alpha_n\,B\Big(1-\frac{r^{2}}{\gamma^{2}};1-\frac{\alpha}{2},\frac{\alpha}{2}\Big).
\end{split}\end{equation}
let $\tilde{t} = 1/(1+t)$, then we have
\begin{equation}\begin{split}
\int_{r<|\bz|<\gamma}P_{r}(\bx,\bz)\,\d \bz 
&= \frac{\pi^{n/2}}{\Gamma(n/2)}\widetilde{C}^\alpha_n\int_{\frac{r^{2}}{\gamma^{2}}}^{1} \tilde{t}^{-\alpha/2}(1-\tilde{t})^{\alpha/2-1}\,\d \tilde{t}.
%&=\frac{\pi^{n/2}}{\Gamma(n/2)}\widetilde{C}^\alpha_n\,B\Big(1-\frac{r^{2}}{\gamma^{2}};1-\frac{\alpha}{2},\frac{\alpha}{2}\Big).
\end{split}\end{equation}

%That is $\int_{r<|\bz|<x}P_{r}(\bx,\bz)\,\d \bz  =\int_{r<|\bz|<x}P_{r}(0,\bz)\,\d \bz $.
Thanks to $P_{r}$ is a probability density function (cf. \eqref{normP}), then
\begin{equation}\label{intpr}
\int_{r<|\bz|<\gamma}P_{r}(\bx,\bz)\,\d \bz=\frac{\pi^{n/2}}{\Gamma(n/2)}\widetilde{C}^\alpha_n\Big[B(1-\frac{\alpha}{2},\frac{\alpha}{2})-B\Big(\frac{r^{2}}{\gamma^{2}};1-\frac{\alpha}{2},\frac{\alpha}{2}\Big)\Big]:=\omega\in(0,1).
\end{equation}
Moreover, by the inverse of the distribution function and \eqref{constants1}, we can easily obtain the desired result \eqref{radius}. %simulate the distribution of the radius of the sphere when it runs in the region $\Omega$,
% \begin{equation}
% \chi = H_{r}^{-1}(\omega) = \frac{r}{\sqrt{1-B^{-1}(k(\omega);1-\frac{\alpha}{2},\frac{\alpha}{2})}},
% \end{equation}
% where \[k(\omega) = \frac{2\omega}{\widetilde{C}^\alpha_n\omega_{n}},~~0\leq \omega\leq1.\]
%This ends the proof.
\end{proof}

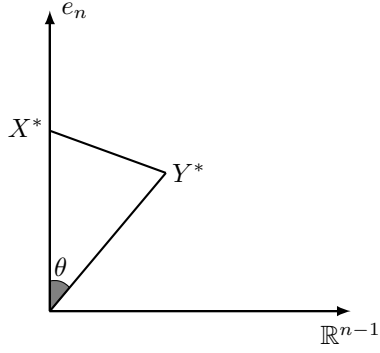
\begin{figure}
\centering
\begin{tikzpicture}[>=latex,scale=0.4,inner sep=2pt]

\path coordinate (o) at (0,0)

coordinate (x) at (0:10) node(X)  at($(o)!1!(x)$)[label=below:$\textcolor{black}{\mathbb{R}^{n-1}}$]{}
coordinate (y) at (90:10)node(Y)  at($(o)!1!(y)$)[label=right:$\textcolor{black}{e_{n}}$]{}
coordinate (v) at (50:6) node at($(o)!1!(v)$){};

\draw [fill=gray,-]($(o)!.1!(y)$)to (o) to($(o)!0.17!(v)$) to [bend right] node[above]{$\theta$}($(o)!.1!(y)$);
\draw [-]($(o)!0!(y)$)to (o) to($(o)!0.6!(y)$) to [right] node[left]{$X^{\ast}$}($(o)!.62!(y)$);
\draw [-,thick]($(o)!0.6!(y)$) to($(o)!1!(v)$) to [right] node[right]{$Y^{\ast}$}($(o)!1!(v)$);
\draw  (o)
           edge [->,thick] (x) edge [-,thick,black] (X)
           edge [->,thick] (y) edge [-,thick,black] (Y)
           edge [-,thick] (v);
\end{tikzpicture}
\caption{The diagram of coordinate axes after rotation}\label{diagfig}
\end{figure}
%\subsubsection{Computation of weight function $\zeta(x)$}
We reformulate Green's function into a concise form, which helps us evaluate the weight function $\zeta(\bx)$ of each ball more easily.
In the following Lemma, we use $r=r_i$ for simplicity.
\begin{lemma}
For $r>0$, $\zeta(\bx)$ in Lemma {\rm\ref{newres}} can be computed by
\begin{equation}\label{zetax}
\begin{split}
\zeta(\bx) =  \widehat{C}^\alpha_n\displaystyle\int_{\B^n_{r}}|\by-\bx |^{\alpha-n} \Big[B\Big(\frac{n-\alpha}{2},\frac{\alpha}{2}\Big)-B\Big(\varrho^{\ast}(\bx,\by);\frac{n-\alpha}{2},\frac{\alpha}{2}\Big)\Big]\,\d \by.
\end{split}\end{equation}
In particular, if $\bx =\bm{0}$, we have
\begin{equation} \label{zeta0}
\zeta(\bm{0}) = \frac{r^\alpha}{2^{\alpha-1}\Gamma^{2}(\frac{\alpha}{2})}\int_{0}^{1}\tilde{\rho}^{\alpha-1}\Big[B\Big(\frac{n-\alpha}{2},\frac{\alpha}{2}\Big)-B\Big(\tilde{\rho}^{2};\frac{n-\alpha}{2},\frac{\alpha}{2}\Big)\Big] \,\d\tilde{\rho}.
 \end{equation}
\end{lemma}
\begin{proof}
In view of \eqref{Qr}, we can rewritten the Green's function as  
\begin{equation*}\label{Qreq}\begin{split}
Q_{r}(\bx,\by) 
&= \widehat{C}^\alpha_n|\by-\bx |^{\alpha-n}\int_{0}^{\varrho(\bx,\by)}\frac{t^{\alpha/2-1}}{(1+t)^{\frac{n}{2}}}\,\d t\\
&= \widehat{C}^\alpha_n|\by -\bx|^{\alpha-n}\int_{\varrho^{\ast}(\bx,\by)}^{1}    \Big{ (}   \frac{1}{m}-1\Big{ )} ^{\alpha/2-1}  \Big{ (}  \frac{1}{m}\Big{ )} ^{-\frac{n}{2}}\frac{1}{m^{2}}\,\d m\nonumber\\
&= \widehat{C}^\alpha_n|\by-\bx |^{\alpha-n}\int_{\varrho^{\ast}(\bx,\by)}^{1}     m^{\frac{n-\alpha}{2}-1}(1-m)^{\alpha/2-1}\,\d m\nonumber\\
&= \widehat{C}^\alpha_n|\by-\bx |^{\alpha-n} \Big[B\Big(\frac{n-\alpha}{2},\frac{\alpha}{2}\Big)-B\Big(\varrho^{\ast}(\bx,\by);\frac{n-\alpha}{2},\frac{\alpha}{2}\Big)\Big],
\end{split}\end{equation*}
where
\begin{equation*} \begin{split}
%&&r(x,y) = \frac{(r^{2} - |x|^{2})(r^{2}-|y|^{2})}{r^{2}|x-y|^{2}},\nonumber\\
& \varrho^{\ast}(\bx,\by) = \frac{r^{2}|\bx-\by|^{2}}{(r^{2} - |\bx|^{2})(r^{2}-|\by|^{2})+r^{2}|\bx-\by|^{2}}.
\end{split}\end{equation*}
This together with \eqref{funzeta} yields% the identity \eqref{zetax}.
%\begin{equation}
%\label{bQ}
%\widetilde Q_{r}(\bx,\by) = \frac{ \widehat{C}^\alpha_n|\by-\bx |^{\alpha-n} [B(\frac{n-\alpha}{2},\frac{\alpha}{2})-B(r^{\ast}(\bx,\by);\frac{n-\alpha}{2},\frac{\alpha}{2})]}{ \zeta(\bx)},
%\end{equation}
%and
\begin{equation*}
\label{zeta}
\zeta(\bx) =  \widehat{C}^\alpha_n\displaystyle\int_{\B^n_{r}}|\by-\bx |^{\alpha-n} \Big[B\Big(\frac{n-\alpha}{2},\frac{\alpha}{2}\Big)-B\Big(\varrho^{\ast}(\bx,\by);\frac{n-\alpha}{2},\frac{\alpha}{2}\Big)\Big]\d \by.
\end{equation*}
Moreover, we obtain from \eqref{rho}  that
\begin{equation}
\label{bQ}
\widetilde Q_{r}(\bx,\by) = \frac{ \widehat{C}^\alpha_n|\by-\bx |^{\alpha-n} [B(\frac{n-\alpha}{2},\frac{\alpha}{2})-B(\varrho^{\ast}(\bx,\by);\frac{n-\alpha}{2},\frac{\alpha}{2})]}{ \zeta(\bx)}.
\end{equation}
In particular, if $\bx =\bm{0}$, we have %from \eqref{rho} that
%\begin{equation*}
%Q_{r}(0,\by) = b(n,\alpha)|\by |^{\alpha-n} [1-I(|\by|^{2};\frac{n-\alpha}{2},\frac{\alpha}{2})],
%\end{equation*}
%Therefore,
%\begin{equation*}
%\Q_{r}(0,\by) = \frac{  \widehat{C}^\alpha_n|\by |^{\alpha-n} [B(\frac{n-\alpha}{2},\frac{\alpha}{2})-B(|\by|^{2};\frac{n-\alpha}{2},\frac{\alpha}{2})]}{ \zeta(0)},
%\end{equation*}
%and
\begin{equation*}\begin{split}
\zeta(\bm{0}) &= \widehat{C}^\alpha_n \displaystyle\int_{\B^{n}_{r}} |\by |^{\alpha-n} \Big[B\Big(\frac{n-\alpha}{2},\frac{\alpha}{2}\Big)-B\Big(|\by|^{2};\frac{n-\alpha}{2},\frac{\alpha}{2}\Big)\Big]\, \d \by \\
&=2\pi\widehat{C}^\alpha_n \prod_{k=1}^{n-2}\int_{0}^{\pi}\sin^{k}\theta \d\theta \int_{0}^{r}\rho^{\alpha-1}\Big[B\Big(\frac{n-\alpha}{2},\frac{\alpha}{2}\Big)-B\Big(\rho^{2};\frac{n-\alpha}{2},\frac{\alpha}{2}\Big)\Big] \,\d\rho\\
%&=&b(n,\alpha)\beta(\frac{n-\alpha}{2},\frac{\alpha}{2})\omega_{n}\int_{0}^{r}\rho^{\alpha-1}[1-I(\rho^{2};\frac{n-\alpha}{2},\frac{\alpha}{2})] d\rho\\
&=\frac{1}{2^{\alpha-1}\Gamma^{2}(\frac{\alpha}{2})}\int_{0}^{r}\rho^{\alpha-1}\Big[B\Big(\frac{n-\alpha}{2},\frac{\alpha}{2}\Big)-B\Big(\rho^{2};\frac{n-\alpha}{2},\frac{\alpha}{2}\Big)\Big] \,\d\rho,
\end{split}\end{equation*}
%where \[b'(n,\alpha) = \frac{1}{2^{\alpha-1}\Gamma^{2}(\frac{\alpha}{2})}.\]
We then conclude that the desired result \eqref{zeta0} holds under substitution $\rho = r\tilde{\rho}$. This ends the proof.
%By substituting variable $\rho = r\tilde{\rho}$, we obtain the desired result \eqref{zeta0}. 
%%\begin{equation*}\begin{split}
%%\zeta(0) &= b'(n,\alpha)r^{\alpha}\int_{0}^{1}\tilde{\rho}^{\alpha-1}[B(\frac{n-\alpha}{2},\frac{\alpha}{2})-B(\tilde{\rho}^{2};\frac{n-\alpha}{2},\frac{\alpha}{2})] \,\d\tilde{\rho}.
%%\end{split}\end{equation*}
%This ends the proof.
\end{proof}

\begin{remark} {\em It should be pointed out that $\zeta(\bm{0})$ can be accurately computed by Jacobi-Gauss quadrature over the interval $(0,1)$ with index $(0,\alpha-1)$. For convenience, in actual computations we adopt the Monte Carlo integration for the calculation of the integral $\zeta(\bx).$
}
\end{remark}
%\begin{remark} {\em {\color{blue} This part needs to be consistent with the above Lemma, that is, related to $\zeta(\bx).$....
%In particular, if $\alpha=2$ and $n=2$, we have
%\begin{equation*}
%Q_{r}(\bx,\by)=-\frac{1}{2\pi}\log\frac{1}{|\by-\bx|},
%\end{equation*}
%which is the Green's function of Laplacian operator in $\mathbb{R}^{2}$, the corresponding Poisson kernel is 
%\begin{equation*}
%\frac{1}{2\pi}\log\frac{1-\rho^{2}}{1+\rho^{2}-2\rho\cos\theta},
%\end{equation*}
%where $\theta$ represents the angle between the vectors $\bx$ and $\by$. In the same way, if $\alpha=2$ and $n=3$,
%\begin{equation*}
%Q_{r}(\bx,\by)=-\frac{1}{4\pi}\bigg[\frac{1}{|\by-\bx|} -\frac{1}{\sqrt{|\by|^{2} +|\bx|^{2}-2\bx\by\cos\theta}-\sqrt{r^{2}+\frac{|\bx|^{2}|\by|^{2}}{r^{2}}-2\bx\by\cos\theta}} \bigg],
%\end{equation*}
%and the corresponding Poisson kernel is given by
%\begin{equation*}
%\frac{1}{4\pi r}\frac{r^{2}-|x|^{2}}{(r^{2}+|x|^{2}-2rx\cos\theta)^{3/2}}.
%\end{equation*}
%}}
%\end{remark}
%\begin{eqnarray}
%\begin{cases}
% P_{r}(0,\bz_{1}) = P_{r}(\bx,\bz_{2}),\\
% Q_{r}(0,\by_{1}) = Q_{r}(\bx,\by_{2}),
% \end{cases}
% \end{eqnarray}
%when $0<|\by_{1}| = |\by_{2} -\bx|<r$, and $|\bz_{1}| = |\bz_{2}-\bx|>r$.

%\hspace{200pt}
%
%
%

\subsubsection{The Monte Carlo method}\label{MCsimulation}
In our simulation, the explicit expression of density functions $P_r(\bx,\bz)$ and $\Q(\bx,\by)$ play an essential role in constructing transition probabilities for a sequence of discrete points outside the ball and the random variables inside the ball, respectively. We now describe the implementations of the Monte Carlo method for the solution of \eqref{ufg}. For any given point $\bx_{0}\in \Omega$, to evaluate $u(\bx_{0})$, we first calculate the shortest distance $r_{1} = {\rm dist}(\bx_{0},\Gamma_\varepsilon)$ from $\bx_0$ to the domain $\Gamma_\varepsilon$, and then draw a inside ball $\B^{n}_{r_{1}}$ tangent to the domain $\Gamma_\varepsilon$ centered at $\bx_{0}$ with radius $r_{1}$. Next, 
we construct a random variable $X_{1}$ (located at $\bx_1$) for outside sphere $\B^{n}_{r_{1}}$ associated with the corresponding density function $P_{r_{1}}(\bx_{0},\bx_{1})$, and  construct another random variable $Y_{1}$  (located at $\by_1$)  for inside the ball $\B^{n}_{r_{1}}$ associated with the density function $\Q(\bx_{0},\by_{1})$. 
%The density $\Q(\bx,\by)$ can be used to construct transition probabilities for a discrete sequence of points.  When calculating the numerical solution of a point $\bx_{0}$ in the region $\Omega$, namely $u_{\ast}(\bx_{0})$, first calculate the shortest distance $r_{1} = d(\bx_{0},\partial \Omega)$ from $\bx_0$ to the boundary $\partial \Omega$, and draw a sphere $\B^{n}_{r_{1}}$ tangent to the boundary with $\bx_{0}$ as the center and $r_{1}$ as the radius. Next, construct a random variable $X_{1}$ that is evaluated outside sphere $\B^{n}_{r_{1}}$, and its density function is $P_{r_{1}}(\bx_{0},\bx_{1})$, and then construct another random variable $Y_{1}$ that is evaluated inside the sphere $\B^{n}_{r_{1}}$ which is follow the density function $\Q(\bx_{0},\by_{1})$. 
 \smallskip

 \begin{itemize}
 \item If $X_{1} = \bx_1$ is outside the region $\Omega$, then the value of $u(\bx_{0})$ can be expressed as:
$$ u(\bx_{0}) =  \zeta(\bx_{0}) \mathbb{E}_{\bx_{0}}[f(Y_{1})] + \mathbb{E}_{\bx_{0}}[g(X_{1})].$$
  If $X_{1} = \bx_1$ is inside the region $\Omega$, we compute the distance $r_{2} = {\rm dist}(\bx_{1},\Gamma_\varepsilon)$ from $x_{1}$ to the domain $\Gamma_\varepsilon$, and draw a new inside ball $\B^{n}_{r_{2}}$ tangent to the boundary $\Gamma_\varepsilon$ centered at $\bx_{1}$ with the radius $r_{2}$. Next, construct a random variable $X_{2}$ for the outside ball $\B^{n}_{r_{2}}$ associated with the density function is $P_{r_{2}}(\bx_{1},\bx_{2})$, and construct another random variable $Y_{2}$ for the inside the ball $\B^{n}_{r_{2}}$ associated with the density function $\Q(\bx_{1},\by_{2})$. 
% 
% If $X_{1} = \bx_1$ is taken outside the region $\Omega$, then, the value of $u_\ast(\bx_{0})$ can be expressed as:
%$$ u_\ast(\bx_{0}) =  \zeta(\bx_{0}) \mathbb{E}_{\bx_{0}}[f(Y_{1})] + \mathbb{E}_{\bx_{0}}[g(X_{1})].$$
%  If $X_{1} = \bx_1$ is taken inside the region $\Omega$, we compute $r_{2} = d(\bx_{1},\partial \Omega)$ which is the shortest distance from $x_{1}$ to the region boundary $\partial \Omega$, and draw a sphere $\B^{n}_{r_{2}}$ tangent to the boundary with $\bx_{1}$ as the center and $r_{2}$ as the radius. Next, construct a random variable $X_{2}$ that is evaluated outside sphere $\B^{n}_{r_{2}}$, and its density function is $P_{r_{2}}(\bx_{1},\bx_{2})$, and then construct another random variable $Y_{2}$ that is evaluated inside the sphere $\B^{n}_{r_{2}}$ which is follow the density function $\Q(\bx_{1},\by_{2})$. 
   \item If $X_{2} = \bx_2$ is outside the region $\Omega$, then the value of $u(\bx_{0})$ can be expressed as:
 \[
 u(\bx_{0}) = \zeta(\bx_{0}) \mathbb{E}_{\bx_{0}}[f(Y_{1})] + \zeta(\bx_{1}) \mathbb{E}_{\bx_{1}}[f(Y_{1})|X_{1}] + \mathbb{E}_{\bx_{0}}[g(X_{2})].
 \]
To this ends, we let 
  \begin{equation}
  \label{uxk}
 u(X_{k}) = \zeta(X_{k})\mathbb{E}_{\bx}[f(Y_{k+1})|X_{k}] + \mathbb{E}_{\bx}[u(X_{k+1})|X_{k}],
 \end{equation}
 where we used the conditional expectations in the above as the densities are determined by the position of $X_{k}$. 
 \item By an induction argument, we suppose that the process exists in the region $\Omega$ on $m^\ast$-th step, then the solution of \eqref{ufg} is given by
 \begin{equation}\begin{split}
 u(\bx_0) & =\mathbb{E}_{\bx}[u(X_{0})] =  \mathbb{E}_{\bx}[ u(X_{m^\ast}) ] + \sum_{k=0}^{m^\ast-1} \mathbb{E}_{\bx}[u(X_{k})-  u(X_{k+1})]\\
  &=\mathbb{E}_{\bx}[ u(X_{m^\ast}) ] + \sum_{k=0}^{m^\ast-1} \mathbb{E}_{\bx}\big[u(X_{k}) -\mathbb{E}_{\bx}[  u(X_{k+1})| X_{k}]\big],
\end{split} \end{equation}
 where  we used the fact that $\mathbb{E}_{\bx}[  u(X_{k+1})| X_{k}] = \mathbb{E}_{\bx}[  u(X_{k+1})] $ in the last equality above.
  \end{itemize}
  \medskip
  
By \eqref{uxk}, we can rewrite the solution as follows:
 \begin{equation}
 \label{Eu}
 u(\bx_0)  = \mathbb{E}_{\bx}[ u(X_{m^{\ast}}) ] + \sum_{k=0}^{m^{\ast}-1}\zeta(X_{k}) \mathbb{E}_{\bx}[f(Y_{k+1})].
 \end{equation}
Then, we can construct a Monte Carlo procedure based on the random sample as 
 \begin{equation}
 \label{Si}
S_{i} = g(X_{m^{\ast}}^{i}) +  \sum_{k=0}^{m^{\ast}-1}\zeta(X_{k}^{i})  f(Y_{k+1}^{i}).
 \end{equation}
 where the index $i$ represents the $i$-th experiment. By \eqref{Eu} and \eqref{Si}, we have $u(\bx_0)\approx E(S_{i})$.

 \begin{corollary}
Let $r>0$ and $\alpha\in(0,2]$, and assume that $f\in L^{1}_{\alpha}(\Omega)\cap C(\overline{\Omega})$ and $g\in L^{1}_{\alpha}(\R^{n}\!\setminus\!\Omega) $, then the solution of \eqref{ufg} in $L^1_\alpha(\mathbb{R}^n)$ can be computed by %the statistic $\bar{S} = \frac{1}{N}\sum_{i=1}^{N}S_{i} $,
\begin{equation*}
u(\bx)= \lim_{N\to \infty}\frac{1}{N}\sum_{i=1}^{N}S_{i} =\mathbb{E}_{\bx}\Big[g(X^{i}_{m^{\ast}}) +  \sum_{k=0}^{m^{\ast}-1}\zeta(X^{i}_{k})  f(Y^{i}_{k+1}) \Big],\;\;\;\bx\in \Omega.
 \end{equation*}

 \end{corollary}
%%%%%%%%%%%%%%%%%%%%%%%%%%%%%%%%%%%%%%%%%%%%%%%%%%%%%%%%%%%%%%%%%%%%%%%%%

\begin{comment}
\begin{figure}[H]
\centering
\includegraphics[width = 0.5\linewidth]{figs/2D.eps}
\caption{The path of walk on sphere ($n = 2$)}% (the times of training: 50000)}
\end{figure}
\begin{figure}[H]
\centering
\includegraphics[width=9cm, height=9cm]{figs/3D1.png}
\caption{The path of walk on sphere ($n = 3$)}% (the times of training: 50000)}
\end{figure}
\end{comment}
Finally, the above derivations lead to the Monte Carlo algorithm outlined in {\bf Algorithm 2.1}. 
\begin{algorithm}\label{algo}
\caption{ Monte Carlo method for \eqref{ufg}.}
  \label{alg:Framwork}
    \begin{algorithmic}
        \REQUIRE (1). the number of paths $N$;\\
    \hspace{29pt}(2). the threshold of the boundary $\varepsilon$;\\
     \hspace{29pt}(3). the index $\alpha$ and the initial point $\bx_{0}\in \Omega$;
      %\hspace{29pt}(4). the formula $d(\bx, \partial\Omega)$: the distance from a fixed point $\bx$ to the boundary $\partial\Omega$;\\

        \FOR{$i$ in $1:N$}
                \STATE Set $k=0$;
        \STATE Step 1. Calculate the $r_{1}=d(\bx_{k},\partial\Omega)$;

    \STATE Step 2. Construct the random variable $X^{i}_{k+1}$ with a density function $P_{r_{k+1}}(\bx^{i}_{k},\bx^{i}_{k+1})$  by \eqref{Pr};
    
    \STATE  Step 3. Construct the random variable $Y^{i}_{k+1}$ with a density function $\Q_{r}(\bx^{i}_{k},\by^{i}_{k+1})$ by \eqref{bQ};

    \STATE  Step 4. Calculate the $r_{k+2}=d(\bx^{i}_{k+1}, \partial\Omega)$.

    \IF{$r_{k+2}\leq0$}
        \STATE $m^{\ast}=k+1$, compute $S_{i}$ by \eqref{Si}.
    \ELSE
        \STATE $k = k+1$, return to Step 4;
    \ENDIF
    \ENDFOR
        \ENSURE  $u(\bx_{0})\approx \frac{1}{N}\sum_{i=1}^{N}S_{i}$.
    \end{algorithmic}
\end{algorithm}

%In this section, we describe in detail the process of balls transfer and the radius of each ball, as well as the expected weight function. Based on the above analysis, we can see that this method can be applied to problems in multiple dimensions and complex domains.

% From Equation (\ref{Eu}), we observe that
% \begin{equation}
% \begin{split}
% \label{vE}
% v(\bx) &= \mathbb{E}_{\bx}[ v(X_{n^{\ast}}) ] + \sum_{k=0}^{n^{\ast}-1}\zeta(X_{k}) \mathbb{E}_{\bx}[f(Y_{k+1})]
% \geq \sum_{k=0}^{n^{\ast}-1}\zeta(X_{k}) \mathbb{E}_{\bx}[f(Y_{k+1})].
% \end{split}
% \end{equation}
% where $c(n,\alpha) = \frac{\Gamma(n/2)}{2^{n}\pi^{n/2}\Gamma^{2}(\alpha/2)}$.
%%%%%%%%%%%%%%%%%%%%%%%%%%%%%%%%%%%%%%%%%%%%%%%%%%%%%%%%%%%%%%%%%%%%%%%%%%%%%%%%%
 \section{The error estimate and theoretical analysis of the algorithm}
The error estimates for the proposed method on the irregular domain appear to be very difficult. To fix the idea, we only focus on the error bound, in expectation form,  of the proposed algorithm for the approximation of fractional Poisson equations \eqref{ufg} on the ball in arbitrary dimensions. Moreover, we will also analyze the effectiveness of the algorithm.

In this section, by \eqref{gammaep} and $\Omega=\B^n_r$,  we have $\Gamma_{\varepsilon} =\bar{\B}^{n}_{r}\setminus \B^{n}_{r-\varepsilon}$.
We first explore the relation between the width of layer $\Gamma_\varepsilon$ and the error, which will be used for the error analysis later on.  

  \begin{lemma}
 \label{epsilonp}
 For any $\varepsilon>0$, and let $u\in L_{\alpha}^{1}(\mathbb{R}^{n})$, $f\in L^{1}_{\alpha}(\B_{r}^{n})\cap C(\overline{\B}_{r}^{n})$ and $g\in L^{1}_{\alpha}(\R^{n}\!\setminus\!\B_{r}^{n})$ with $r>0$ and $\alpha\in(0,2]$, there holds
% can find a positive constant $\kappa \in(0,\alpha)$ such that
\begin{equation}
\big|\mathbb{E}[g(X^\prime_{m^{\ast}}) ] - \mathbb{E}[u(X_{m^{\ast}}) ]\big|\leq \frac{2^{1-\alpha}M}{\alpha\Gamma^{2}(\alpha/2)}\,\varepsilon^{\alpha},\;\;\;X_{m^\ast}\in\mathbb{R}^{n}\setminus \B^{n}_{r-\varepsilon},%\;\;\;\kappa\in(0,\alpha),
\end{equation}
where 
\begin{equation}
X^\prime_{m^{\ast}} =
 \begin{cases}
 X_{m^\ast},\;\;& {\rm if} \;\; X_{m^\ast}\in \mathbb{R}^{n}\setminus \B^{n}_{r},\\[5pt]
 \displaystyle\inf_{\chi\in\partial \B_{r}^{n}}\{|\chi - X_{m^\ast}|\},
%\chi, \{ \chi\in\partial \B_{r}^{n}, |\chi - X_{m^\ast}| = r - |X_{m^\ast}| \},
\;\;& {\rm if} \;\;X_{m^\ast}\in  \B^{n}_{r}\setminus \B^{n}_{r-\varepsilon}.
\end{cases}
\end{equation}
\end{lemma}
\begin{proof}
%The proof of the Lemma \ref{epsilonp} is based on a special Dirichlet problem for fractional Poisson's equation. 
%We proceed with the following problem  
% \begin{equation}
%\begin{cases}
%(-\Delta)^{\frac{\alpha}2}v(\bx) = f(\bx) &{\rm if}\;\bx\in \B^{n}_{r},\\[3pt]
%v(\bx) = g(\bx)&{\rm if}\;\bx\in\mathbb{R}^{n}\setminus \B^{n}_{r}.
%\end{cases}
%\end{equation}
Without loss of generality, we assume that $g(\bx) = 0$, it is straightforward to extend this result to the non-homogeneous case, i.e., $g(\bx)\neq0$. 
%Without loss of generality, let $g(\bx) = 0$ and $f(\bx)$ is bounded and continuous for $\bx\in \B^{n}_{r}$, and
%In this proof, we assume that $\bx\in \Gamma_{\varepsilon}$, where $\Gamma_{\varepsilon} =\bar{\B}^{n}_{r}\setminus \B^{n}_{r-\varepsilon}$. 
 For $X_{m^{\ast}}\in   \mathbb{R}^{n}\setminus \B^{n}_{r}$, we can find that $u(X_{m^{\ast}}) = g(X^\prime_{m^{\ast}})$, the conclusion is clearly valid.
% \[
% |\mathbb{E}[g(X_{m^{\ast}}^{i}) ] - \mathbb{E}[u(X^{i}_{m^{\ast}}) ]| <\mathcal{O}( \varepsilon).
% \]
%  Recalling the result of $Q_{r}(\bx,\by)$, 
% \begin{equation}
% \begin{split}
%Q_{r}(\bx,\by) &=\widehat{C}^\alpha_n\displaystyle|\by-\bx |^{\alpha-n} [B(\frac{n-\alpha}{2},\frac{\alpha}{2})-B(r^{\ast}(\bx,\by);\frac{n-\alpha}{2},\frac{\alpha}{2})],\nonumber
%%&\leqb(n,s)B(\frac{n}{2}-s,s)\displaystyle|2r |^{2s-n} [1-I(r^{\ast}(x,y);\frac{n}{2}-s,s)]\nonumber\\
%%&\leq b(n,\alpha)\beta(\frac{n-\alpha}{2},\frac{\alpha}{2})\displaystyle|2r |^{\alpha-n} \nonumber\leq c(n,\alpha)\cdot r^{\alpha-n},
%\end{split}
%\end{equation}
Next, we turn to prove the case $X_{m^{\ast}}\in \B^{n}_{r}$.
By using \eqref{lamrho}, \eqref{intforcite}, and \eqref{zetax}, we can obtain that %the bound of $\zeta(\bx)$ as
  \begin{eqnarray}%\begin{split}
 \label{bound_zeta}
\zeta(\bx) %&=& \int_{\B^{n}_{r}} Q_{r}(\bx,\by)\,\d\by\nonumber\\
&& =\widehat{C}^\alpha_n \int_{\B^{n}_{r}} |\by-\bx|^{\alpha-n} \Big[B\Big(\frac{n-\alpha}{2},\frac{\alpha}{2}\Big)-B\Big(\varrho^{\ast}(\bx,\by);\frac{n-\alpha}{2},\frac{\alpha}{2}\Big)\Big]\,\d\by\nonumber
\\\nonumber
&& =\widehat{C}^\alpha_n \int_{\B^{n}_{r}} |\by-\bx|^{\alpha}\cdot|\by-\bx|^{-n}  \Big[B\Big(\frac{n-\alpha}{2},\frac{\alpha}{2}\Big)-B\Big(\varrho^{\ast}(\bx,\by);\frac{n-\alpha}{2},\frac{\alpha}{2}\Big)\Big]\,\d\by\\\nonumber
&&< \widehat{C}^\alpha_n \pi\prod_{k=1}^{n-3}\int_{0}^{\pi} \sin^{k}\theta \,\d\theta\int _{0}^{r} \int_{0}^{\pi} 
(2r^{2} -|\bx|^{2})^{\alpha/2} \frac{ \rho^{n-1}\sin^{n-2}\theta}{(\rho^{2}+|\bx|^{2}-2|\bx|\rho\cos\theta)^{n/2}} \, \d\theta \d\rho \\\nonumber
&& =  \widehat{C}^\alpha_n (2r^{2} -1)^{\alpha/2}\pi\prod_{k=1}^{n-3}\int_{0}^{\pi} \sin^{k}\theta\,\d\theta\int _{0}^{r} \int_{0}^{\pi} 
 \frac{ \rho^{n-1}\sin^{n-2}\theta }{(\rho^{2}+1-2\rho\cos\theta)^{n/2}} \,\d\theta \d\rho \\
 && =  \widehat{C}^\alpha_n (2r^{2} -1)^{\alpha/2}\pi\prod_{k=1}^{n-3}\int_{0}^{\pi} \sin^{k}\theta\,\d\theta \int_{0}^{r} \frac{\rho^{n-1}}{\rho^{n-2}(\rho^{2}-1)} \int_{0}^{\pi} \sin^{n-2}\theta\,\d\theta \d\rho\\\nonumber
 && = \mathcal{C}^\alpha_n (2r^{2} -1)^{\alpha/2}\int_{0}^{r} \frac{\rho}{\rho^{2}-1}\,\d\rho= \mathcal{C}^\alpha_n/2(2r^{2} -1)^{\alpha/2} \int_{0}^{r} \frac{1}{\rho+1} + \frac{1}{\rho-1}\,\d\rho\\\nonumber
&&=\mathcal{C}^\alpha_n/2 (2r^{2} -1)^{\alpha/2} \ln|r^{2}-1|,\nonumber
%\end{split}
\end{eqnarray}
 where the constant $\mathcal{C}^\alpha_n = 2^{2-\alpha}B(\frac{n-\alpha}{2},\frac{\alpha}{2})/\Gamma^{2}(\frac{\alpha}{2})$. %By the results of Equation (\ref{vE}) and the bound of $\zeta(\bx)$, it can be deduced that
%\begin{eqnarray}
%u(\bx)\geq c \mathbb{E}(n^{\ast}),
%\end{eqnarray}
%and $c$ denotes positive constant. 
%Since $f(\bx)$ is a bounded continuous function for $\bx\in \B^{n}_{r}\setminus \Gamma_{\varepsilon} $, denote by $|f(\bx)|<M$, $M$ is a constant, it is clear that
By \eqref{resu}, \eqref{funzeta}, and \eqref{bound_zeta}, we find that for $\by \in\B^{n}_{r}\setminus\Gamma_{\varepsilon}$, 
\begin{equation*} 
\begin{split}
u(\bx)
 &= \int_{\B^{n}_{r}\setminus\Gamma_{\varepsilon}} f(\by)Q_{r}(\bx,\by)\,\d\by   \leq  \int_{\B^{n}_{r}\setminus\Gamma_{\varepsilon}} |f(\by)|Q_{r}(\bx,\by)\,\d\by \leq M\zeta(\bx)\\
%& = \mathcal{C}^\alpha_nM /2(2(r-\varepsilon)^{2} -1)^{\alpha/2} \int_{0}^{r-\varepsilon} \frac{1}{\rho+1} + \frac{1}{\rho-1}d\rho\\
&=\mathcal{C}^\alpha_nM/2 (2(r-\varepsilon)^{2} -1)^{\alpha/2} \ln|(r-\epsilon)^{2}-1|.
\end{split}
 \end{equation*}
Retrospect to \eqref{Eu}, if $X_{m^{\ast}}\in \Gamma_{\epsilon}$, it is evident that
\[u(\bx)=\zeta(X_{m^{\ast}}) \mathbb{E}_{\bx}[f(Y_{m^{\ast}+1})] + \sum_{k=0}^{m^{\ast}-1}\zeta(X_{k}) \mathbb{E}_{\bx}[f(Y_{k+1}),\]
which implies
\begin{equation*}
\begin{split}
\mathbb{E}_{\bx}[ u(X_{m^{\ast}}) ] &=  u(\bx)- \sum_{k=0}^{m^{\ast}-1}\zeta(X_{k}) \mathbb{E}_{\bx}[f(Y_{k+1})]\leq \int_{\B_{r_{m_{\ast}}}^{n}} f(\by) Q_{r}(\bx,\by)\,\d\by\leq M \int_{\B_{r_{m_{\ast}}}^{n}} Q_{r}(\bx,\by)\,\d\by.
 \end{split}
\end{equation*}
Moreover, by  \eqref{funzeta} and \eqref{zeta0}, we find that
\begin{equation}
\begin{split}
\mathbb{E}[ u(X_{m^{\ast}}) ] &\leq \frac{M}{2^{\alpha-1}\Gamma^{2}(\frac{\alpha}{2})}\int_{0}^{\varepsilon}\rho^{\alpha-1}\Big[B\Big(\frac{n-\alpha}{2},\frac{\alpha}{2}\Big)-B\Big(\rho^{2};\frac{n-\alpha}{2},\frac{\alpha}{2}\Big)\Big]\,\d\rho\\
&\leq  \frac{M}{2^{\alpha-1}\Gamma^{2}(\frac{\alpha}{2})}\int_{0}^{\varepsilon}\rho^{\alpha-1}\,\d\rho = \frac{2^{1-\alpha}M}{\alpha\Gamma^{2}(\alpha/2)}\,\varepsilon^{\alpha}.
 \end{split}
\end{equation}
Then, for $X_{m^{\ast}}\in  \Gamma_{\varepsilon}$,
\[|\mathbb{E}[u(X_{m^{\ast}}) ] - \mathbb{E}[g(X^\prime_{m^{\ast}})]|\leq \frac{2^{1-\alpha}M}{\alpha\Gamma^{2}(\alpha/2)}\,\,\varepsilon^{\alpha} .\]
A combination of the above estimates leads to desired result. 
 %This ends the proof.
  \end{proof}
 \begin{lemma}
 For any $\varepsilon>0$, and assume that $u\in L_{\alpha}^{1}(\mathbb{R}^{n})$, $f\in L^{1}_{\alpha}(\B_{r}^{n})\cap C(\overline{\B}_{r}^{n})$ and $g\in L^{1}_{\alpha}(\R^{n}\!\setminus\!\B_{r}^{n})$ with $r>0$ and $\alpha\in(0,2]$, then we have %exists a constant $\kappa\in(0,\alpha)$ such that
\begin{equation}\label{errorbound}
 \mathbb{E}_{\bx} [\bar{S}-u(\bx)]^{2}\leq\mathcal{O}(N^{-1}+\varepsilon^{2\alpha}).%\;\;\;C(r,n,\alpha)  = 2^{3-2\alpha}M^{2}/\Gamma^{4}(\alpha/2).
\end{equation}
%where $C(r,n,\alpha)  = 2^{3-2\alpha}M^{2}/\Gamma^{4}(\alpha/2)$.
 \end{lemma}
 \begin{proof}
It is known that the random sample $S_{i}$ is independently and identically distributed, so ${\rm }(S_{i},S_{j})=0$ for $i\neq j$. Then, we have
 \begin{equation*}
\V(\bar{S}) = \V\Big(\frac{1}{N}\sum_{i=1}^{N}S_{i}\Big) = \frac{1}{N^{2}}\sum_{i=1}^{N}\V(S_{i}) = \mathcal{O}(N^{-1}).
 \end{equation*}
Thanks to Lemma  \ref{epsilonp}, we arrive at 
 \begin{eqnarray}
% \begin{split}
\mathbb{E}_{\bx} [\bar{S}-u(\bx)]^{2}&&= \mathbb{E}_{\bx}\{ \bar{S} -\mathbb{E} [\bar{S}] +\mathbb{E} [\bar{S}] -u(\bx)\}^{2}\nonumber\\
&&\leq 2\mathbb{E}_{\bx} \{\bar{S}-  \mathbb{E}[\bar{S}]    \}^{2} + 2\mathbb{E}_{\bx}\{  \mathbb{E}[\bar{S}] -u(\bx) \}^{2}\nonumber\\
&&\leq 2\V(\bar{S}) + 2\bigg\{   \frac{1}{N}\sum_{i=1}^{N}\mathbb{E}[g(X_{m^{\ast}}^{i}) +  \sum_{k=0}^{n^{\ast}-1}\zeta(X_{k}^{i})  f(Y_{k+1}^{i})]\nonumber\\
&&\quad -\frac{1}{N}\sum_{i=1}^{N}\Big(\mathbb{E}[u(X^{i}_{m^{\ast}})] - \sum_{k=0}^{n^{\ast}-1}\zeta(X^{i}_{k}) \mathbb{E}[f(Y^{i}_{k+1})] \Big)\bigg\}^{2}\nonumber\\
&&\leq  2\V(\bar{S}) + \frac{2}{N^{2}}\bigg\{  \sum_{i=1}^{N} \Big|\mathbb{E}[g(X_{m^{\ast}}^{i}) ] - \mathbb{E}[u(X^{i}_{m^{\ast}}) ] \Big|  \bigg\}^{2}.
%\end{split}
 \end{eqnarray}
Therefore,  we can bound the error from the above estimates.
 %a combination of the above estimates leads to the desired result \eqref{errorbound}.
%As a result, $\mathbb{E} [\bar{S}-u(x)]^{2}\leq \mathcal{O}(N^{-1} + \varepsilon^{2p})$.
 \end{proof}
 
 \begin{lemma}
 \label{Em}
 For any $\varepsilon>0$, $\alpha\in(0,2]$, and given initial point $\bx\in\mathbb{B}^n_r$ with $r>0$, then the probability of the point $\bx$ leaving the domain $\mathbb{B}^n_r$ is positive, that is,
%For any given initial point $\bx$, there exist a positive probability associate with the point leaving the domain $\Omega $, that is, the point will reach the outside of the region the next time it moves, it follows that
 \begin{equation} 
  \label{L}
 \mathbb{E}_{\bx}(m^{\ast}) <1+\frac{q_{\ast}}{(1-p_{\ast})^{2}}, 
 \end{equation}
where the constants
\begin{eqnarray}
%&&p_{1}= \frac{\pi^{n/2}}{\Gamma(n/2)}\widetilde{C}^\alpha_n\Big[B\Big(\frac{\alpha}{2},1-\frac{\alpha}{2}\Big)-B\Big(\frac{|\bx|^{2}}{r^{2}};\frac{\alpha}{2},1-\frac{\alpha}{2}\Big)\Big],\nonumber\\\nonumber
&&p_{\ast}:= p_{\ast}(n,\alpha,r,\varepsilon)= \frac{\pi^{n/2}}{\Gamma(n/2)}\widetilde{C}^\alpha_n\Big[B\Big(\frac{\alpha}{2},1-\frac{\alpha}{2}\Big)-B\Big(\frac{\varepsilon^{2}}{r^{2}};\frac{\alpha}{2},1-\frac{\alpha}{2}\Big)\Big],\nonumber\\
&&q_{\ast}:= q_\ast(n,\alpha,r,\varepsilon)= 1- \frac{\pi^{n/2}}{\Gamma(n/2)}\widetilde{C}^\alpha_n\Big[B\Big(\frac{\alpha}{2},1-\frac{\alpha}{2}\Big)-B\Big(\frac{(r-\varepsilon)^{2}}{r^{2}};\frac{\alpha}{2},1-\frac{\alpha}{2}\Big)\Big].\nonumber
\end{eqnarray}
with $\widetilde{C}^\alpha_n$ given in \eqref{constants1}.
 \end{lemma}
 \begin{proof}
 Denote by $\B_{d_{i}}^{n,\bx_{i}}$ the $i$-th inside ball centered at $\bx_i$ and tangent the domain $\Gamma_{\varepsilon}$, that is, the ball $\B_{d_{i}}^{n,\bx_{i}}$ tangent the boundary $\partial\B^n_{r-\varepsilon}$.
Then, it is easy to check that the radius of the $i$-th ball is $d_{i}=r-|\bx_{i}|-\varepsilon$. 
%We denote by $d_{i}=r-|\bx_{i}|$ the radius of the $i$-th ball, which is the distance between the point $\bx_{i}\in \Omega$ to the boundary $\partial \Omega$. 
We set $p(d_{i},\bx_{i},\by_{i}) = \text{Pr}\{ \by_{i}\in \B^{n}_{r-\varepsilon}\setminus \B_{d_{i}}^{n,\bx_{i}}| \bx_{i}\in \B^{n,\bx_{i}}_{d_{i}}  \}$, which represents the probability that the $(i +1)$-th ball remains inside the region $\B^n_{r-\varepsilon}$ under the condition that the $i$-th ball is inside the region $\B^n_{r-\varepsilon}$.
By \eqref{intpr}, we have
 \[
 p(d_{i},\bx_{i},\by_{i}) = \int_{d_{i}<|\by_{i}|<r-\varepsilon}P_{d_{i}}(\bx_{i},\by_{i})\,\d \by_{i} = \frac{\pi^{n/2}}{\Gamma(n/2)}\widetilde{C}^\alpha_n\Big[B\Big(\frac{\alpha}{2},1-\frac{\alpha}{2}\Big)-B\Big(\frac{d_{i}^{2}}{(r-\varepsilon)^{2}};\frac{\alpha}{2},1-\frac{\alpha}{2}\Big)\Big].
 \]
To simplify the notations, we use $p_{i} = p(d_{i}, \bx_{i}, \by_{i}) \in(0,1)$, and $q_{i} = 1-p_{i}$, then 
% \begin{equation*}
% \begin{cases}
% \mathbb{P}_{\bx}(m^{\ast} = 1) =  q_{1}, \\[2pt]
%  \mathbb{P}_{\bx}(m^{\ast} = 2) =  p_{1} q_{2},\\
%  \hspace{55pt}\vdots\\
%\mathbb{P}_{\bx}(m^{\ast} = k) = p_{1} \cdot p_{2}\cdots p_{k-1}\cdot q_{k}.
%%&= \prod_{i=1}^{k-1} p(d_{i},\bx_{i},\by_{i}) \cdot [1-p(d_{k},\bx_{k},\by_{k})].
% \end{cases}\end{equation*}
 % To simplify the notations, we denote $0<p_{i} = p(d_{i},\bx_{i},\by_{i})<1$ and $q_{i} = 1-p_{i}$, then
 \begin{equation*}
 \begin{split}
 \mathbb{E}_{\bx}(m^{\ast}) &=\sum_{k=1}^{M}k \mathbb{P}_{\bx}(m^{\ast} = k)=  q_{1} + 2p_{1}q_{2} + 3p_{1}p_{2}q_{3} +\cdots + Mp_{1}p_{2}\cdots p_{M-1}q_{M}=\sum_{k = 1}^{M}k\prod_{i=1}^{k-1}p_{i}q_{k}\\
 &=\sum_{k = 1}^{M}k\prod_{i=1}^{k-1}\frac{\pi^{n/2}}{\Gamma(n/2)}\widetilde{C}^\alpha_n\Big[B\Big(\frac{\alpha}{2},1-\frac{\alpha}{2}\Big)-B\Big(\frac{d_{i}^{2}}{(r-\varepsilon)^{2}};\frac{\alpha}{2},1-\frac{\alpha}{2}\Big)\Big]\\
 &\, \,\, \, \, \, \, \,  \, \cdot \Big\{1-\frac{\pi^{n/2}}{\Gamma(n/2)}\widetilde{C}^\alpha_n\Big[B\Big(\frac{\alpha}{2},1-\frac{\alpha}{2}\Big)-B\Big(\frac{d_{k}^{2}}{(r-\varepsilon)^{2}};\frac{\alpha}{2},1-\frac{\alpha}{2}\Big)\Big] \Big\}.
 \end{split}
 \end{equation*}
Let $p_{\ast} =\displaystyle\max_{i}\{p_{i} \}$ and  $q_{\ast} =\displaystyle\max_{i}\{q_{j} \}$, it is easy to verify that
\begin{eqnarray}\label{em1}
%\begin{split}
\mathbb{E}_{x}(m^{\ast}) &&<  q_{1} + 2p_{1}q_{\ast} + 3p_{1}p_{\ast} q_{\ast} + 4p_{1}p_{\ast}^{2}q_{\ast}  +\cdots + Mp_{1}p_{\ast}^{M-1} q_{\ast} \nonumber\\[2pt]
&&=  1-2p_{1} +p_{1}\big(1 + 2q_{\ast} + 3p_{\ast}q_{\ast} + 4p_{\ast}^{2}q_{\ast} + \cdots + Mp_{\ast}^{M-1}q_{\ast}\big)\nonumber \\
&&=  1-2p_{1} +p_{1}\Big(1 + \frac{q_{\ast}}{p_{\ast}}\big(1 +  2p_{\ast} + 3p_{\ast}^{2} + 4p_{\ast}^{3}+ \cdots + Mp_{\ast}^{M-1}\big)-\frac{q_{\ast}}{p_{\ast}}\Big)\\\nonumber
&&=  1-2p_{1} +p_{1}\Big(1 + \frac{q_{\ast}}{p_{\ast}}\,\partial_{p_\ast}\!\big(p_{\ast} +  p_{\ast}^{2} + p_{\ast}^{3} + p_{\ast}^{4}+ \cdots + p_{\ast}^{M}\big)-\frac{q_{\ast}}{p_{\ast}}\Big)\\\nonumber
&& = 1-2p_{1} +p_{1}\Big(1 +\frac{q_{\ast}}{p_{\ast}}\,\partial_{p_\ast}\!\Big(\frac{p_{\ast}(1-p_{\ast}^M)}{1-p_{\ast}} \Big)-\frac{q_{\ast}}{p_{\ast}}\Big).\nonumber
%\end{split}
\end{eqnarray}
As $0<p_{\ast},q_{\ast}<1$, it is evident that $p_{\ast}^{M}\to 0$ for $M\to\infty$, which implies $\frac{p_{\ast}(1-p_{\ast}^M)}{1-p_{\ast}}\to \frac{p_{\ast}}{1-p_{\ast}}$. A direct computation gives $\partial_{p_\ast}(\frac{p_{\ast}}{1-p_{\ast}}) = \frac{1}{(1-p_{\ast})^{2}}$, then we find from \eqref{em1} that
\begin{equation*}
\mathbb{E}_{\bx}(m^{\ast}) < 1-p_{1} +\frac{p_1\,q_{\ast}}{p_\ast(1-p_{\ast})^{2}}<1+\frac{p_1\,q_{\ast}}{p_\ast(1-p_{\ast})^{2}}<1+\frac{q_{\ast}}{(1-p_{\ast})^{2}},
\end{equation*}
where we used the fact that $p_1<p_*$ in the last inequality. This ends the proof.
 \end{proof}
 \begin{figure}[H]
\hspace{0.0in}
\subfigure%[The average number of steps]
{
%\label{fig:ns} %% label for first subfigure
\includegraphics[width=6.3cm,height=5cm]{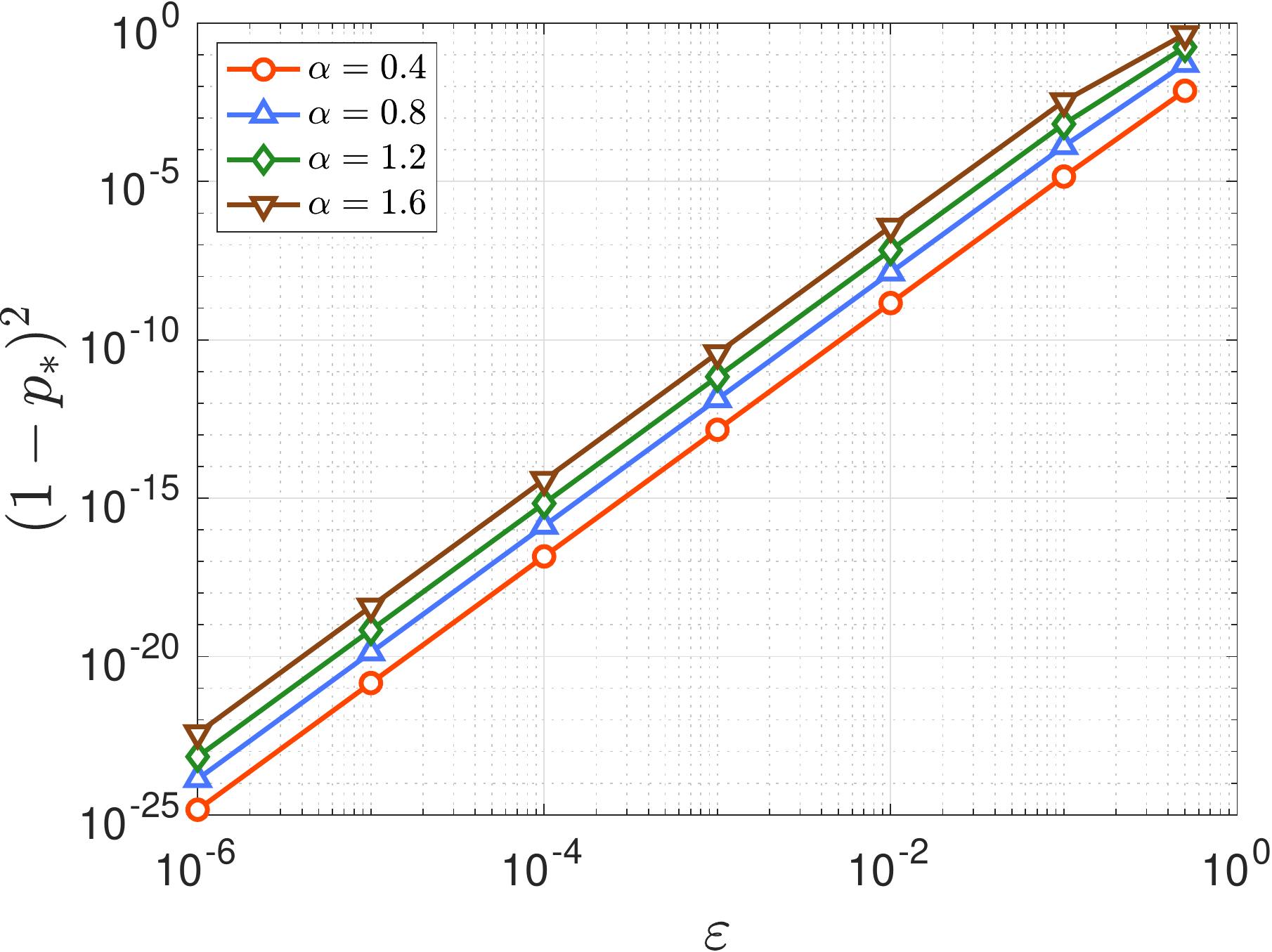}}
\hspace{0.3in}
\subfigure%[Error]
{
%\label{fig:sv} %% label for first subfigure
\includegraphics[width=6.3cm,height=5cm]{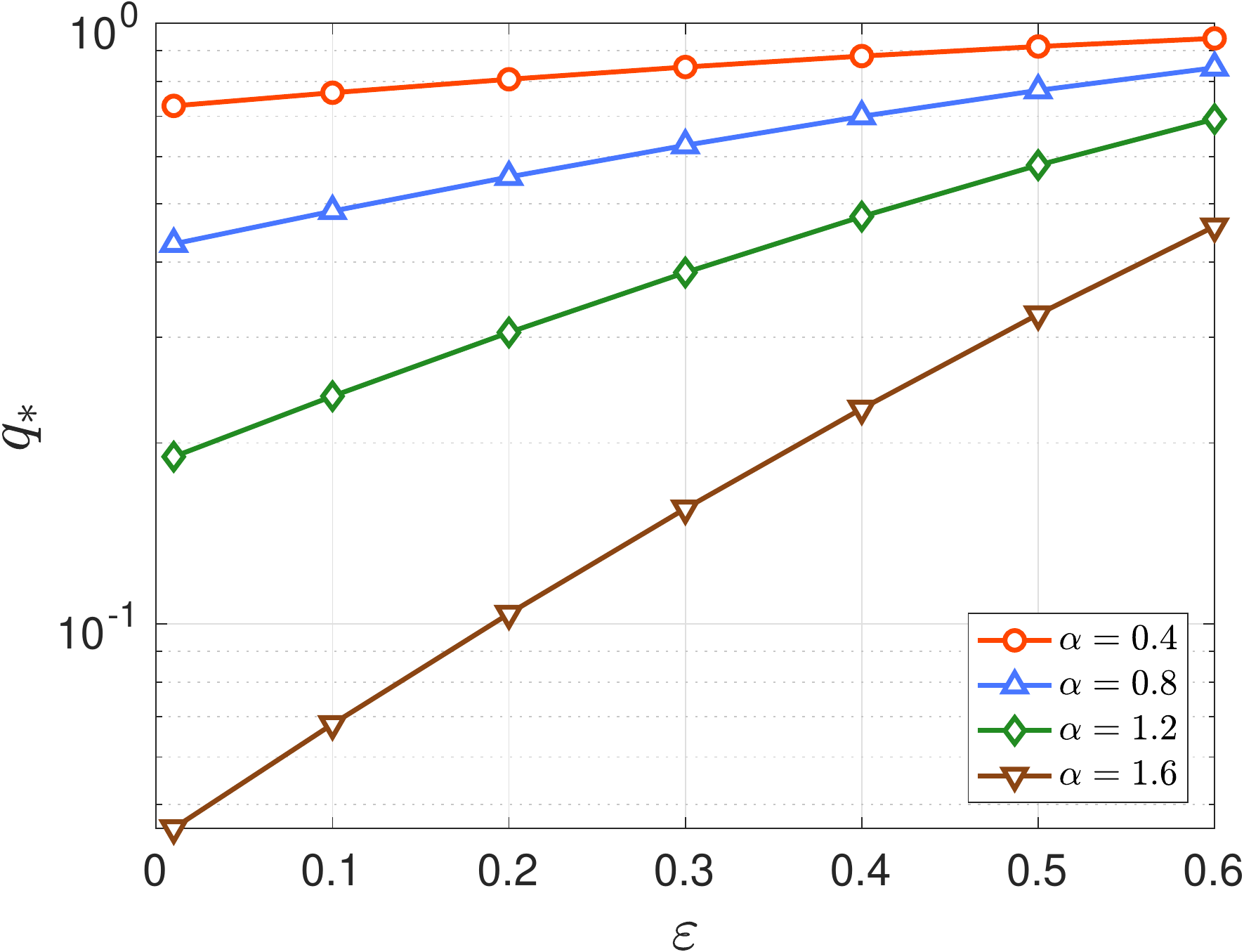}}

\caption{Left: the value $1-p^{2}_\ast$ versus $\varepsilon$; Right: the value $q_{\ast}$ versus $\varepsilon$. }\label{pq}
\end{figure}

 In Fig.\,\ref{pq}, we plot the magnitude of $(1-p_\ast)^2$ and $q_\ast$ against various $\varepsilon$ in the log-log and semi-log scale, respectively. We observe that the values of $(1-p_{\ast})^2$ and $q_\ast$ increase as $\varepsilon$ increase. %More precisely, the thickness of the $\varepsilon$ layer is in proportion to the values of $(1-p_{\ast})^2$ and $q_\ast$.
  They indicate that the greater the width $\varepsilon$, the lower the probability that the ball in the domain $\Omega$ will remain in the next step. On the contrary, the thicker the width $\varepsilon$, the higher probability the small ball will leave the domain $\Omega$ in the next step. We observe that they are consistent with the theoretical result.
 
% More precisely, the thickness of the $\varepsilon$ layer has an inverse relationship with the value of $p_{\ast}^2$, the greater the $\varepsilon$, the less likely it is that the ball in the region $\Omega$ will remain in the region in the next step. On the contrary, as shown in Fig.\,\ref{pq} (right) the thicker the $\varepsilon$ layer is, the more likely the ball will leave the region in the next step. In other words, the greater the $q_{\ast}$ will be, which is consistent with the conclusion we have drawn above.

 %%%%%%%%%%%%%%%%%%%%%%%%%%%%%%%%%%%%%%%%%%%%%%%%%%%%%%%%
\section{Numerical Results}
We present ample numerical results to illustrate our numerical methods,
 in particular, the Algorithm \ref{alg:Framwork} for computing solutions of fractional Poisson equation \eqref{ufg}. 
These examples include smooth functions and functions with low regularities at the boundary, and we also checked
the efficiency of the proposed method in two dimensions and even in $10$-dimensions. 
Note that all computations below were performed in MATLAB R2018a on a 64-bit Macbook Pro laptop with a quad-core Intel Core i5 processor at 2.3 GHz
and 8GB of RAM. %Therefore, all reported timings are in seconds.
For the sake of simplicity, we work on the problem with exact solutions. 
%
%
%
%Now let's use a few examples to verify the results proved in this article. For the elliptic problem (\ref{ufg}) with source term $f$, exterior condition $g$. 
%In what follows,  we adopt the following formula to obtain the numerical solution,
%\begin{equation*}
% u(\bx) \approx u_\ast(\bx)= \mathbb{E}_{\bx}[ g(X_{m^{\ast}}) ] + \sum_{k=0}^{m^{\ast}-1}\zeta(X_{k}) \mathbb{E}_{\bx}[f(Y_{k+1})].
% \end{equation*}
%For each implementation path, we start from the initial point $\bx$ to simulate the movement trajectory of discrete points. Starting from the current position, if the next point is still inside the region, we continue to proceed to the next step and update the position of the point until the point moves out of the given region. Thus, the value $S_{i}$ of each path is calculated, and the final statistic $\bar{S}$ is the result we need. 
%In the following numerical example, 
In what follows, all the numerical results are obtained by Algorithm \ref{alg:Framwork}, and all the numerical errors are computed via randomly selected points $\{\bx_i\}_{i=1}^{N}$ with $N=10000$ as below
\begin{equation*}\text{Error}=\frac{1}{N}\Big(\sum_{i=1}^{N}(u(\bx_i)-u_\ast(\bx_i))^2\Big)^{\frac12},\end{equation*}
where $u_\ast(\bx)$ denote the numerical solution. The other parameter are set to be $\varepsilon = 10^{-6}$. 

 \begin{exa} {\bf (Exact solution in 2D with homogeneous BCs)} We first consider \eqref{ufg} on a unit disk $\Omega=\B^2_1$ with the following exact solutions:
 \begin{equation}
 u(\bx)=(1-|\bx|^2)^{\frac{\alpha}2}_+, \;\;\;
 \end{equation}
 where $a_{+}=\max\{a,0\}$. According to {\rm\cite{Dyda2017Fractional}}, the source term reduce to a constant, i.e., $f=2^{\alpha}\Gamma^2(\frac{\alpha}{2}+1).$ In this case, the nonlocal boundary condition becomes homogeneous, that is, $g(\bx)=0$ in $\Omega^c$.  
% \begin{eqnarray}
%\begin{cases}
%(-\Delta)^{\alpha/2}u(x,y) = 2^{\alpha}\Gamma(2+\frac{\alpha}{2})  \Gamma(1 +\frac{\alpha}{2})(1-(1+\frac{\alpha}{2}) (|x|^{2}+|y|^{2})), ~~~~~~if~x,y\in B_{1}, \\
%\\
%u(x,y) = 0,~~~~~~~~~~~~~~~if~x,y\in\mathbb{R}^{n}\setminus B_{1}.
%\end{cases}
%\end{eqnarray}
%In this problem, the source term $f(x,y)=2^{\alpha}\Gamma(2+\alpha/2)  \Gamma(1 +\alpha/2)(1-(1+\alpha/2) (|x|^{2}+|y|^{2}))$ and $g(x,y)=0$. The corresponding analytical solution is given by
%\[
%u(x,y) = (1-|x|^{2}-|y|^{2})^{1+\alpha/2}.
%\]
\end{exa}
\begin{figure}[!h]
\subfigure[$\alpha = 0.4$]{
\includegraphics[width=3.75cm,height=3.5cm]{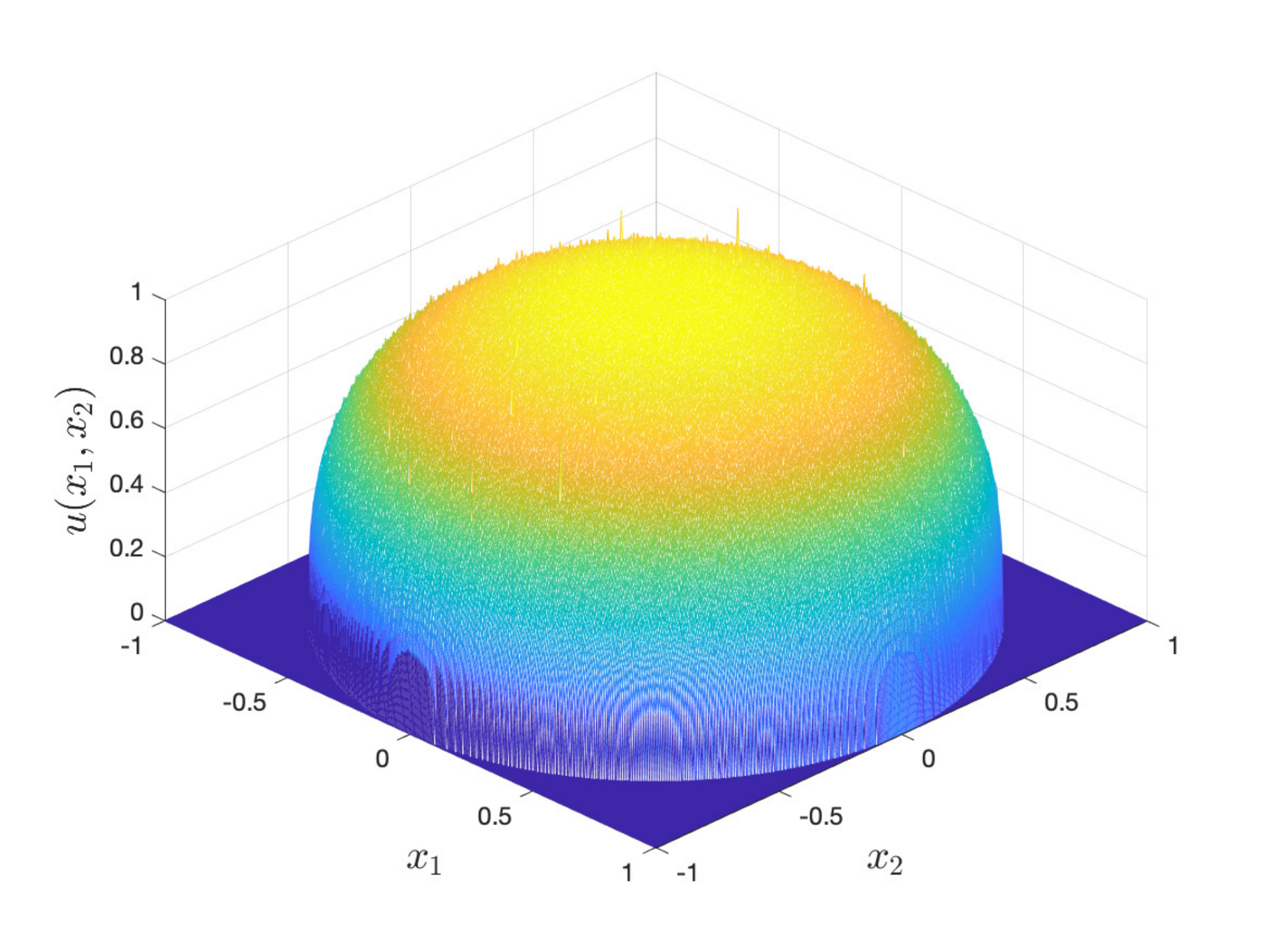}}\hspace{-0.1in}
\subfigure[$\alpha = 0.8$]{
\includegraphics[width=3.75cm,height=3.5cm]{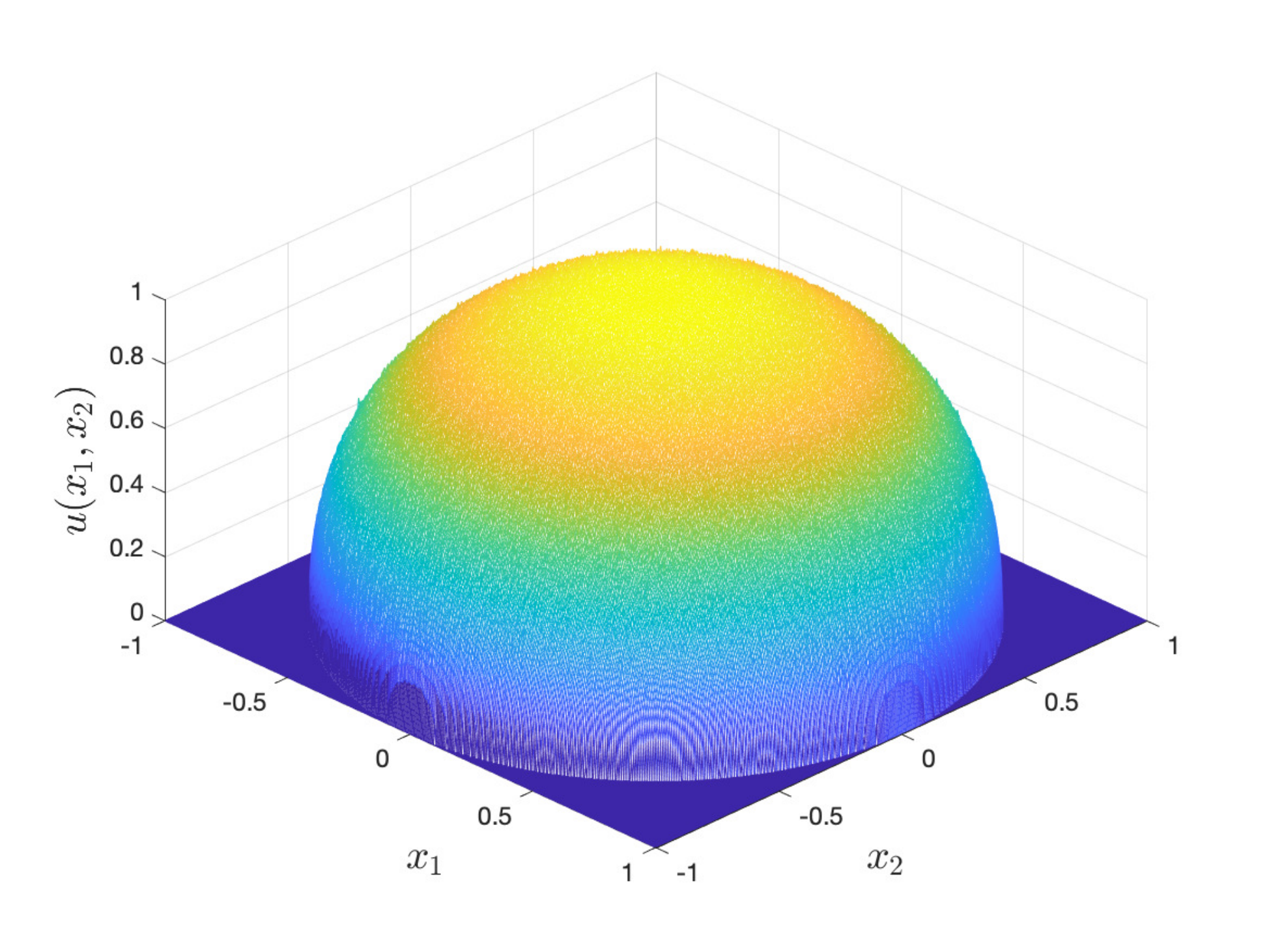}}\hspace{-0.1in}
\subfigure[$\alpha = 1.2$]{
\includegraphics[width=3.75cm,height=3.5cm]{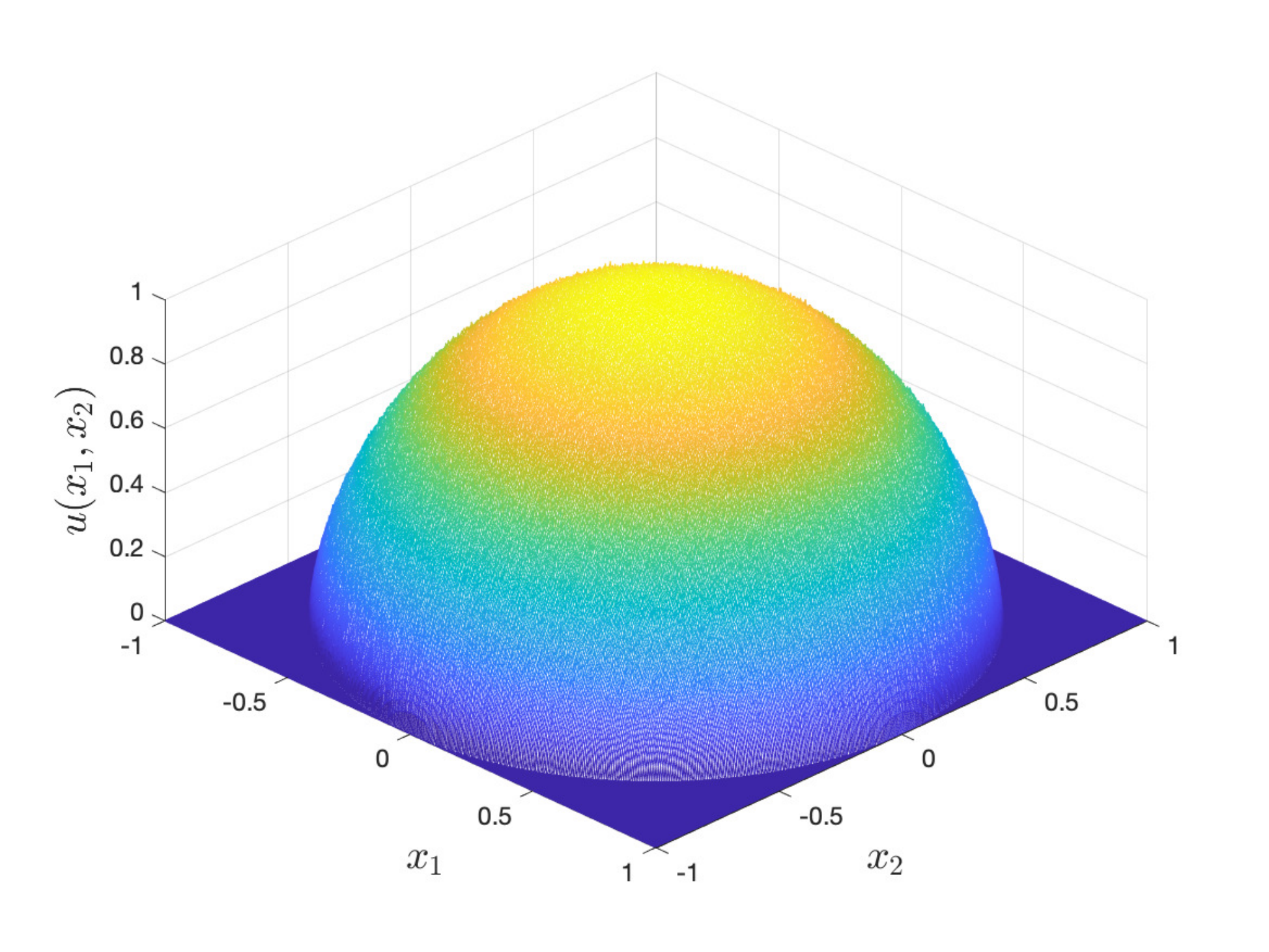}}\hspace{-0.1in}
\subfigure[$\alpha = 1.6$]{
\includegraphics[width=3.75cm,height=3.5cm]{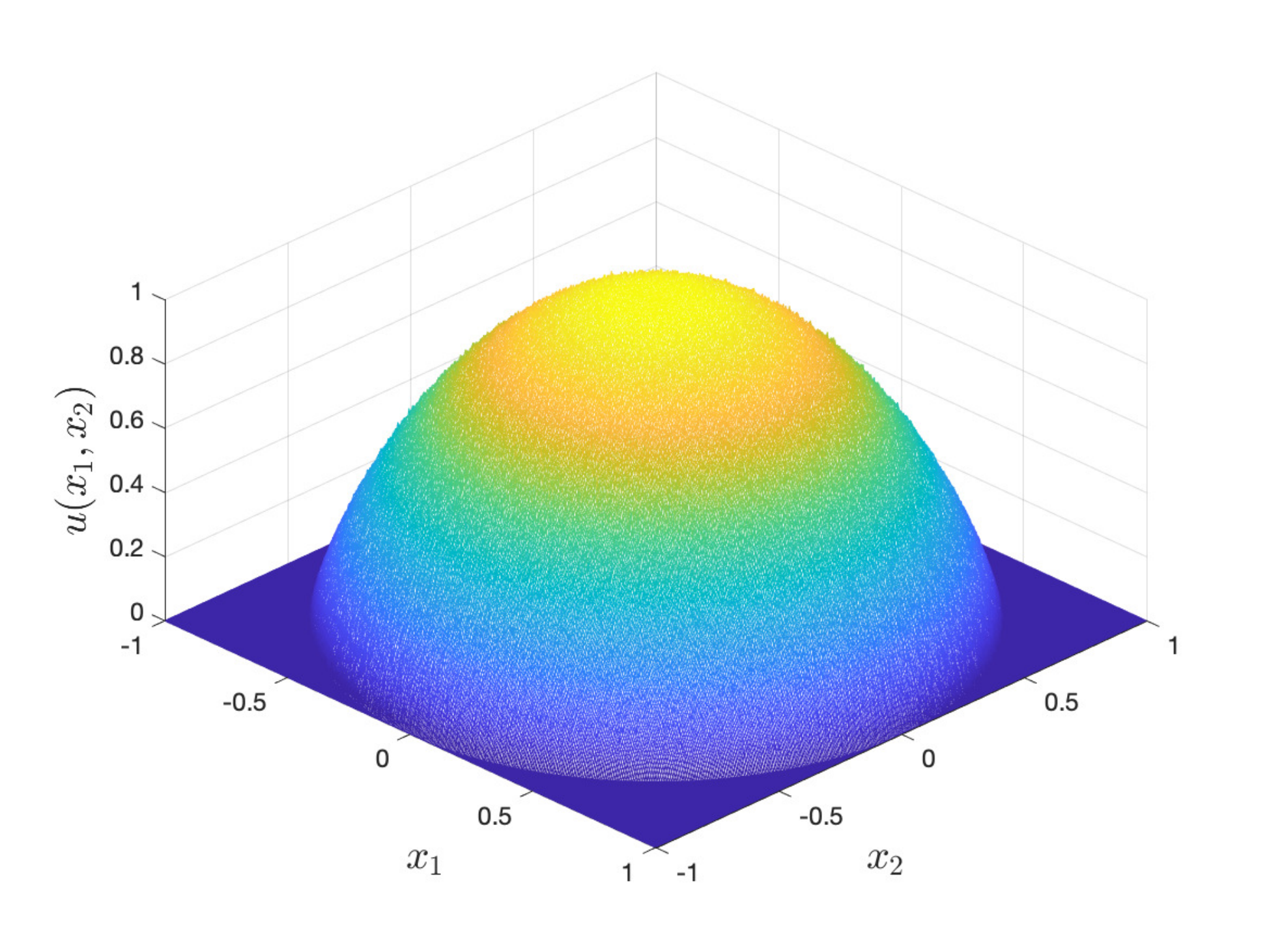}}
\caption{Profiles of the numerical solutions  with various $\alpha$.}\label{exam2dnum}
\end{figure}

In order to demonstrate the efficiency of the proposed algorithm, we plot in Fig.\,\ref{exam2dnum} the profiles of numerical solutions with $\alpha=0.4,\,0.8,\,1.2,\,1.6$, in which we compute the approximate result at each given point $\bx=(x_{1},x_{2})$ by employing 5000 paths. The exact solution has singularity near the boundary and behaves like ${\rm dist}(\bx,\partial\Omega)^{\frac{\alpha}2}$, where ${\rm dist}(\bx,\partial\Omega)$ denotes the distance function from $\bx$ to the boundary $\partial\Omega$.
It can be seen from Fig.\,\ref{exam2dnum}, with the decrease of the index $\alpha$,  the singular layers near the boundary are thinner and sharper as expected. We note that the computational complexity of the algorithm is $\mathcal{O}(N)$, where $N$ represents the number of paths we set.

\begin{figure}[H]
\hspace{0.15in}
\subfigure%[The average number of steps]
{
%\label{fig:ns} %% label for first subfigure
\includegraphics[width=6.3cm,height=5cm]{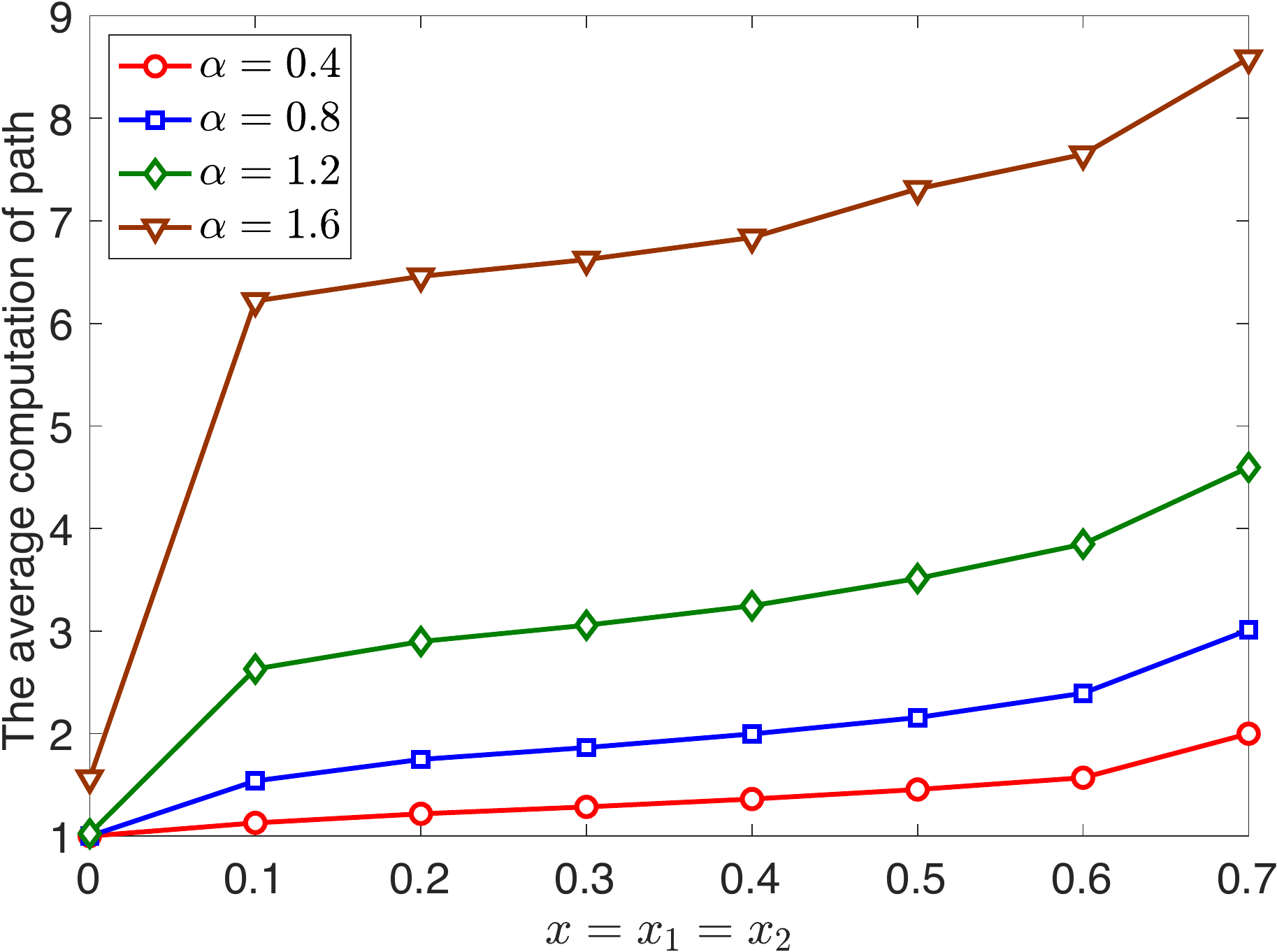}}
\hspace{0.3in}
\subfigure%[Error]
{
%\label{fig:sv} %% label for first subfigure
\includegraphics[width=6.3cm,height=5cm]{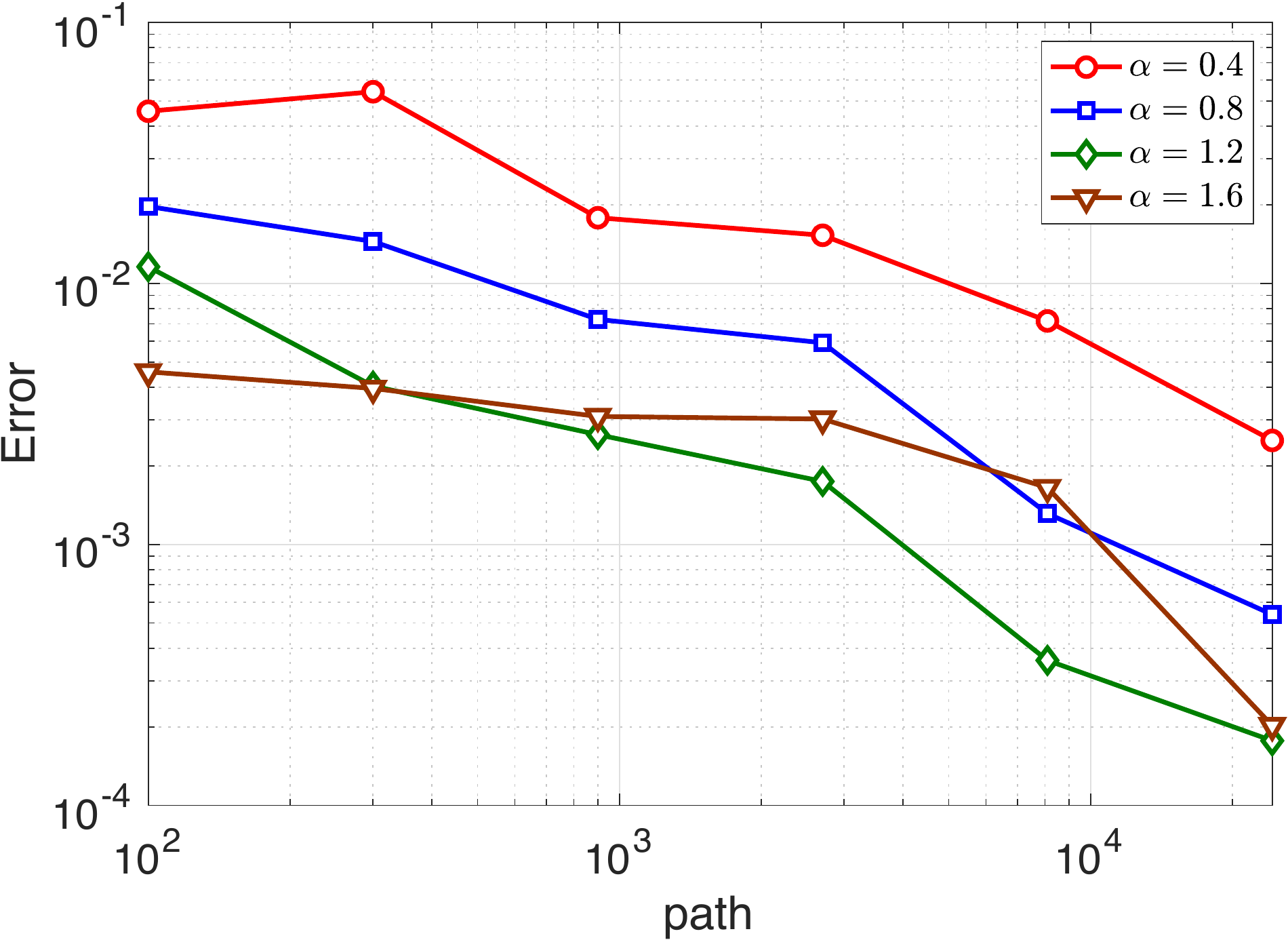}}

\caption{Simulation for two-dimension. Left: the average number of steps for fixed point with various $\alpha$; Right: numerical errors against $\varepsilon$. }\label{exam2d}
\end{figure}

In Fig.\,\ref{exam2d} (left), we shows the average number of iterations required for each path when calculating $u(\bx)$ with different $\alpha$, that is, the average of steps required for point $\bx$ to escape from domain $\B_{1}^2$.  We observe that when the starting point is farther away from the center, it will need more iterations, as the corresponding circle is smaller for the point near the boundary. In other words, the larger the area of $\B^{2}_{1}\setminus\B^{2}_{r_{1}}$ will be, and the more likely the next point will fall in this area. In addition, when all conditions are set to be the same, as expected, the number of iterations increases as $\alpha$ increases. As $\alpha$ tends to $2$, the distance from a fixed point $\bx$ to the next point will be smaller, which leads to an increase in the number of iterations escaping from region $\B^{2}_{1}$.  In Fig.\,\ref{exam2d} (right), we plot the numerical errors, in log-log scale, against the number of paths with different fractional order $\alpha$.

\begin{exa} {\bf (Source problem in 2D with non-homogeneous BCs)} We next consider \eqref{ufg} with the following source functions:
\begin{equation}
f(\bx)=\Gamma(2+\alpha)_2F_1\Big(\frac{2+\alpha}{2},\frac{3+\alpha}{2};1;-|\bx|^2\Big),\;\;{\rm on}\;\Omega=\B^2_1.
\end{equation}
Fortunately, we have the explicit expression of exact solution $u(\bx)= (1+|\bx|^2)^{-\frac32}$ on $\R^2,$
which implies the nonhomegeneous boundary condition $g(\bx)=u(\bx)$ on $\Omega^c$.

\begin{figure}[!h]
\hspace{0.15in}
\subfigure{
\includegraphics[width=6.3cm,height=5cm]{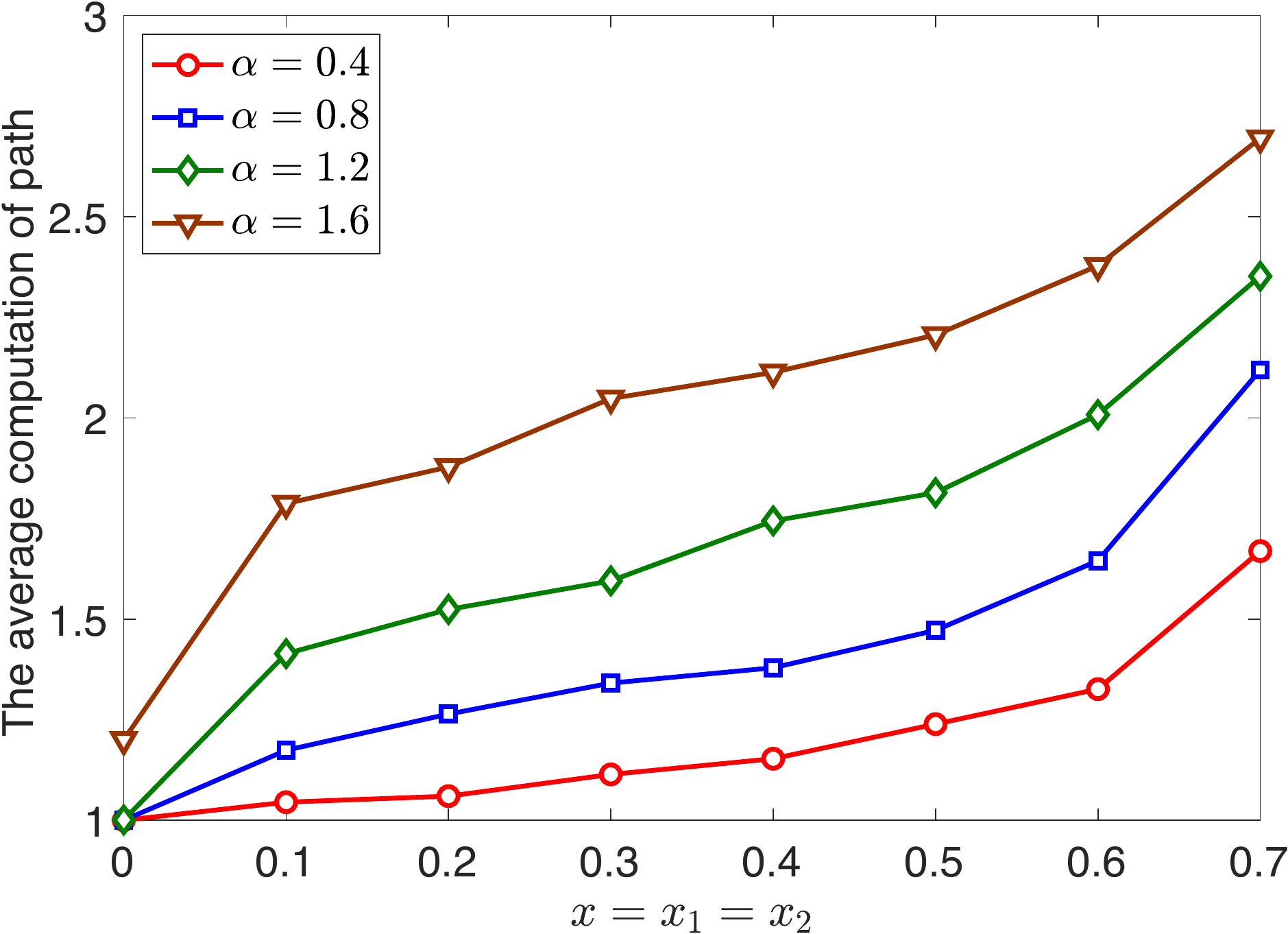}}
\hspace{0.3in}
\subfigure{
\includegraphics[width=6.3cm,height=5cm]{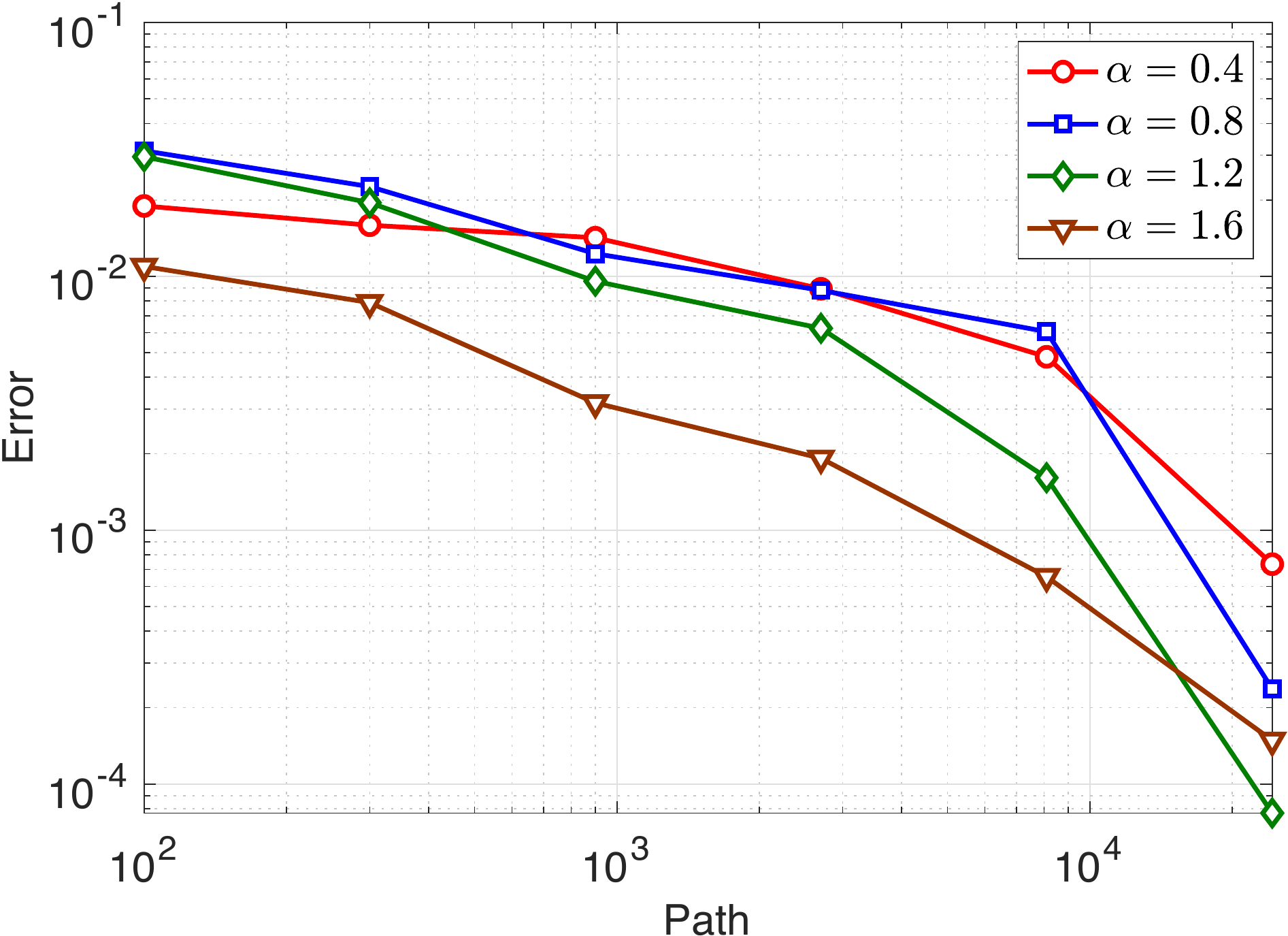}}

\caption{Simulation for two-dimension. Left: the average number of steps for fixed point with various $\alpha$; Right: numerical errors against the number of path with various $\alpha$.}\label{exam2nh}
\end{figure}
 
% \begin{eqnarray}
%\begin{cases}
%(-\Delta)^{\alpha/2}u(x,y) = 2^{\alpha}\Gamma(2+\alpha/2)  \Gamma(1 +\alpha/2)(1-(1+\alpha/2) (|x|^{2}+|y|^{2})), ~~~~~~if~x,y\in B_{1}, \\
%\\
%u(x,y) = \frac{1}{(1+(|x|^{2}+|y|^{2}))^{\alpha/2}},~~~~~~~~~~~~~~~if~x,y\in\mathbb{R}^{n}\setminus B_{1}.
%\end{cases}
%\end{eqnarray}
%In this problem, the source term $f(x,y)=2^{\alpha}\Gamma(2+\alpha/2)  \Gamma(1 +\alpha/2)(1-(1+\alpha/2) (|x|^{2}+|y|^{2}))$ and $g(x,y)=1/(1+(|x|^{2}+|y|^{2}))^{\alpha/2}$. The corresponding analytical solution is given by
\end{exa}

We use the Monte Carlo method suggested in Algorithm\,\ref{alg:Framwork} to resolve this problem. 
In Fig.\,\ref{exam2nh} (left), we plot the average number of iterations required for each path to simulate $u(\bx)$ with various fractional order $\alpha$. The errors of numerical solution against the number of path with various $\alpha$ in log-log scale are shown in Fig.\,\ref{exam2nh} (right).
We observe that the numerical errors decay as the number of path increases. They indicate that our algorithm is very effective for IFL with nonhomogeneous BCs.

\begin{figure}[!h]
%\hspace{0.1in}
%\subfigure {
%\includegraphics[width=6.3cm,height=5cm]{figs/comp_ac_3D.pdf}}
%\hspace{0.3in}
%\subfigure {
%\includegraphics[width=6.3cm,height=5cm]{figs/err_N_path_3D.pdf}}
%
\hspace{0.15in}
\subfigure {
\includegraphics[width=6.3cm,height=5cm]{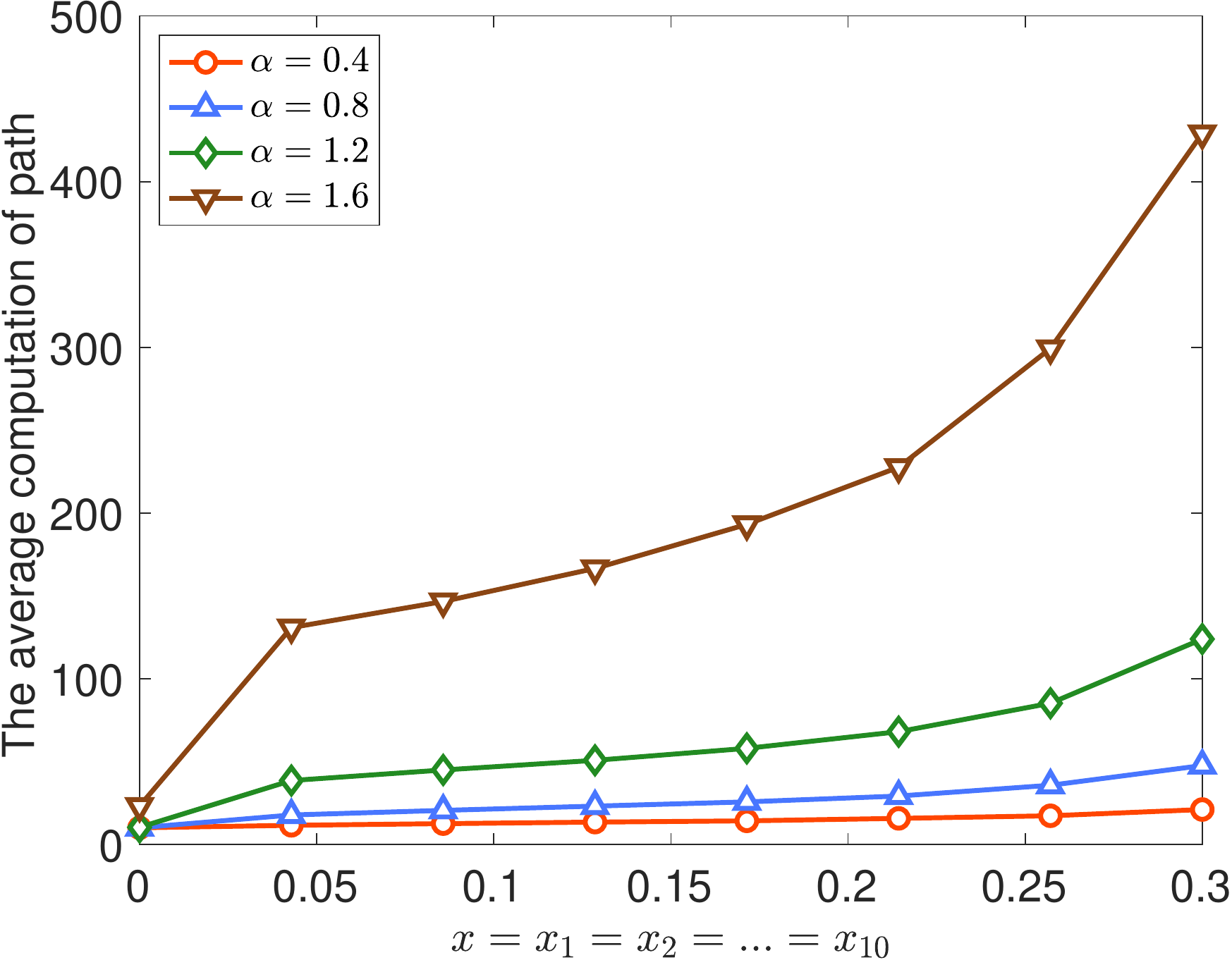}}
\hspace{0.3in}
\subfigure {
\includegraphics[width=6.3cm,height=5cm]{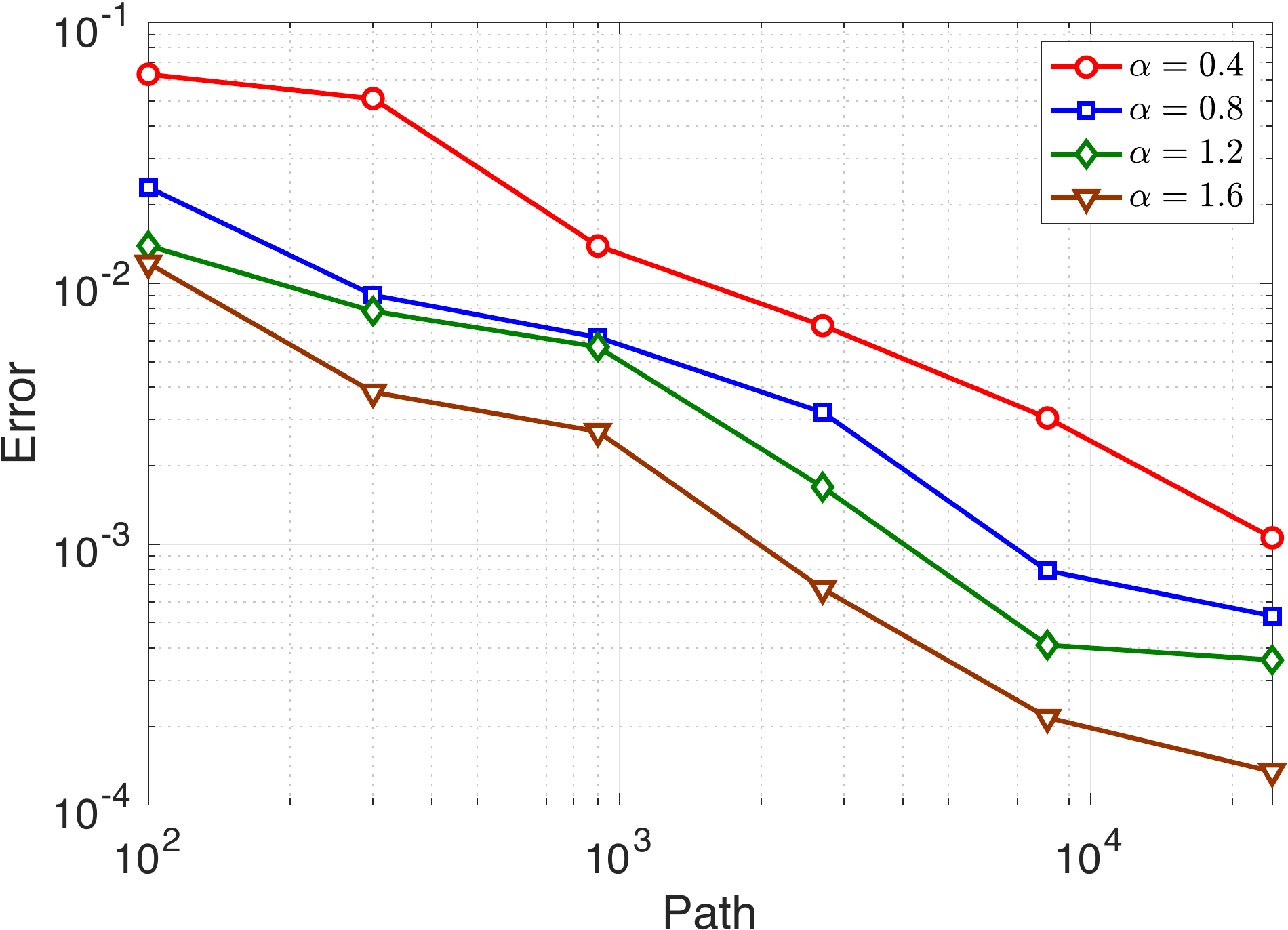}}
\caption{Simulation for $10$-dimension. Left: the average number of steps for fixed point with various $\alpha$; Right: numerical errors against the number of path with various $\alpha$.}\label{exam3d10d}
\end{figure}

 \begin{exa} {\bf (Exact solution in 10D with homogeneous BCs)} Again, we consider \eqref{ufg} in high dimensions with the following exact solutions {\rm(cf. \cite{Dyda2017Fractional})}:
 \begin{equation}
 u(\bx)=(1-|\bx|)^{\frac{\alpha}2}_+, \;\;\bx\in \Omega=\B^{10}_1, 
 \end{equation}
 then the source term reduce to a constant, i.e., $f(\bx)=2^{\alpha}\Gamma(1+\frac{\alpha}2)\Gamma(\frac{n+\alpha}2)/\Gamma(\frac{n}2).$ In this case, the nonlocal boundary condition is homogeneous, that is, $g(\bx)=0$ in $\Omega^c$.  
 
% \begin{figure}[H]
%\subfigure%[The average number of steps]
%{
%%\label{fig:ns} %% label for first subfigure
%\includegraphics[width=6.5cm,height=5cm]{figs/comp_ac_10D.pdf}}
%\hspace{0.5in}
%\subfigure%[Error]
%{
%%\label{fig:sv} %% label for first subfigure
%\includegraphics[width=6.5cm,height=5cm]{figs/err_N_path_10D.pdf}}
%%\hspace{0.5in}
%\caption{Simulation for $10$-dimension. Left: the average number of steps and sample variance for fixed point with various $\alpha$; Right: numerical errors against the number of path with various $\alpha$.}\label{exam10d}
%\end{figure}

\end{exa}

In Fig.\,\ref{exam3d10d} (left), we plot the average number of iterations required for each path when calculating $u(\bx)$ for different $\alpha$. The numerical errors of numerical solution against the number of path with various $\alpha$ in log-log scale are plotted in Fig.\,\ref{exam3d10d} (right).
As with the results in 2D, Fig.\,\ref{exam3d10d} demonstrate that the average number of iterations required to compute $u(\bx)$ for the same $\alpha$ increases as $|\bx|$ increases for 10-dimensional problems. For a given point $\bx\in\Omega$, the average number of iterations required increases in proportion to the value of fractional order $\alpha$. Moreover, the relationship between the number of paths and numerical error is consistent with the two-dimensional problem.
In particular, they indicate that our algorithm is very effective and stable for high dimensional problems. Indeed, this is one of the main advantages of the proposed method, which distinguish our approach from other existing methods for fractional PDEs.

%Although it is not easy to find green's function for a general domain, it is difficult to solve it directly.  But using stochastic algorithms, we can analyze each ball tangent to the region in the path simulation, so that we can approximate the value of the function in the complex domain. According to this idea, numerical solutions are performed for several general area problems.
 \begin{exa} {\bf (IFL on irregular domains)}
  In this example, we shall present the numerical results of our algorithm for IFL on irregular domains. 
To check the accuracy, we consider \eqref{ufg} with the following source function and boundary condition:
\begin{equation}
f(\bx)=2^{\alpha}\Gamma(1+\alpha/2)_1F_1\Big(\frac{2+\alpha}{2};1;-|\bx|^{2}\Big),\;\;\bx\in\Omega;\;\;\;g(\bx)=e^{-|\bx|^2},\;\;\bx\in\Omega^c.
\end{equation}
In this case, we take $\Omega = [-1,1]^2\setminus (0,1]^2$, which represents an L-shaped domain.
According to \cite{sheng2019fast}, we have the explicit expression of exact solution $u(\bx)= e^{-|\bx|^2}$ on $\Omega.$
%thus the boundary condition can be set as $g(\bx)=u(\bx)$ over $\Omega^c$. 
\end{exa}

 \begin{figure}[!h]
 \hspace{0.15in}
\subfigure%[The average number of steps]
{
\includegraphics[width=6.3cm,height=5cm]{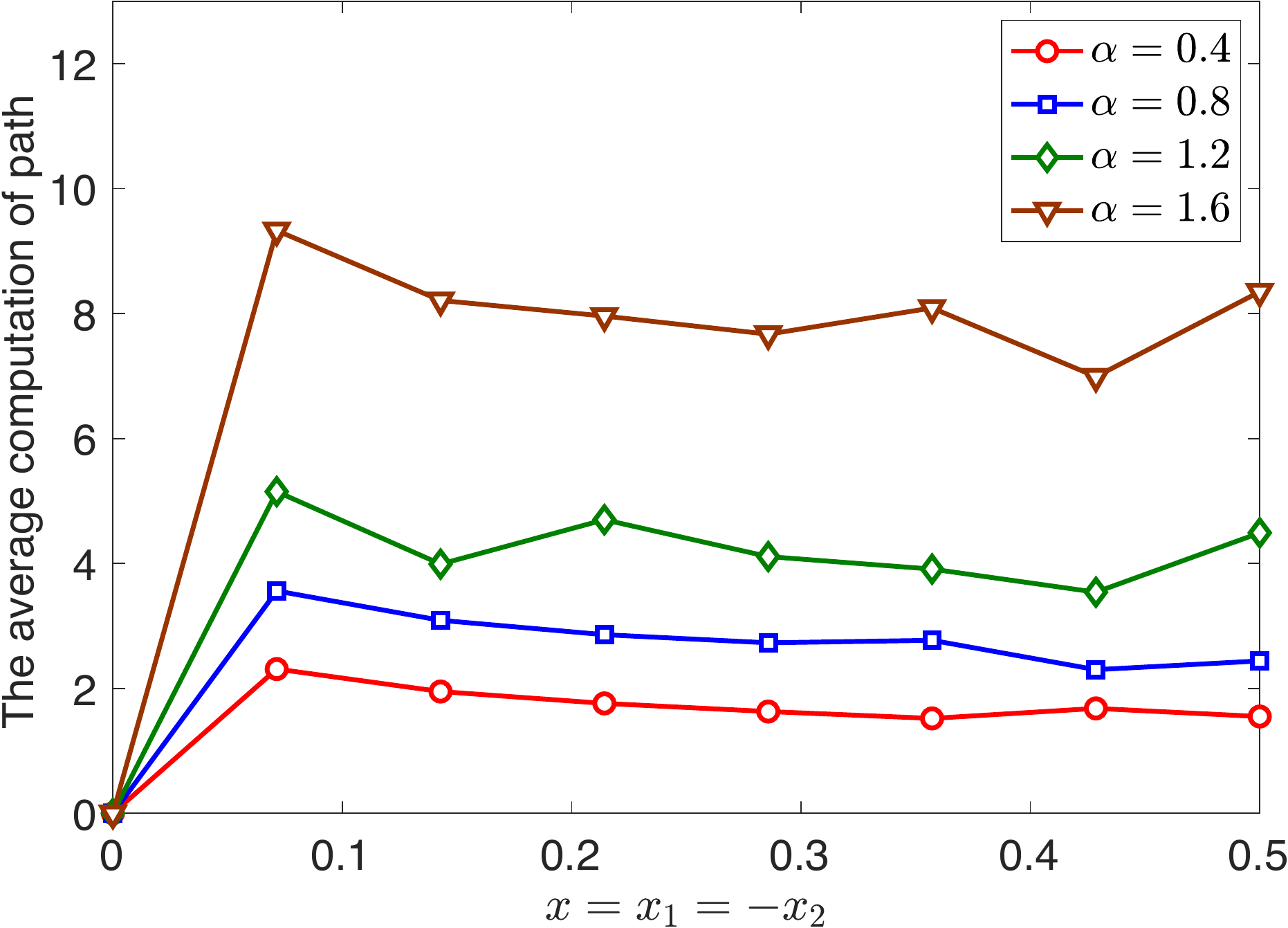}}
\hspace{0.3in}
\subfigure%[Error]
{
\includegraphics[width=6.3cm,height=5cm]{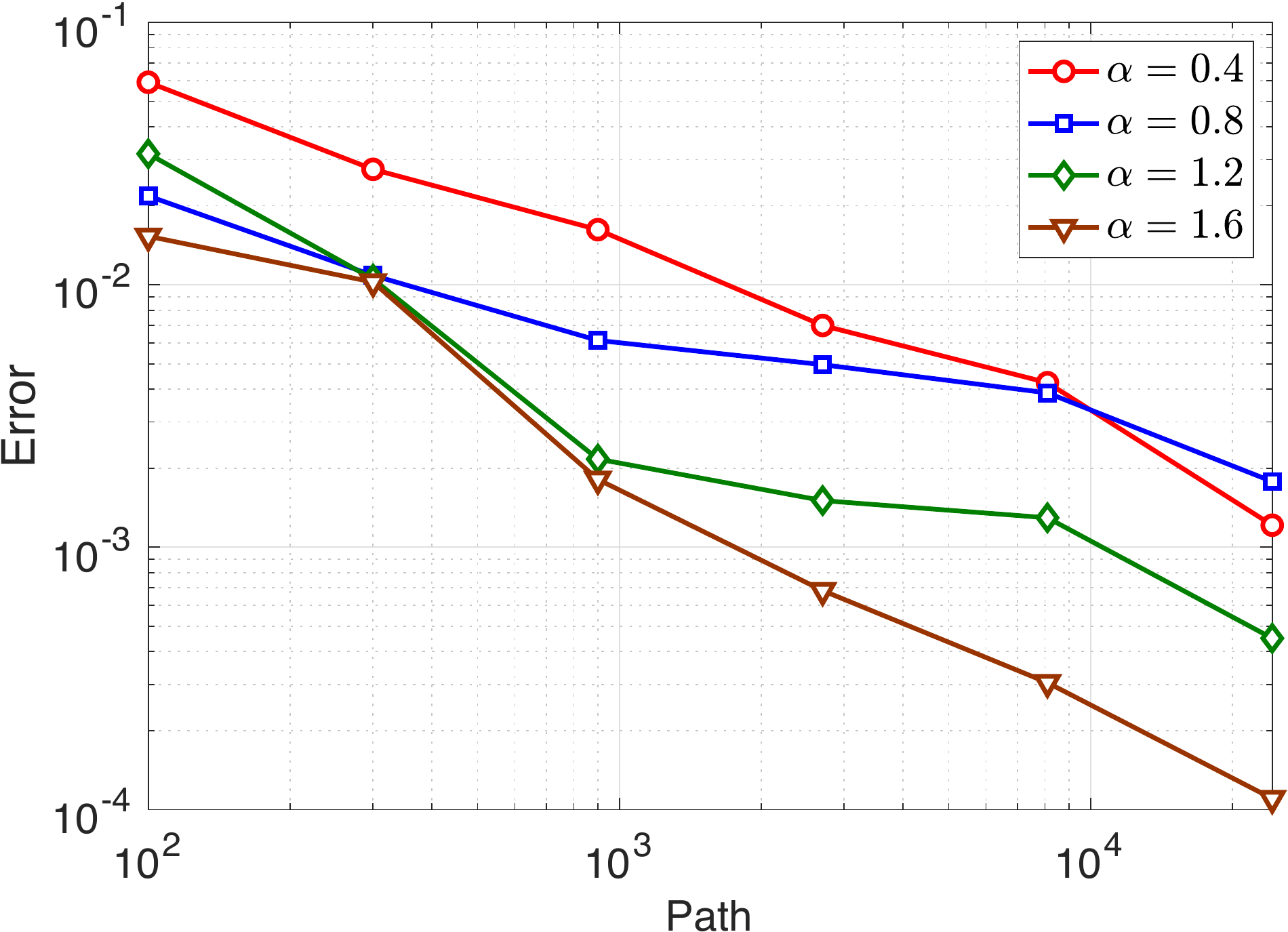}}
\caption{Simulation for two-dimension on complex domains. Left: the average number of steps for fixed point with various $\alpha$; Right: numerical errors against the number of path with various $\alpha$.}\label{examLs}
\end{figure}

In Fig.\,\ref{examLs} (left), we show the average number of iterations required for each path when calculating $u(\bx)$ for different fractional power $\alpha$. We observe from Fig.\,\ref{examLs} (left) that the average number of iterations for each path does not increase when $|\bx|>0.1$. For the $L$-shaped domain, with the increase of the coordinates of point $\bx$ along with the $(1,-1)$-direction, the distance between the point and boundary first increases and then gradually decreases. Thus, the average number of iterations required for each path first increases and then gradually decreases. The numerical errors against the path number is listed in Fig.\,\ref{examLs} (right). As excepted, the numerical errors decay algebraically as the increase of the total path. These are consistent with our previous theoretical analysis results.

To show the efficiency of our method on the complex domains, we further consider the following three cases:
%Next, we shall extend our algorithm for IFL on the following different domains:
%{\bf (i).}  regular hexagon domain; {\bf (ii).} annulus; {\bf (iii).} rectangular domain.
%To check the accuracy, we consider \eqref{ufg} with the following source functions.
 
 {\bf (i).}For the stripe domain, we take
\begin{equation}
f(\bx)=2^{\alpha}\Gamma(1+\frac \alpha 2) (\cos^{\frac \alpha 3}(c_{2}\bx)+ \sin^{\frac \alpha 2}(c_{1}\bx)) \cos(-|\bx|^2)\;\;{\rm on}\;\Omega,
\end{equation}
where $c_{1} = (\frac\pi 3, -\frac\pi 4)$, $c_{2} = (-\frac\pi 2, \frac{2\pi} 3)$, $\bx = (x_{1},x_{2})^{T}$, and $\Omega$ is the stripe domain on $[-5,5]\times [-0.5,0.5]$, the boundary condition can be set as $g(\bx)=0$ over $\Omega^c$.

{\bf (ii).}For regular hexagon domain, we take
\begin{equation}
f(\bx)=\sin^{2}(c_{1} \bx ) + \cos^{2}(c_{2}\bx) - (\alpha x_{1} x_{2})^{3},\;\;{\rm on}\;\Omega,
\end{equation}
where $\Omega$ is the regular hexagon domain on $[-1,1]^2$, the boundary condition can be set as $g(\bx)=0$ over $\Omega^c$.

{\bf (iii).}For the annulus domain, we take 
\begin{equation}
f(\bx)=\cos(x_{2}^{2}-2x_{1}x_{2}) -\sin( x_{1}^{2} +2x_{1}x_{2} ),\;\;{\rm on}\;\Omega,
\end{equation}
where $\Omega = \{\bx\in \R^2: 0.3<|\bx|^2<1 \} $ is the annulus domain, the boundary condition can be set as $g(\bx)=0$ over $\Omega^c$.

\begin{figure}[!ht]
\subfigure{
\includegraphics[width=14cm,height=1.3cm]{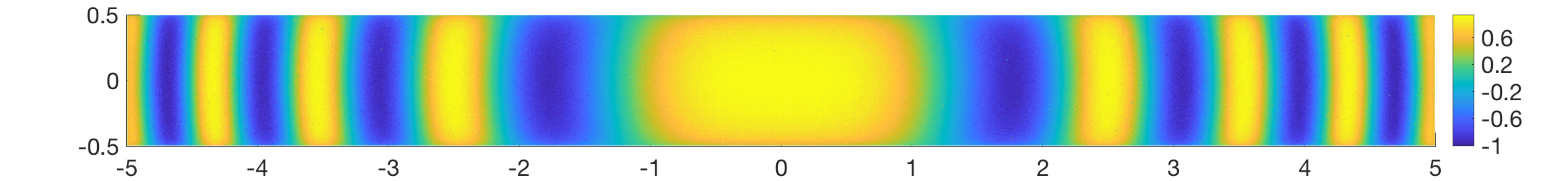}}\vspace{-.15in}

\subfigure{
\includegraphics[width=14cm,height=1.3cm]{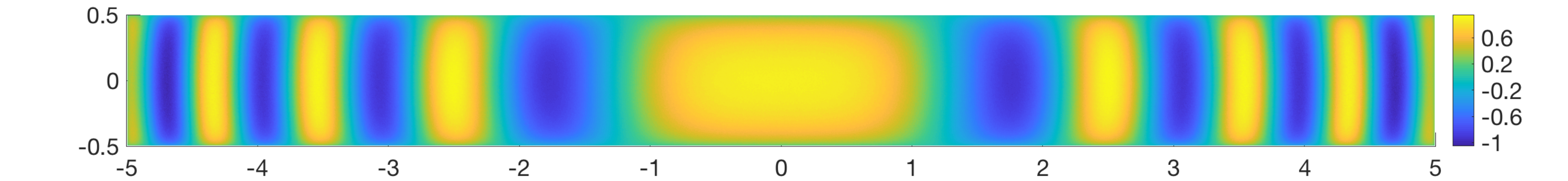}}\vspace{-.15in}

\subfigure{
\includegraphics[width=14cm,height=1.3cm]{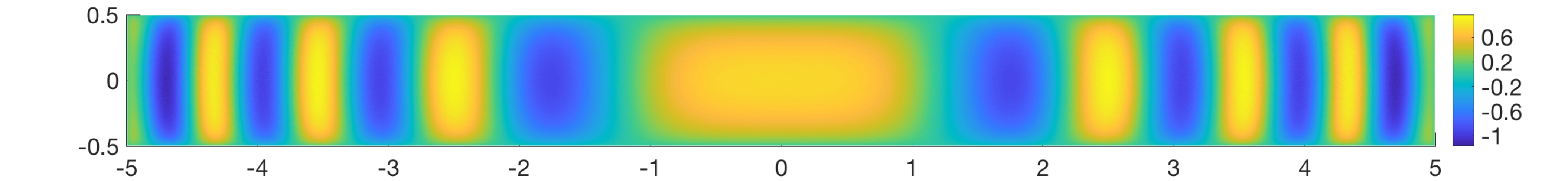}}\vspace{-.15in}

\subfigure{
\includegraphics[width=14cm,height=1.3cm]{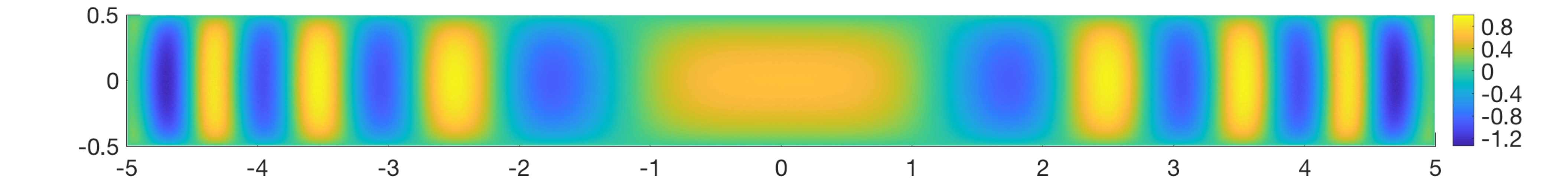}}\vspace{-.15in}
\caption{Profiles of the numerical solutions  zoomed in the stripe domain $[-5,5]\times [-0.5,0.5]$ with $\alpha=0.4,0.8,1.2,1.6$ (from top to bottom).}\label{example2drec}
\end{figure}

\begin{figure}[!ht]
\subfigure[$\alpha = 0.4$]{
\includegraphics[width=3.7cm,height=3.4cm]{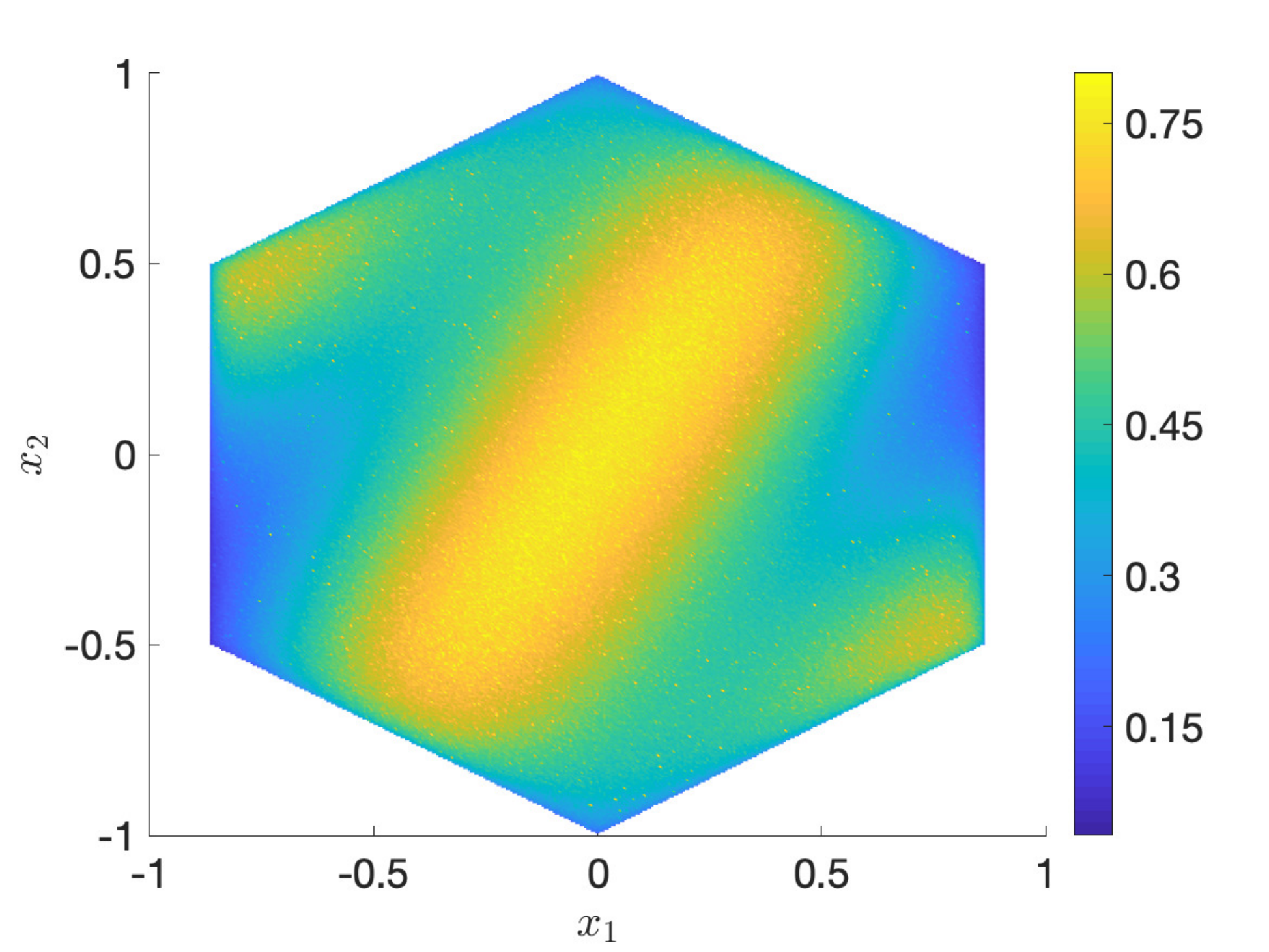}}\hspace{-0.1in}
\subfigure[$\alpha = 0.8$]{
\includegraphics[width=3.7cm,height=3.4cm]{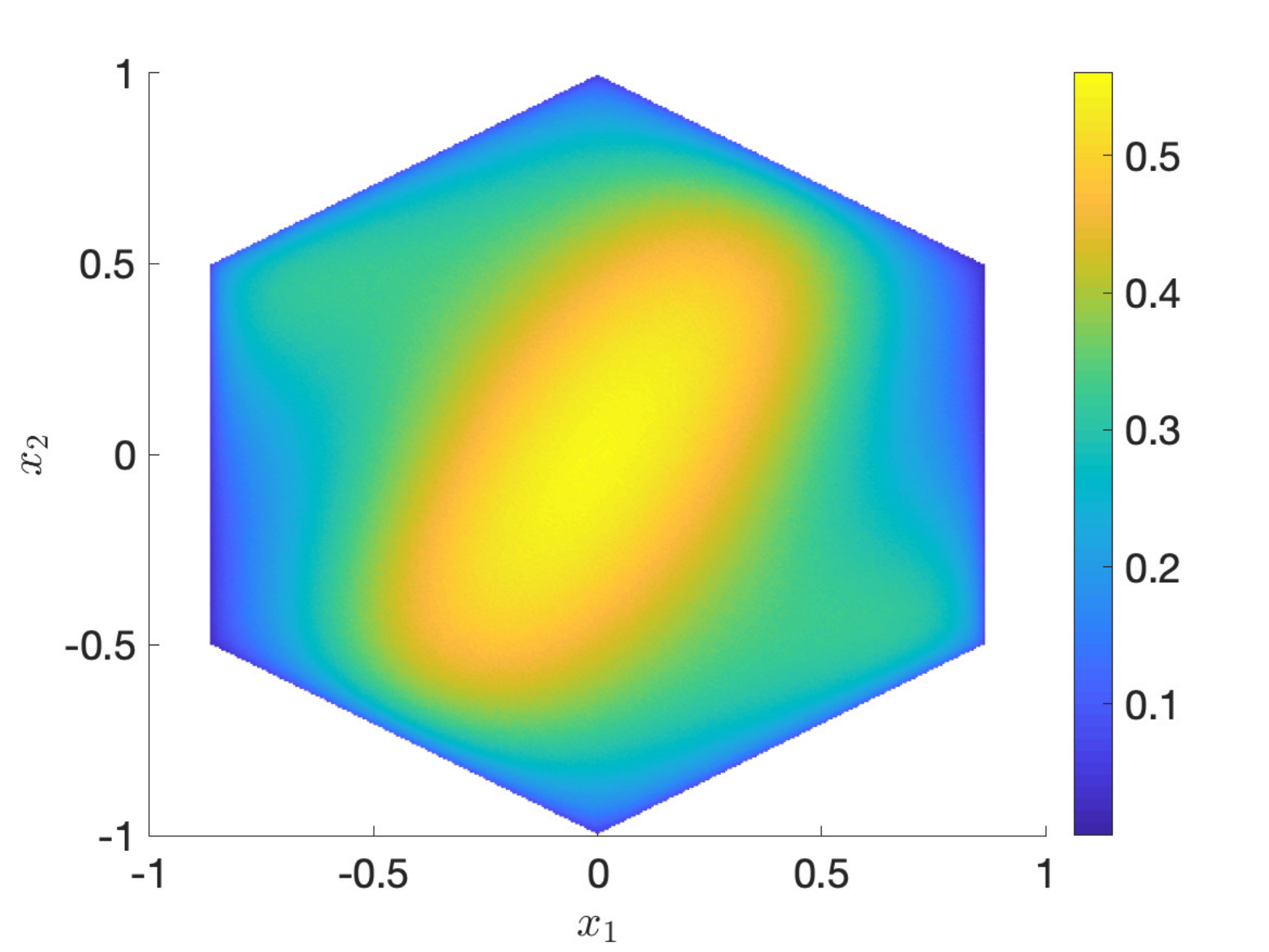}}\hspace{-0.1in}
\subfigure[$\alpha = 1.2$]{
\includegraphics[width=3.7cm,height=3.4cm]{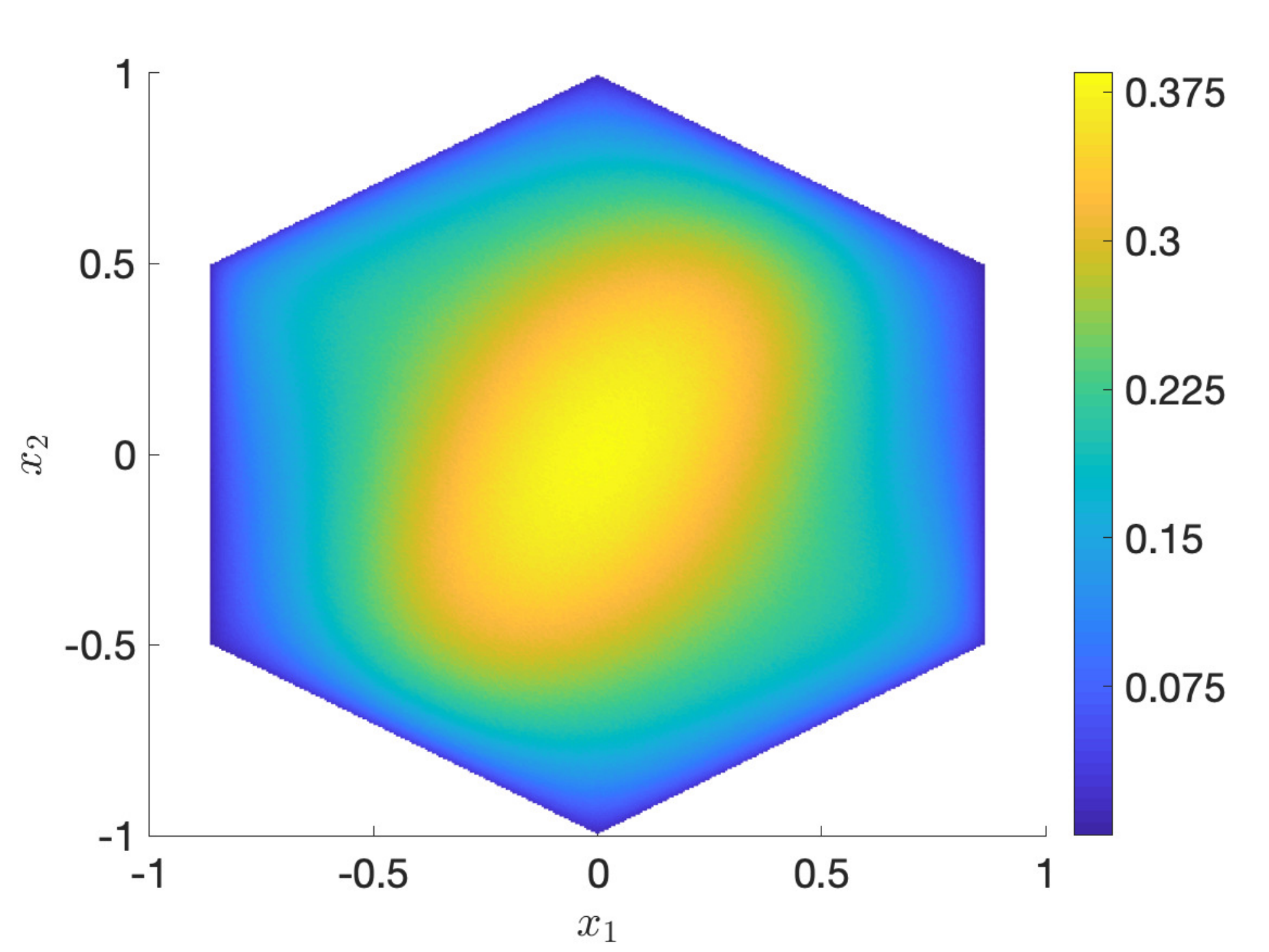}}\hspace{-0.1in}
\subfigure[$\alpha = 1.6$]{
\includegraphics[width=3.7cm,height=3.4cm]{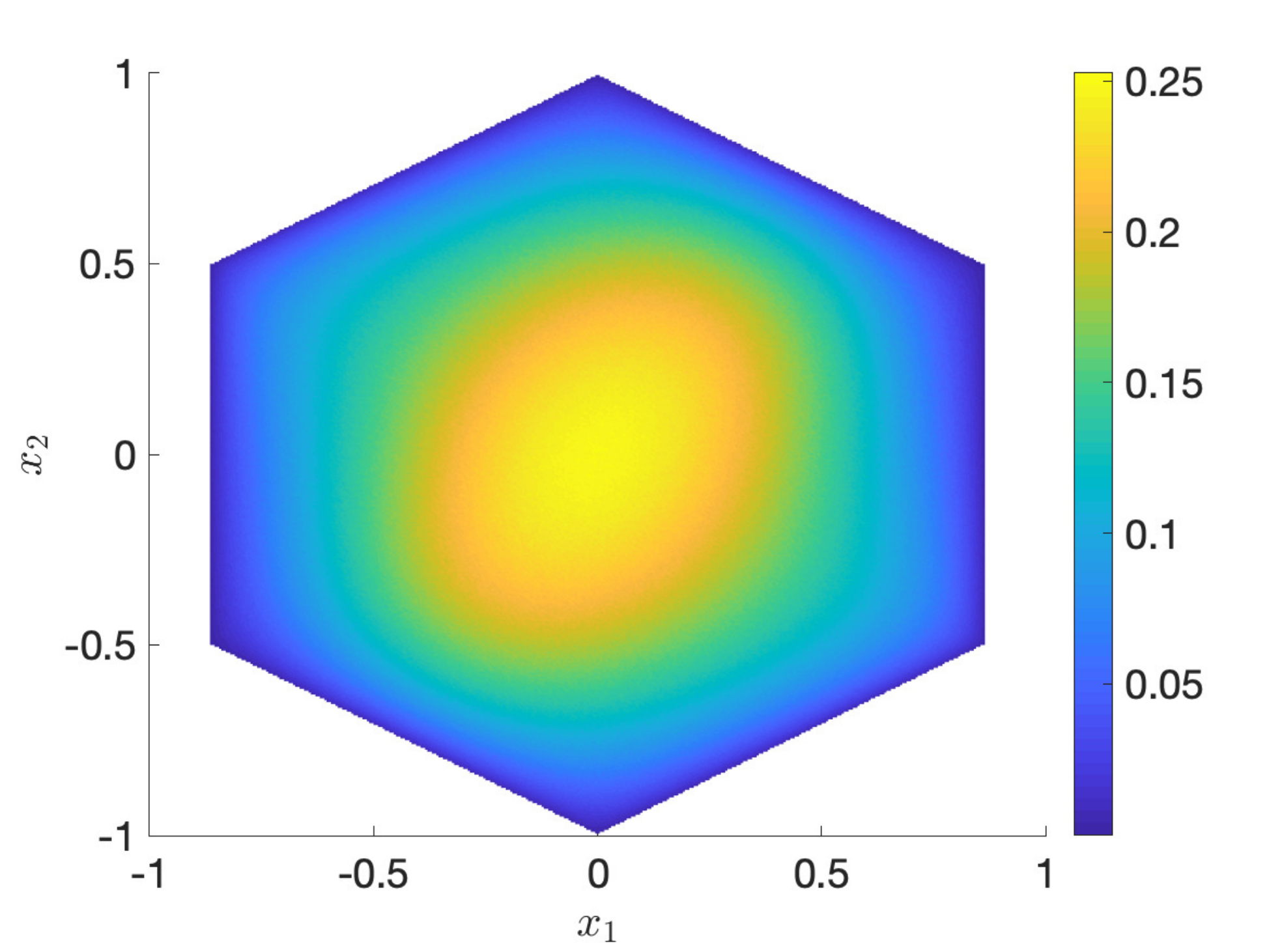}}

\subfigure[$\alpha = 0.4$]{
\includegraphics[width=3.7cm,height=3.4cm]{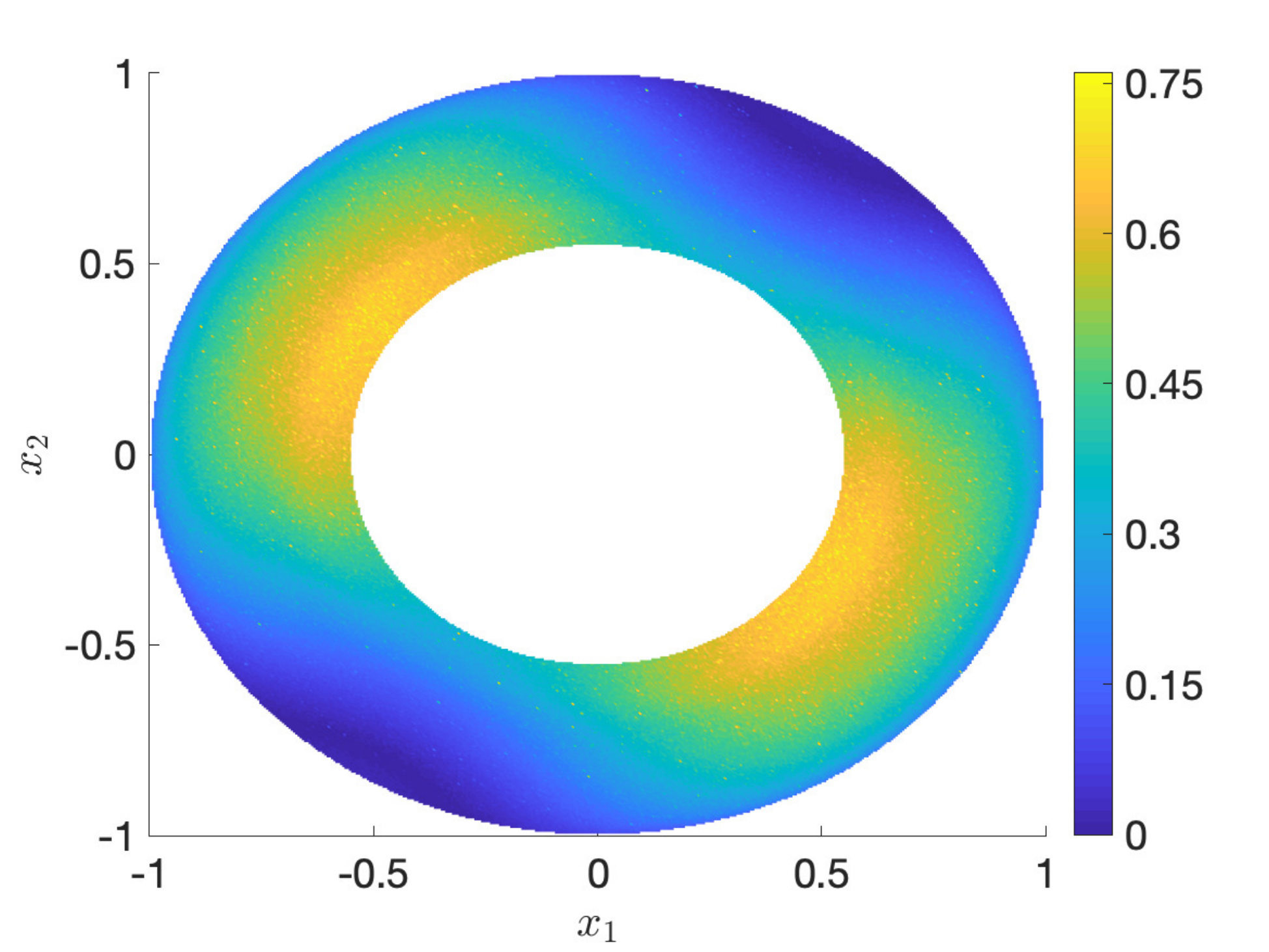}}\hspace{-0.1in}
\subfigure[$\alpha = 0.8$]{
\includegraphics[width=3.7cm,height=3.4cm]{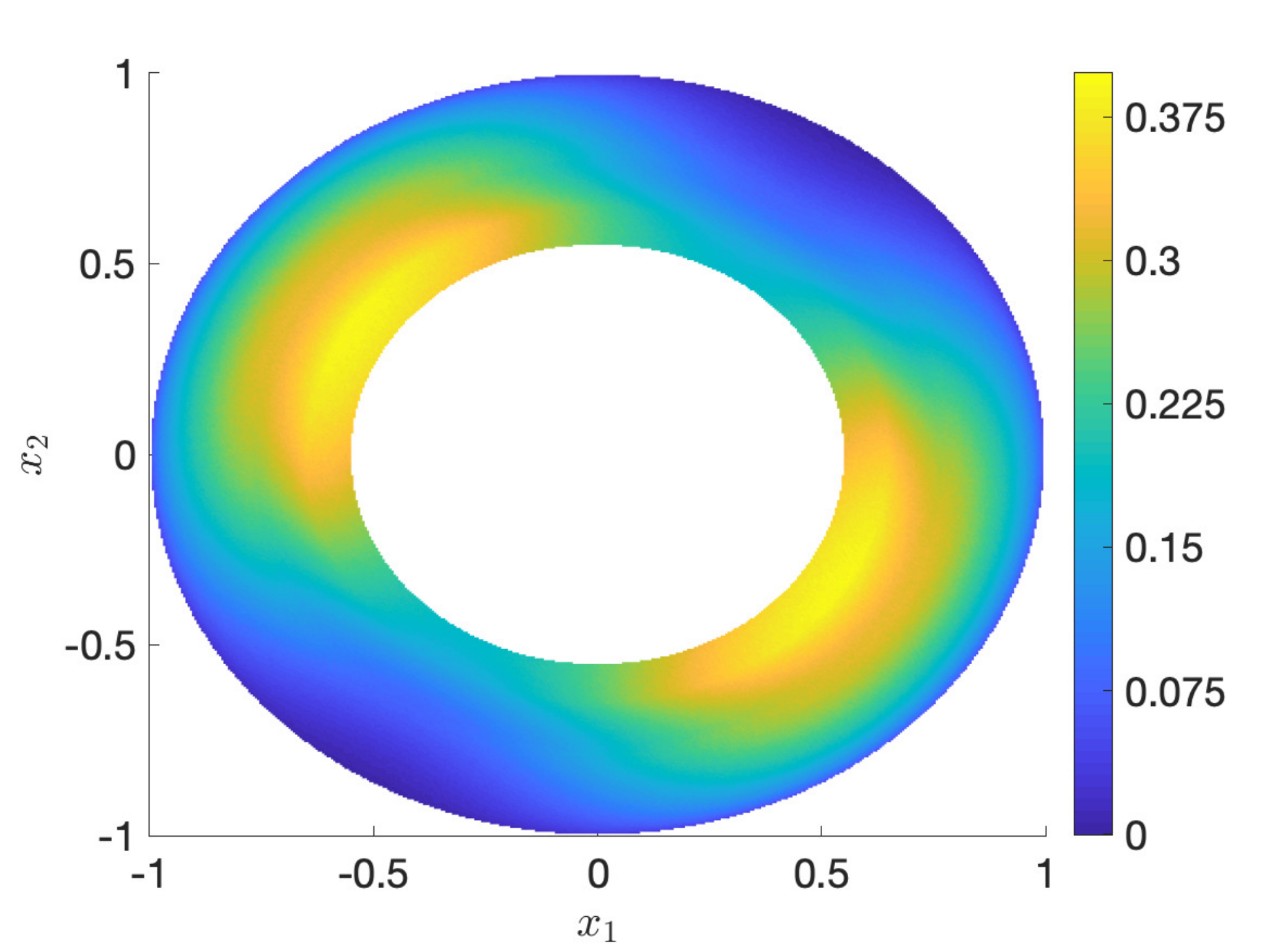}}\hspace{-0.1in}
\subfigure[$\alpha = 1.2$]{
\includegraphics[width=3.7cm,height=3.4cm]{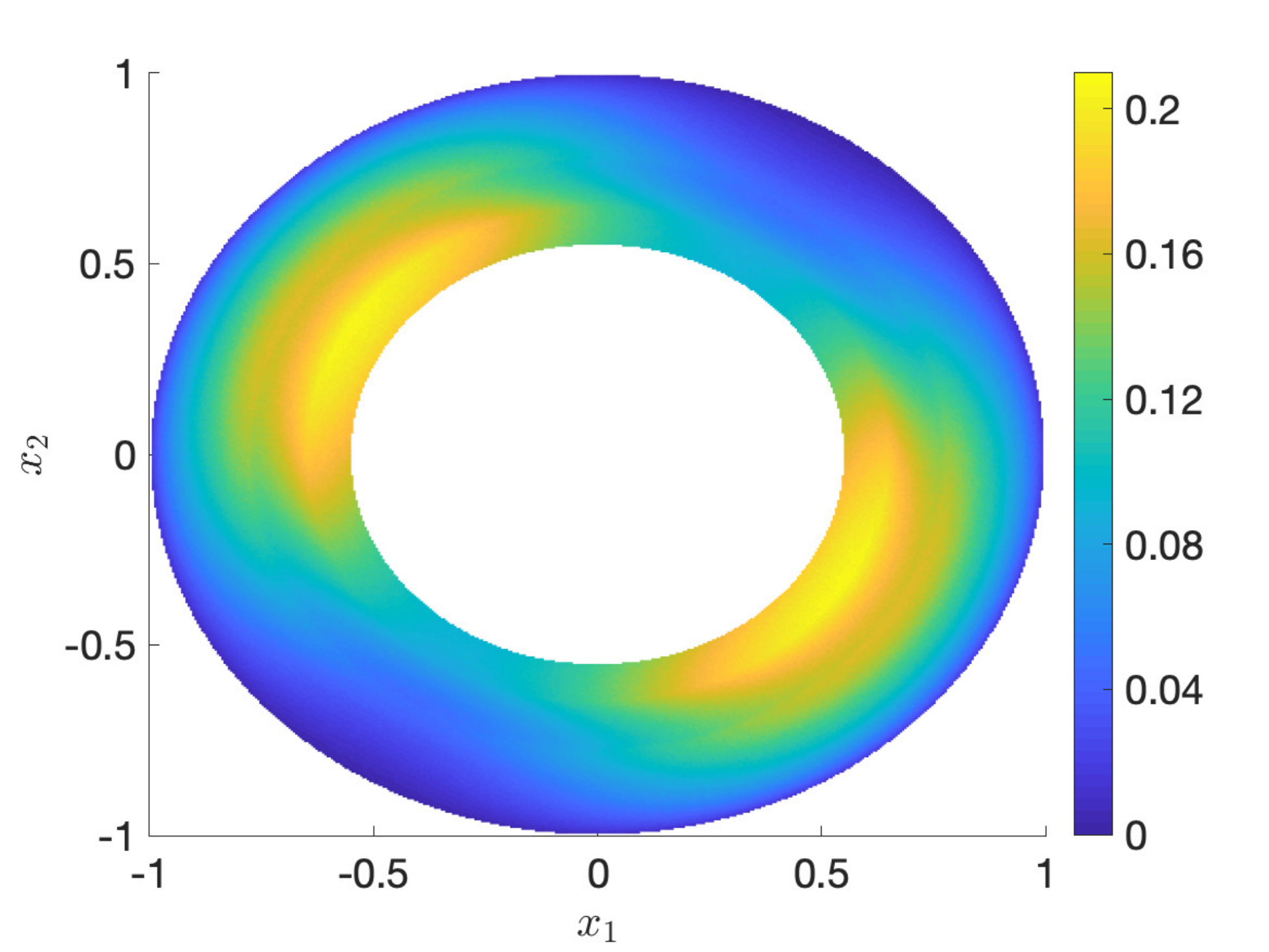}}\hspace{-0.1in}
\subfigure[$\alpha = 1.6$]{
\includegraphics[width=3.7cm,height=3.4cm]{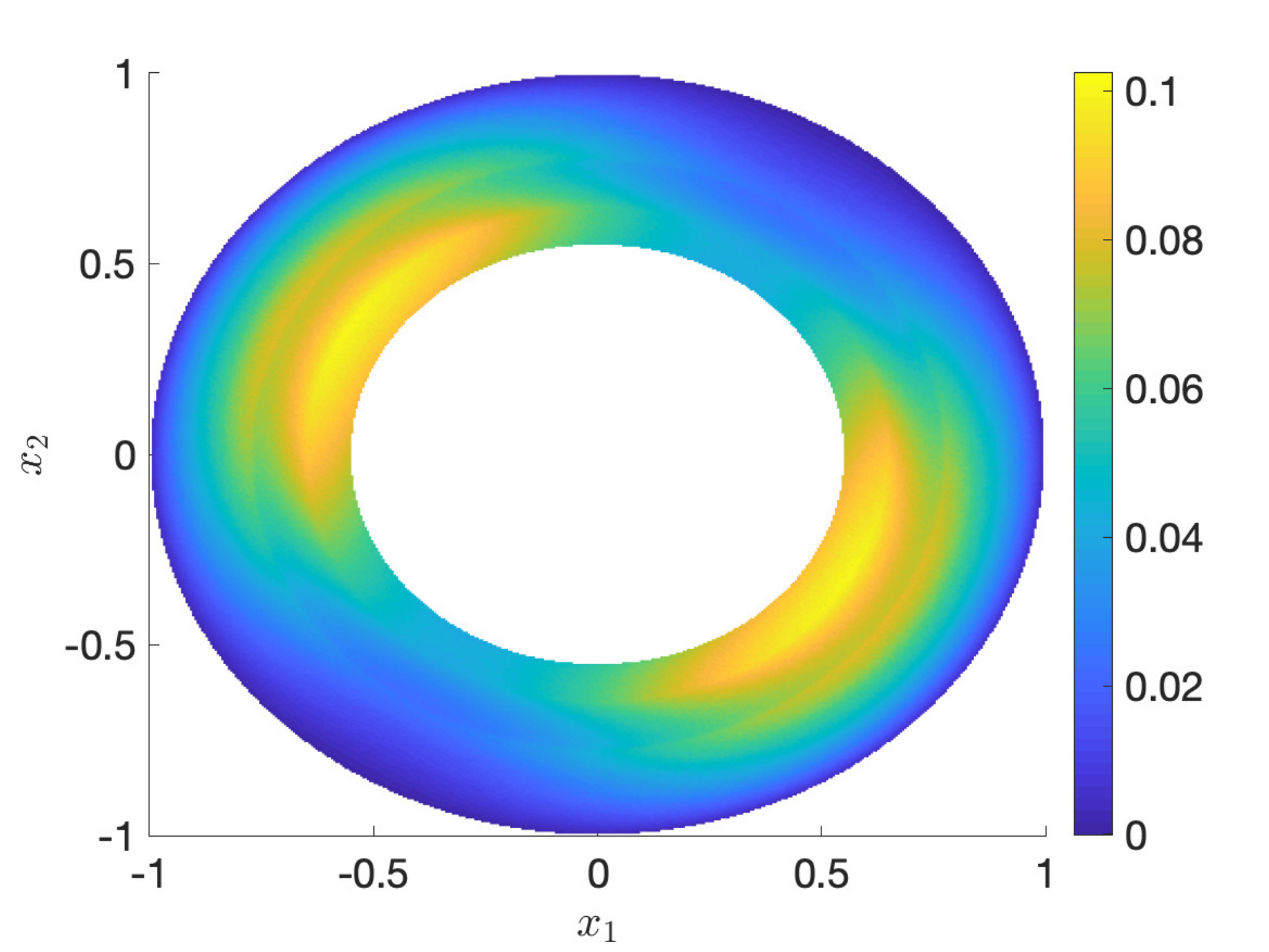}}
\caption{Profiles of the numerical solutions  with various $\alpha$. Top: zoomed in the hexagon domain; Botton: zoomed in the annulus domain.}\label{exam2dhexnum}
\end{figure}

 We use the proposed Monte Carlo method to calculate numerical solutions on three different shaped domains, and plot the corresponding profiles with varous fractional order $\alpha$ in Fig.\,\ref{example2drec} and \ref{exam2dhexnum}. %We observe that with the increase of $\alpha$ index, the decay rate of the solutions also increases.
  We observe similar behaviors as in the circular domain case, and the singularity layer near the boundary becomes thinner as expected as the fractional power $\alpha$ decreases.
\section{Conclusion}
This paper proposes an efficient Monte Carlo method for solving fractional PDEs on bounded domains in arbitrary dimensions. The key to the efficiency of the Monte Carlo algorithm is to construct a new Feynman-Kac representation in expectation form based on the explicit expressions of the Green functions. Then, we can avoid calculating the complex integrals associated with $\alpha$-stable process function. The representation \eqref{iresolu} allows us only to calculate the expectation of a random variable with known density functions, which enjoys the explicit expression of the Green functions. Each jump in the angular direction is distributed uniformly with the aid of the spherical coordinates, so we only need to appropriately add sample nodes in the angular direction with the increased dimensions. Consequently, the computation time will increase at a moderate pace with the increase of dimensions. Thus, this new process possesses fascinating merits for solving high-dimensional problems.
 More importantly,  we prove the efficiency of our algorithm and establish the error analysis of the proposed method for the fractional Poisson equation. Our numerical experiments demonstrate that our algorithms are efficient and accurate.
 In addition, the Monte Carlo method allows us to use parallel computing to improve the calculation efficiency and significantly reduce the computation time.

The idea of the proposed Monte Carlo method can be extended to the time-dependent problems with Feynman-Kac representation in a similar setting. In the future, we can combine the proposed method with the deep neural network for solving the nonlocal models driven by a stochastic process. % as it can approximate any given function with the aid of nonlinear structure, which provides a new idea for dealing with problems in high dimensions.

\section*{Acknowledgments}
%We would like to acknowledge the assistance of volunteers in putting together this example manuscript and supplement.
C. Sheng would like to thank Professor Jie Shen at Purdue University for many valuable suggestions and enlightening discussions. % during his visit to the Shanghai University of Finance and Economics.

\bibliographystyle{siamplain}

\end{document}